\documentclass[12pt,a4paper]{article}

\usepackage{amsfonts,amsmath,amsthm,amssymb,amsopn}
\usepackage{graphicx}
\usepackage{epstopdf}
\usepackage{algorithmic}
\usepackage{enumitem}

\usepackage{bm,color}
\usepackage{stmaryrd}
\usepackage[T1]{fontenc}
\usepackage{float, caption, subcaption}
\usepackage{booktabs}
\usepackage{diagbox}
\usepackage{multirow}
\usepackage{siunitx}
\usepackage{pifont}
\usepackage{hyperref}
\usepackage{cleveref}
\newcommand{\cmark}{\ding{51}}%
\newcommand{\xmark}{\ding{55}}%

\newcommand\bbR{\mathbb{R}}

\newcommand\bF{\bm{F}}

\newcommand\bv{\bm{v}}

\newcommand\bU{\bm{U}}

\newcommand\bR{\bm{R}}

\newcommand\bJ{\bm{J}}
\newcommand\bD{\bm{D}}

\newcommand\dd{\mathrm{d}}
\newcommand\pd[2]{\frac{\partial {#1}}{\partial {#2}}}

\newcommand\abs[1]{\lvert #1 \rvert}
\newcommand\norm[1]{\left\lVert #1 \right\rVert}
\DeclareMathOperator{\diag}{diag}

\newcommand\xr{i+\frac12}
\newcommand\xl{i-\frac12}
\newcommand\yr{j+\frac12}
\newcommand\yl{j-\frac12}

\theoremstyle{plain}
\newtheorem{lemma}{Lemma}[section]

\theoremstyle{definition}
\newtheorem{definition}{Definition}[section]

\newtheorem{theorem}{Theorem}[section]
\newtheorem{example}{Example}[section]
\newtheorem{proposition}{Proposition}[section]

\theoremstyle{remark}
\newtheorem{remark}{Remark}[section]

\crefname{equation}{}{}
\crefname{figure}{Figure}{Figures}
\crefname{table}{Table}{Tables}
\crefname{example}{Example}{Examples}
\crefname{section}{Section}{Sections}

\renewcommand{\title}{Active flux methods for hyperbolic conservation laws -- flux vector splitting and bound-preservation}

\newcommand{\authorOne}{Junming Duan\footnote{Corresponding author. Institute of Mathematics, University of W\"urzburg, Emil-Fischer-Stra\ss e 40, 97074 W\"urzburg, Germany, junming.duan@uni-wuerzburg.de}}
\newcommand{\authorTwo}{Wasilij Barsukow\footnote{Institut de Math\'ematiques de Bordeaux (IMB), CNRS UMR 5251, University of Bordeaux, 33405 Talence, France, wasilij.barsukow@math.u-bordeaux.fr}}
\newcommand{\authorThree}{Christian Klingenberg\footnote{Institute of Mathematics, University of W\"urzburg, Emil-Fischer-Stra\ss e 40, 97074 W\"urzburg, Germany, christian.klingenberg@uni-wuerzburg.de}}

\begin{document}

\begin{center} \Large
\title

\vspace{1cm}

\date{}
\normalsize

\authorOne, \authorTwo, \authorThree
\end{center}

\begin{abstract}

The active flux (AF) method is a compact high-order finite volume method that simultaneously evolves cell averages and point values at cell interfaces.
Within the method of lines framework, the existing Jacobian splitting-based point value update incorporates the upwind idea but suffers from a stagnation issue
	for nonlinear problems due to inaccurate estimation of the upwind direction, and also from a mesh alignment issue partially resulting from decoupled point value updates.
This paper proposes to use flux vector splitting for the point value update,
offering a natural and uniform remedy to those two issues.
To improve robustness, this paper also develops bound-preserving (BP) AF methods for hyperbolic conservation laws.
Two cases are considered: preservation of the maximum principle for the scalar case, and preservation of positive density and pressure for the compressible Euler equations.
The update of the cell average is rewritten as a convex combination of the original high-order fluxes and robust low-order (local Lax-Friedrichs or Rusanov) fluxes, and the desired bounds are enforced by choosing the right amount of low-order fluxes.
A similar blending strategy is used for the point value update.
	In addition, a shock sensor-based limiting is proposed to enhance the convex limiting for the cell average, which can suppress oscillations well.
	Several challenging tests are conducted to verify the robustness and effectiveness of the BP AF methods,
	including flow past a forward-facing step and high Mach number jets.



\end{abstract}

\section{Introduction}\label{sec:introduction}
This paper focuses on the development of robust active flux (AF) methods for hyperbolic conservation laws.
The AF method is a new finite volume method \cite{Eymann_2011_Active_InCollection, Eymann_2011_Active_InProceedings,Eymann_2013_Multidimensional_InCollection,Roe_2017_Is_JoSC},
that was inspired by \cite{VanLeer_1977_Towards_JoCP}.
Apart from cell averages, it incorporates additional degrees of freedom (DoFs) as point values located at the cell interfaces, evolved simultaneously with the cell average.
The original AF method employs a globally continuous representation of the numerical solution using a piecewise quadratic reconstruction,
leading naturally to a third-order accurate method with a compact stencil.
The introduction of point values at the cell interfaces avoids the usage of Riemann solvers as in usual Godunov methods, because the numerical solution is continuous across the cell interface and the numerical flux is available directly.

The independence of the point value update adds flexibility to the AF methods.
Based on the evolution of the point value, there are generally two kinds of AF methods.
The original one uses exact or approximate evolution operators and Simpson's rule for flux quadrature in time, i.e., it does not require time integration methods like Runge-Kutta methods.
Exact evolution operators have been studied for linear equations in \cite{Barsukow_2019_Active_JoSC,Fan_2015_Investigations_InCollection,Eymann_2013_Multidimensional_InCollection, VanLeer_1977_Towards_JoCP}.
Approximate evolution operators have been explored for Burgers' equation \cite{Eymann_2011_Active_InCollection,Eymann_2011_Active_InProceedings,Roe_2017_Is_JoSC,Barsukow_2021_active_JoSC},
the compressible Euler equations in one spatial dimension \cite{Eymann_2011_Active_InCollection,Helzel_2019_New_JoSC,Barsukow_2021_active_JoSC},
and hyperbolic balance laws \cite{Barsukow_2021_Active_SJoSC, Barsukow_2023_Well_CoAMaC}, etc.
One of the advantages of the AF method over standard finite volume methods is its structure-preserving property.
For instance, it preserves the vorticity and stationary states for multi-dimensional acoustic equations \cite{Barsukow_2019_Active_JoSC}, and it is naturally well-balanced for acoustics with gravity \cite{Barsukow_2021_Active_SJoSC}.

Since it may not be convenient to derive exact or approximate evolution operators for nonlinear systems, especially in multi-dimensions,
another kind of generalized AF method was presented in \cite{Abgrall_2023_Combination_CoAMaC, Abgrall_2023_Extensions_EMMaNA, Abgrall_2023_active}.
A method of lines was used, where the cell average and point value updates are written in semi-discrete form and advanced in time with time integration methods. 
In the point values update, the Jacobian matrix is split based on the sign of the eigenvalues (Jacobian splitting (JS)), and upwind-biased stencils are used to compute the approximation of derivatives.
There are some deficiencies of the JS-based AF methods, e.g., the stagnation issue \cite{Helzel_2019_New_JoSC} for nonlinear problems,
and mesh alignment issue in 2D to be introduced in \Cref{sec:mesh_alignment}.
Some remedies are suggested for the stagnation issue,
e.g., using discontinuous reconstruction \cite{Helzel_2019_New_JoSC} or
evaluating the upwind direction using more neighboring information \cite{Barsukow_2021_active_JoSC}. 

Solutions to hyperbolic conservation laws often stay in an \emph{admissible state set} $\mathcal{G}$,
also called the invariant domain.
	For instance, the solutions to initial value problems of scalar conservation laws satisfy a strict maximum principle (MP) \cite{Dafermos_2000_Hyperbolic_book}.
	Physically, both the density and pressure in the solutions to the compressible Euler equations should stay positive.
It is desired to conceive so-called bound-preserving (BP) methods,
i.e., those guaranteeing that the numerical solutions at a later time will stay in $\mathcal{G}$, if the initial numerical solutions belong to $\mathcal{G}$.
The BP property of numerical methods is very important for both theoretical analysis and numerical stability.
Many BP methods have been developed in the past few decades,
e.g., a series of works by Shu and collaborators \cite{Zhang_2011_Maximum_PotRSAMPaES, Hu_2013_Positivity_JoCP, Xu_2014_Parametrized_MoC},
a recent general framework on BP methods \cite{Wu_2023_Geometric_SR},
and the convex limiting approach \cite{Guermond_2018_Second_SJoSC, Hajduk_2021_Monolithic_C&MwA, Kuzmin_2020_Monolithic_CMiAMaE},
which can be traced back to the flux-corrected transport (FCT) schemes for scalar conservation laws \cite{Cotter_2016_Embedded_JoCP, Guermond_2017_Invariant_SJoNA, Lohmann_2017_Flux_JoCP, Kuzmin_2012_Flux_book}.
The previous studies on the AF methods pay limited attention to high-speed flows, or problems involving strong discontinuities,
with some efforts on the limiting for the point value update, see e.g. \cite{Barsukow_2021_active_JoSC, Helzel_2019_New_JoSC, Chudzik_2021_Cartesian_AMaC}.
	Although those limitings can reduce oscillations, the new
	cell average may violate the bound even for linear advection \cite{Barsukow_2021_active_JoSC, Helzel_2019_New_JoSC},
	and it is not straightforward to extend them to the multi-dimensional case.
	In \cite{Chudzik_2021_Cartesian_AMaC, Chudzik_2023_Review_InProceedings}, the authors proposed to adopt a discontinuous reconstruction based on the scaling limiter \cite{Zhang_2011_Maximum_PotRSAMPaES}.
    The flux is computed based on the limited point values, resulting in BP AF methods for scalar conservation laws.
In a very recent paper, the MOOD \cite{Clain_2011_high_JoCP} based stabilization was adopted to achieve the BP property \cite{Abgrall_2023_Activea} in an a posteriori fashion.

This paper presents a new way for the point value update to cure the stagnation and mesh alignment issues,
develops suitable BP limitings for the AF methods,
and also proposes a shock sensor-based limiting to further suppress oscillations.
The main contributions and findings in this work can be summarized as follows.
\begin{enumerate}[label=\textbf{\roman{enumi}).}, wide=0pt, nosep]
	\item We propose to employ the flux vector splitting (FVS) for the point value update,
	which can cure both the stagnation and the mesh alignment issues effectively,
	because the FVS couples the neighboring information in a uniform and natural way.
	The AF method based on the FVS is also shown to give better results than the JS, especially the local Lax-Friedrichs (LLF) FVS, in terms of the CFL number and shock-capturing ability.
	\item We develop BP limitings for both the cell average and point value by blending the high-order AF methods with the first-order LLF method in a convex combination.
	The main idea is to retain as much as possible of the high-order method while guaranteeing the numerical solutions to be BP,
	and the blending coefficients are computed by enforcing the bounds.
	We show that using a suitable time step size and BP limitings,
	the BP AF methods satisfy the MP for scalar conservation laws, and preserve positive density and pressure for the compressible Euler equations.
	\item We design a shock sensor-based limiting, which helps to reduce oscillations by detecting shock strength.
	It is shown to strongly improve the shock-capturing ability in the numerical tests.
	\item Several challenging numerical tests are used to demonstrate the robustness and effectiveness of our BP AF methods.
	Moreover, for the forward-facing step problem, our BP AF method captures small-scale features better compared to the third-order DG method with the TVB limiter on the same mesh resolution, while using fewer DoFs, demonstrating its efficiency and potential for high Mach number flows.
\end{enumerate}

The remainder of this paper is organized as follows.
\Cref{sec:1d_af_schemes} introduces the 1D AF methods based on the FVS for the point value update.
\Cref{sec:2d_af_schemes} extends our FVS-based AF methods to the 2D case.
To design BP methods, \Cref{sec:2d_limitings} describes our convex limiting approach for the cell average, and the limiting for the point value.
The shock sensor-based limiting is also proposed in \Cref{sec:2d_limitings} to suppress oscillations.
The 1D limitings can be reduced from the 2D case, and more details are given in Section \ref{sec:1d_limitings} in Appendix.
Some numerical tests are conducted in  \Cref{sec:results} to experimentally demonstrate the accuracy, BP properties, and shock-capturing ability of the methods.
\Cref{sec:conclusion} concludes the paper with final remarks.

\section{1D active flux methods}\label{sec:1d_af_schemes}
This section presents the generalized AF methods using the method of lines for the 1D hyperbolic conservation laws
\begin{equation}\label{eq:1d_hcl}
	\bU_t + \bF(\bU)_x = 0,\quad
	\bU(x,0) = \bU_0(x),
\end{equation}
where $\bU\in \bbR^{m}$ is the vector of $m$ conservative variables, $\bF\in \bbR^{m}$ is the flux function,
and $\bU_0(x)$ is assumed to be initial data of bounded variation.
Two cases are of particular interest.
The first is a scalar conservation law ($m=1$)
\begin{equation}\label{eq:1d_scalar}
	u_t + f(u)_x = 0,\quad u(x,0) = u_0(x).
\end{equation}
The second case is that of compressible Euler equations of gas dynamics
with $\bU=(\rho, \rho v, E)^\top$ and $\bF=(\rho v, \rho v^2 + p, (E+p)v)^\top$,
where $\rho$ denotes the density, $v$ the velocity,
$p$ the pressure, and $E=\frac12\rho v^2 + \rho e$ the total energy with $e$ the specific internal energy.
The perfect gas equation of state (EOS) $p = (\gamma-1)\rho e$ is used to close the system with the adiabatic index $\gamma > 1$.
Note that this paper uses bold symbols to refer to
vectors and matrices, such that they are easier to distinguish from scalars.
%

Assume that a 1D computational domain is divided into $N$ cells
$I_i = [x_{\xl}, x_{\xr}]$ with cell centers $x_i = (x_{\xl} + x_{\xr})/2$ and cell sizes $\Delta x_i = x_{\xr}-x_{\xl}$, $i=1,\cdots,N$.
The DoFs of the AF methods are the approximations to cell averages of the conservative variable as well as point values at the cell interfaces,
allowing some freedom in the choice of the point values,
e.g. conservative variables, primitive variables, entropy variables, etc.
This paper only considers using the conservative variables, and the DoFs are denoted by
\begin{equation}
	\overline{\bU}_i(t) = \dfrac{1}{\Delta x_i}\int_{I_i} \bU_h(x,t) ~\dd x,\quad
	\bU_{\xr}(t) = \bU_h(x_{\xr}, t),
\end{equation}
where $\bU_h(x,t)$ is the numerical solution.
The cell average is updated by integrating \cref{eq:1d_hcl} over $I_i$ in the following semi-discrete finite volume manner
\begin{equation}\label{eq:semi_av_1d}
	\dfrac{\dd \overline{\bU}_i}{\dd t} = -\dfrac{1}{\Delta x_i}\left[\bF(\bU_{\xr}) - \bF(\bU_{\xl})\right].
\end{equation}
Thus, the accuracy of \cref{eq:semi_av_1d} is determined by the approximation accuracy of the point values.
It was so far (e.g. in \cite{Abgrall_2023_Extensions_EMMaNA}) considered sufficient to update the point values with any finite-difference-like formula
\begin{equation}\label{eq:semi_pnt_1d}
	\dfrac{\dd \bU_{\xr}}{\dd t} = - \bm{\mathcal{R}}\left(\bU_{i+\frac12-l_1}(t), \overline{\bU}_{i+1-l_1}(t), \cdots, \overline{\bU}_{i+l_2}(t), \bU_{i+\frac12+l_2}(t)\right), ~l_1,l_2\geqslant 0,
\end{equation}
with $\bm{\mathcal{R}}$ a consistent approximation of $\partial\bF/\partial x$ at $x_{\xr}$,
as long as it gave rise to a stable method.
This paper explores further conditions on $\bm{\mathcal{R}}$ for nonlinear problems.

\subsection{Stagnation issue when using Jacobian splitting}\label{sec:1d_js}
Let us first briefly describe the point value update based on the JS \cite{Abgrall_2023_Extensions_EMMaNA}, which reads
\begin{equation}\label{eq:1d_semi_js}
	\dfrac{\dd \bU_{\xr}}{\dd t} = - \left[\bJ^+(\bU_{\xr})\bD^+_{\xr}(\bU) + \bJ^-(\bU_{\xr})\bD^-_{\xr}(\bU) \right],
\end{equation}
where the splitting of the Jacobian matrix $\bJ = \bJ^{+} + \bJ^{-}$ is defined as
\begin{equation*}
	\bJ^\pm = \bR\bm{\Lambda}^\pm\bR^{-1},~
	\bm{\Lambda}^\pm = \diag\{\lambda_1^\pm, \dots, \lambda_m^\pm\},
\end{equation*}
based on the eigendecomposition
${\partial \bF}/{\partial\bU} = \bR\bm{\Lambda}\bR^{-1}$,
$\bm{\Lambda} = \diag\{\lambda_1, \dots, \lambda_m\}$,
where $\lambda_1,\cdots,\lambda_m$ are the eigenvalues, with the columns of $\bR$ the corresponding eigenvectors,
and $a^+ = \max\{a, 0\}, ~a^- = \min\{a, 0\}$.
To derive the approximation of the derivatives in \cref{eq:1d_semi_js}, one can first obtain a high-order reconstruction for $\bU$ in the upwind cell, and then differentiate the reconstructed polynomial.
As an example, a parabolic reconstruction in cell $I_i$ is
\begin{align}
	\bU_{\texttt{para}, 1}(x) =& -3(2\overline{\bU}_i - \bU_{\xl} - \bU_{\xr}) \frac{x^2}{\Delta x_i^2}
	+ (\bU_{\xr} - \bU_{\xl}) \frac{x}{\Delta x_i} \nonumber\\
	&+ \frac14(6\overline{\bU}_i - \bU_{\xl} - \bU_{\xr})\label{eq:parabolic_reconstruction}
\end{align}
satisfying
$\bU_{\texttt{para}, 1}(\pm\Delta x_i/2) = \bU_{i\pm\frac12}$,
$\frac1{\Delta x_i}\int_{-\Delta x_i/2}^{\Delta x_i/2} \bU_{\texttt{para}, 1}(x) ~\dd x = \overline{\bU}_{i}$.
Then the derivatives are
\begin{subequations}\label{eq:parabolic_av}
	\begin{align}
		\bD^{+}_{\xr}(\bU) =\bU_{\texttt{para}, 1}'(\Delta x_i/2) &= \dfrac{1}{\Delta x_i}\left(2 \bU_{\xl}- 6 \overline{\bU}_i  + 4 \bU_{\xr} \right), \\
		\bD^{-}_{\xr}(\bU) &= \dfrac{1}{\Delta x_{i+1}}\left(- 4 \bU_{\xr} + 6 \overline{\bU}_{i+1} - 2 \bU_{i+\frac32} \right).
	\end{align}
\end{subequations}

One of the deficiencies of using the JS is the stagnation issue that appears in certain setups for nonlinear problems, as observed in \cite{Helzel_2019_New_JoSC, Barsukow_2021_active_JoSC}.
	As shown in \cref{ex:1d_burgers} for Burgers' equation,
the numerical solution based on the JS without limiting gives a spike in the cell average at the initial discontinuity $x=0.2$, which grows linearly in time.
The reason for this behavior is the inaccurate estimation of the upwind direction at the cell interface, required to split the Jacobian in \cref{eq:1d_semi_js}.
In this example, there are two successive point values with different signs near the initial discontinuity: $u_{\xl}=2$, $u_{\xr}=-1$.
At the cell interface $x_{\xl}$ or $x_{\xr}$, depending on the details of initialization, the upwind discretization in \cref{eq:parabolic_av} only uses the data from the left or right,
leading to zero derivatives, thus the point values $u_{\xl}$ and $u_{\xr}$ stay unchanged.
However, according to the update of the cell average \cref{eq:semi_av_1d},
$\bar{u}_i$ increases gradually (which is the observed spike).
Proposed solutions to handle the stagnation issue involve estimating the Jacobian not only at the relevant cell interface, but also at the neighboring interfaces,
and to select a better upwind direction (e.g. \cite{Barsukow_2021_active_JoSC}),
or achieve the same by blending (e.g. \cite{Chudzik_2023_Review_InProceedings}).
As will be shown below, using FVS instead of the JS naturally has a similar eﬀect.

\subsection{Point value update using flux vector splitting}\label{sec:1d_fvs}
In this paper, we propose to use the FVS for the point value update, which was originally used to identify the upwind directions, and is simpler and somewhat more efficient than Godunov-type methods for solving hyperbolic systems \cite{Toro_2009_Riemann}.
The FVS for the point value update reads
\begin{equation}\label{eq:1d_semi_fvs}
	\dfrac{\dd \bU_{\xr}}{\dd t} 
	= - \left[\widetilde{\bD}^+\bF^{+}(\bU)+ \widetilde{\bD}^-\bF^{-}(\bU) \right]_{\xr},
\end{equation}
where the flux $\bF$ is split into the positive and negative parts $\bF= \bF^{+} + \bF^{-}$
satisfying
\begin{equation}\label{eq:FVS_condition}
	\lambda\left(\pd{\bF^{+}}{\bU}\right) \geqslant 0, \quad \lambda\left(\pd{\bF^{-}}{\bU}\right) \leqslant 0,
\end{equation}
i.e., all the eigenvalues of $\pd{\bF^{+}}{\bU}$ and $\pd{\bF^{-}}{\bU}$ are non-negative and non-positive, respectively.
Different FVS can be adopted as long as they satisfy the constraint \cref{eq:FVS_condition}, to be discussed later.
Finite difference formulae to approximate the flux derivatives are obtained as follows.
From the reconstruction of $\bU$ \cref{eq:parabolic_reconstruction},
one can evaluate the flux $\bF$, and also the split fluxes $\bF^\pm$ pointwise. 
We compute them at the endpoints of the cell and in the middle.
Then a parabolic reconstruction for, say, $\bF^{+}$ in the cell $I_i$ is obtained as follows%
\begin{equation*}
	\bF^{+}_{\texttt{para}, 2}(x) = 
 2(\bF^{+}_{\xl} - 2\bF^{+}_i + \bF^{+}_{\xr}) \frac{x^2}{\Delta x_i^2}
	+ (\bF^{+}_{\xr} - \bF^{+}_{\xl}) \frac{x}{\Delta x_i} + \bF^{+}_{i},
\end{equation*}
satisfying
$\bF^{+}_{\texttt{para}, 2}(\pm\Delta x_i/2) = \bF^{+}_{i\pm\frac12}= \bF^{+}(\bU_{i\pm\frac12})$,
and
$\bF^{+}_{\texttt{para}, 2}(0) = \bF^{+}_{i}= \bF^{+}(\bU_i)$.
The cell-centered point value is
$\bU_i = (- \bU_{\xl} + 6\overline{\bU}_i - \bU_{\xr})/4$.
Then the discrete derivatives are
\begin{subequations}\label{eq:parabolic_pnt}
	\begin{align}
		\left(\widetilde{\bD}^{+}\bF^+\right)_{\xr} = \left(\bF^{+}_{\texttt{para}, 2}\right)'(\Delta x_i/2) &= \dfrac{1}{\Delta x_i}\left(\bF^{+}_{\xl} - 4 \bF^{+}_i  + 3 \bF^{+}_{\xr} \right), \\
		\left(\widetilde{\bD}^{-}\bF^-\right)_{\xr} &= \dfrac{1}{\Delta x_{i+1}}\left(- 3 \bF^{-}_{\xr} + 4 \bF^{-}_{i+1} - \bF^{-}_{i+\frac32} \right).
	\end{align}
\end{subequations}
These finite differences are third-order accurate.
While the reconstructions of both $\bU$ and $\bF$ are parabolic, the coefficients in the formula \cref{eq:parabolic_pnt} differ from that in \cite{Abgrall_2023_Extensions_EMMaNA} because \cref{eq:parabolic_pnt} uses the cell-centered value rather than the cell average.

The FVS-based point value update borrows the information from the neighbors naturally, and can eliminate the generation of the spike effectively, as shown in \cref{fig:1d_burgers_shock},
similar to the idea of the remedy in \cite{Barsukow_2021_active_JoSC}.
Note that we still use the original continuous reconstruction in the AF methods.
We remark that, in AF methods, it is not clear how to define the point values at discontinuities, thus there may be other methods to fix the stagnation issue.

\subsubsection{Local Lax-Friedrichs flux vector splitting}
The first FVS we consider is the LLF FVS, defined as
\begin{equation*}
	\bF^\pm = \frac12(\bF(\bU) \pm \alpha \bU),
\end{equation*}
where the choice of $\alpha$ should fulfill \cref{eq:FVS_condition} across the spatial stencil.
In our implementation, it is determined by
\begin{equation}\label{eq:1d_Rusanov_alpha}
	\alpha_{\xr} = \max_{r} \left\{\varrho(\bU_r)\right\},
	~r \in \left\{ i-\frac12, i, i+\frac12, i+1, i+\frac32\right\},
\end{equation}
where $\varrho$ is the spectral radius of $\partial\bF/\partial\bU$.
One can also choose $\alpha$ to be the maximal spectral radius in the whole domain, corresponding to the (global) LF splitting.
Note, however, that a larger $\alpha$ generally leads to a smaller time step size and more dissipation.

\subsubsection{Upwind flux vector splitting}
One can also split the Jacobian matrix based on each characteristic field,
\begin{equation}\label{eq:upwind_fvs}
	\bF^\pm = \frac12(\bF(\bU)\pm \abs{\bJ} \bU),\quad
	\abs{\bJ} = \bR(\bm{\Lambda}^{+} - \bm{\Lambda}^{-})\bR^{-1}.
\end{equation}
 Note that we evaluate the Jacobian at three locations in the cell $I_i$ to get corresponding $\bF^+$.
For linear systems, one has $\bF = \bJ \bU$, so \cref{eq:upwind_fvs} reduces to the JS, because in this case
\begin{equation*}
	\bF^\pm = \frac12(\bJ\pm \abs{\bJ})\bU
	= \bR\bm{\Lambda}^{\pm}\bR^{-1} \bU= \bJ^{\pm}\bU,
\end{equation*}
with $\bm{J}^{\pm}$ a constant matrix so that
$\widetilde{\bD}^{\pm}\bF^\pm(\bU) = \bJ^{\pm}\widetilde{\bD}^{\pm}\bU$,
which is the same as $\bJ^{\pm}\bD^{\pm}\bU$ if $\bD^{+}$ and $\widetilde\bD^{+}$ are derived from the same reconstructed polynomial.
In other words, for linear systems, the AF methods using this FVS are the same as the original JS-based AF methods.

Such an FVS is also known as the Steger-Warming (SW) FVS \cite{Steger_1981_Flux_JoCP} for the Euler equations,
since the ``homogeneity property'' $\bF = \bJ \bU$ holds \cite{Toro_2009_Riemann}.
One can write down the SW FVS explicitly
\begin{align*}
	\bF^{\pm} &= \begin{bmatrix}
		\frac{\rho}{2\gamma}\alpha^\pm \\
		\frac{\rho}{2\gamma}\left(\alpha^\pm v + a (\lambda_2^\pm - \lambda_3^\pm)\right) \\
		\frac{\rho}{2\gamma}\left(\frac12\alpha^\pm v^2 + a v (\lambda_2^\pm - \lambda_3^\pm) + \frac{a^2}{\gamma-1}(\lambda_2^\pm + \lambda_3^\pm)\right) \\
	\end{bmatrix},
\end{align*}
where
$\lambda_1 = v, ~\lambda_2 = v + a, ~\lambda_3 = v - a,
~
	\alpha^\pm = 2(\gamma-1)\lambda_1^\pm + \lambda_2^\pm + \lambda_3^\pm$,
and $a = \sqrt{\gamma p / \rho}$ is the sound speed.

\begin{remark}\label{rmk:sw_fvs}
It should be noted that $\bF^\pm$ in this FVS may not be differentiable with respect to $\bU$ for nonlinear systems (e.g. Euler), as the splitting is based on the absolute value.
In \cite{Leer_1982_Flux_InProceedings}, the mass flux of $\bF^\pm$ is shown to be not differentiable, which might explain the accuracy degradation in \cref{ex:2d_vortex}.
\end{remark}

\subsubsection{Van Leer-H\"anel flux vector splitting for the Euler equations}
Another popular FVS for the Euler equations was proposed by van Leer \cite{Leer_1982_Flux_InProceedings},
and improved by \cite{Haenel_1987_accuracy_InCollection}.
The flux can be split based on the Mach number $M=v/a$ as
\begin{equation*}
	\bF = \begin{bmatrix}
		\rho a M \\
		\rho a^2(M^2 + \frac{1}{\gamma}) \\
		\rho a^3M(\frac{1}{2}M^2 + \frac{1}{\gamma-1}) \\
	\end{bmatrix} = \bF^+ + \bF^-,\quad
	\bF^{\pm} = \begin{bmatrix}
			\pm\frac{1}{4}\rho a (M\pm1)^2 \\
			\pm\frac{1}{4}\rho a (M\pm1)^2 v + p^{\pm} \\
			\pm\frac{1}{4}\rho a (M\pm1)^2 H \\
		\end{bmatrix},
\end{equation*}
with the enthalpy $H=(E+p)/\rho$,
and the pressure splitting
$p^{\pm} = \frac{1}{2}(1\pm\gamma M)p$.
This FVS gives a quadratic differentiable splitting with respect to the Mach number.

\begin{remark}
	Different FVS may lead to different stability conditions but it is difficult to perform the analysis theoretically.
	We provide experimental CFL numbers for different FVS in some 1D tests.
	Our numerical tests in  \Cref{sec:results} show that the AF methods based on the FVS generally give better results than the JS,
	and the LLF FVS is the best among all the three FVS in terms of the CFL number and non-oscillatory property for high-speed flows involving strong discontinuities.
\end{remark}

\subsection{Time discretization}
The fully-discrete scheme is obtained by using the SSP-RK3 method \cite{Gottlieb_2001_Strong_SRa}
\begin{equation}\label{eq:ssp_rk3}
	\begin{aligned}
		\bU^{*} &= \bU^{n}+\Delta t^n \bm{L}\left(\bU^{n}\right),\\
		\bU^{**} &= \frac{3}{4}\bU^{n}+\frac{1}{4}\left(\bU^{*}+\Delta t^n \bm{L}\left(\bU^{*}\right)\right),\\
		\bU^{n+1} &= \frac{1}{3}\bU^{n}+\frac{2}{3}\left(\bU^{**}+\Delta t^n \bm{L}\left(\bU^{**}\right)\right),
	\end{aligned}
\end{equation}
where $\bm{L}$ is the right-hand side of the semi-discrete schemes \cref{eq:semi_av_1d} or  \cref{eq:semi_pnt_1d}.
The time step size is determined by the usual CFL condition
\begin{equation}\label{eq:cfl_condition}
	\Delta t^n = \dfrac{C_{\texttt{CFL}}}{\max\limits_{i}\{\varrho(\overline{\bU}_i)/\Delta x_i\}}.
\end{equation}

\section{2D active flux methods on Cartesian meshes}\label{sec:2d_af_schemes}
This section presents the generalized AF methods using the method of lines for the 2D hyperbolic conservation laws
\begin{equation}\label{eq:2d_hcl}
	\bU_t + \bF_1(\bU)_x + \bF_2(\bU)_y = 0,\quad
	\bU(x,y,0) = \bU_0(x,y).
\end{equation}
We will consider the scalar conservation law
\begin{equation}\label{eq:2d_scalar}
	u_t + f_1(u)_x + f_2(u)_y = 0,\quad u(x,y,0)=u_0(x,y),
\end{equation}
and the Euler equations with $\bU=(\rho, \rho\bv, E)^\top$, 
$\bF_1=(\rho v_1, \rho v_1^2 + p, \rho v_1v_2, (E+p)v_1)^\top$, 
$\bF_2=(\rho v_2, \rho v_1v_2, \rho v_2^2 + p, (E+p)v_2)^\top$,
where $\bv = (v_1, v_2)$ is the velocity vector, and the other notations are the same as for 1D in \Cref{sec:1d_af_schemes}.
The SSP-RK3 method is used to obtain the fully-discrete method.

Assume that a 2D computational domain is divided into $N_1\times N_2$ cells,
$I_{i,j} = [x_{\xl}, x_{\xr}]\times[y_{\yl}, y_{\yr}]$ with the cell sizes $\Delta x_i = x_{\xr}-x_{\xl}, \Delta y_j = y_{\yr}-y_{\yl}$, and cell centers $(x_i, y_j) = \left(\frac12(x_{\xl} + x_{\xr}), \frac12(y_{\yl} + y_{\yr})\right)$, $i=1,\cdots,N_1,~ j=1,\cdots,N_2$.
The DoFs consist of the cell average $\overline{\bU}_{i,j}(t) = \frac{1}{\Delta x_i\Delta y_j}\int_{I_{i,j}} \bU_h(x,y,t) ~\dd x\dd y$,
the face-centered values $\bU_{i+\frac12,j}(t) = \bU_h(x_{i+\frac12}, y_j, t),
\bU_{i,j+\frac12}(t) = \bU_h(x_{i}, y_{j+\frac12}, t)$,
and the value at the corner $\bU_{i+\frac12,j+\frac12}(t) = \bU_h(x_{i+\frac12}, y_{j+\frac12}, t)$,
where $\bU_h(x,y,t)$ is the numerical solution.
A sketch of the DoFs for the third-order AF method (for the scalar case) is given in \Cref{fig:2d_af_dofs}.

\begin{figure}[hptb!]
	\centering
	\includegraphics[width=0.35\linewidth]{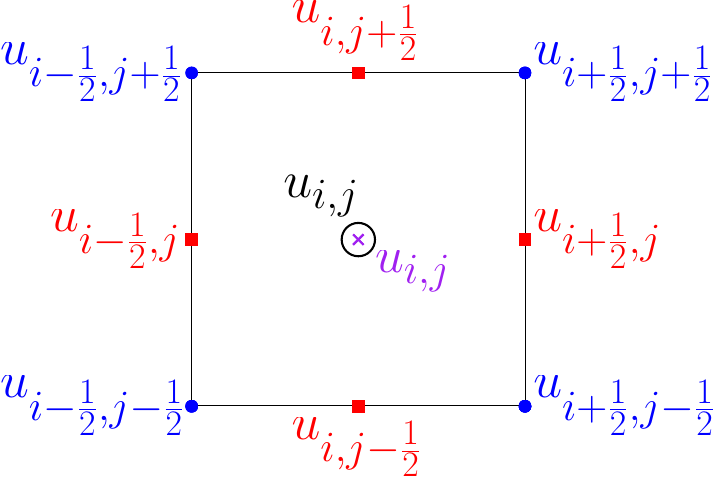}
	\caption{The DoFs for the third-order AF method: cell average (circle), face-centered values (squares), values at corners (dots). Note that the cell-centered point value $u_{i,j}$ (cross) is used in constructing the scheme, but does not belong to the DoFs.}
	\label{fig:2d_af_dofs}
\end{figure}

The cell average is evolved as follows
\begin{equation}\label{eq:semi_av_2d}
	\frac{\dd \overline{\bU}_{i,j}}{\dd t} =
	-\frac{1}{\Delta x_i}\left(\widehat{\bF}_{\xr,j} - \widehat{\bF}_{\xl,j}\right)
	-\frac{1}{\Delta y_j}\left(\widehat{\bF}_{i,\yr} - \widehat{\bF}_{i,\yl}\right),
\end{equation}
where $\widehat{\bF}_{\xr,j}$ and $\widehat{\bF}_{i,\yr}$ are the numerical fluxes
\begin{equation}\label{eq:2d_num_fluxes}
	\widehat{\bF}_{\xr,j} = \frac{1}{\Delta y_j}\int_{y_{\yl}}^{y_{\yr}} \bF_1(\bU_h(x_{\xr},y))~\dd y,~
	\widehat{\bF}_{i,\yr} = \frac{1}{\Delta x_i}\int_{x_{\xl}}^{x_{\xr}} \bF_2(\bU_h(x,y_{\yr}))~\dd x.
\end{equation}
To achieve third-order accuracy, one can use Simpson's rule
\begin{equation}\label{eq:flux_simpson}
	\widehat{\bF}_{\xr,j} = \frac16\left(\bF_1(\bU_{\xr,\yl}) + 4\bF_1(\bU_{\xr,j}) + \bF_1(\bU_{\xr,\yr})\right).
\end{equation}

\subsection{Point value update using flux vector splitting}
For the evolution of the point values, consider the following general form
\begin{equation}\label{eq:semi_pnt_2d}
	\frac{\dd \bU_{\sigma}}{\dd t} = - \bm{\mathcal{R}}\left(\overline{\bU}_c(t), \bU_{\sigma'}(t)\right),~c\in\mathcal{C}(\sigma), \sigma'\in\Sigma(\sigma),
\end{equation}
where $\bm{\mathcal{R}}$ is a consistent approximation of $\partial\bF_1/\partial x + \partial\bF_2/\partial y$ at the point $\sigma$, $\mathcal{C}(\sigma)$ and $\Sigma(\sigma)$ are the spatial stencils containing the cell averages and point values, respectively.
One can use the JS in \cite{Abgrall_2023_active}, or employ the FVS for the point value update.
E.g. for the point value at the corner $(x_{\xr}, y_{\yr})$ the FVS-based update reads
\begin{equation}\label{eq:2d_semi_fvs_node}
	\dfrac{\dd \bU_{\xr,\yr}}{\dd t} 
	= - \sum_{\ell=1}^2 \left[\widetilde{\bD}_\ell^+\bF_\ell^{+}(\bU)+ \widetilde{\bD}_\ell^-\bF_\ell^{-}(\bU) \right]_{\xr,\yr},
\end{equation}
where the fluxes are split as $ \bF_\ell= \bF_\ell^{+} + \bF_\ell^{-}$,
$\lambda\left(\frac{\partial\bF_\ell^{+}}{\partial\bU}\right) \geqslant 0$, $\lambda\left(\frac{\partial\bF_\ell^{-}}{\partial\bU}\right) \leqslant 0$.
The explicit expressions of the 2D FVS used in this paper can be found in Appendix Section \ref{sec:2d_fvs}.
The finite difference operators $\widetilde{\bD}_1^{\pm}$ and $\widetilde{\bD}_2^{\pm}$ can be obtained similarly to \Cref{sec:1d_fvs}.
For third-order accuracy, starting with a bi-parabolic reconstruction
of $\bU$ and computing a bi-parabolic interpolation of $\bF_\ell^\pm$, one thus obtains $\widetilde{\bD}_1^{\pm}$ in the $x$-direction as
	\begin{align*}
		\left(\widetilde{\bD}_1^{+}\bF_1^+\right)_{\xr,\yr} &= \dfrac{1}{\Delta x_i}\left((\bF_1)_{\xl,\yr}^{+} - 4 (\bF_1)_{i,\yr}^{+} + 3 (\bF_1)_{\xr,\yr}^{+} \right), \\
		\left(\widetilde{\bD}_1^{-}\bF_1^-\right)_{\xr,\yr} &= \dfrac{1}{\Delta x_{i+1}}\left(- 3 (\bF_1)_{\xr,\yr}^{-} + 4 (\bF_1)_{i+1,\yr}^{-} - (\bF_1)_{i+\frac32,\yr}^{-} \right).
	\end{align*}

For the face-centered point value at $(x_{\xr}, y_{j})$, the FVS-based update reads
\begin{equation}\label{eq:2d_semi_fvs_facex}
	\dfrac{\dd \bU_{\xr,j}}{\dd t} 
	= - \left[\widetilde{\bD}_1^+\bF_1^{+}(\bU)+ \widetilde{\bD}_1^-\bF_1^{-}(\bU) \right]_{\xr,j}
	- \left(\widetilde{\bD}_2\bF_2(\bU)\right)_{\xr,j},
\end{equation}
where
	\begin{align*}
		\left(\widetilde{\bD}_1^{+}\bF_1^+\right)_{\xr,j} &= \dfrac{1}{\Delta x_i}\left((\bF_1)_{\xl,j}^{+} - 4 (\bF_1)_{i,j}^{+} + 3 (\bF_1)_{\xr,j}^{+} \right), \\
		\left(\widetilde{\bD}_1^{-}\bF_1^-\right)_{\xr,j} &= \dfrac{1}{\Delta x_{i+1}}\left(- 3 (\bF_1)_{\xr,j}^{-} + 4 (\bF_1)_{i+1,j}^{-} - (\bF_1)_{i+\frac32,j}^{-} \right), \\
		\left(\widetilde{\bD}_2\bF_2\right)_{\xr,j} &= \dfrac{1}{\Delta y_j}\left((\bF_2)_{\xr,\yr} - (\bF_2)_{\xr,\yl} \right),
	\end{align*}
and the cell-centered point value is computed from the bi-parabolic reconstruction \cite{Abgrall_2023_active} as
\begin{align}
	\bU_{i,j} = \frac{1}{16}\Big[
		36\overline{\bU}_{i,j} &- 4\left(\bU_{\xl,j}+\bU_{\xr,j}+\bU_{i,\yl}+\bU_{i,\yr}\right) \nonumber\\
		&- \left(\bU_{\xl,\yl} + \bU_{\xr,\yl} + \bU_{\xl,\yr} + \bU_{\xr,\yr}\right)
	\Big]. \label{eq:2d_cell_center_parabolic}
\end{align}
The update for the point value at $(x_{i}, y_{\yr})$ is omitted here,
which is similar to \eqref{eq:2d_semi_fvs_facex}.

\subsection{Mesh alignment issue when using Jacobian splitting}\label{sec:mesh_alignment}
The mesh alignment issue was observed for the fully-discrete AF methods in \cite{Maeng_2017_Advective},
where the convergence rate reduces to 2 for the linear advection problem,
when the advection velocity is aligned with the grid.
For the generalized AF methods based on the JS, such an issue is also observed.
Consider \Cref{ex:2d_sod}, where we solve a quasi-2D Sod shock tube along the $x$-direction on a $100\times 2$ uniform mesh.
As shown in \Cref{fig:2d_sod_density}, the density based on the JS shows large deviations between the contact discontinuity and shock wave.
From \Cref{fig:2d_sod_density_decoupled}, it can be seen that the solutions of the DoFs at the corner $(x_{\xr},y_{\yr})$ and horizontal face $(x_i, y_{\yr})$ are decoupled from that at the vertical face $(x_{\xr}, y_j)$ and cell averages.
The reason is complicated because the mesh alignment issue seems to be caused by the decoupled point value update and its interaction with the JS.


\subsection{Boundary treatment}
	The  numerical boundary conditions can be implemented using ghost cells as usual finite volume methods.
	Take the reflective boundary for the Euler equations as an example.
	Let $x=x_{N_1-\frac12}$ be the boundary,
	then the cell averages and point values in the ghost cell $I_{N_1,j}$ are given by
	\begin{align*}
		&\overline{\bU}_{N_1,j} = \mathcal{M}(\overline{\bU}_{N_1-1,j}),~
		\bU_{N_1+\frac12,j} = \mathcal{M}(\bU_{N_1-\frac32,j}), \\
		&\bU_{N_1,j-\frac12} = \mathcal{M}(\bU_{N_1-1,j-\frac12}),~
		\bU_{N_1+\frac12,j-\frac12} = \mathcal{M}(\bU_{N_1-\frac32,j-\frac12}),
	\end{align*}
	where $\mathcal{M}$ reverses the sign of the $\rho v_1$ component while keeping others unchanged.
	Then the point value update at the boundary can be computed in the same way as the interior points, but the numerical flux on the boundary for the cell average is computed through the LLF flux as suggested in \cite{Abgrall_2023_Activea}.
	For instance, the flux $\bF_1(\bU_{N_1-\frac12,j-\frac12})$ in the right-hand side of \cref{eq:flux_simpson} is replaced by
	$\widehat{\bF}_1^{\texttt{LLF}}(\bU_{N_1-1,j-\frac12}, \mathcal{M}(\bU_{N_1-1,j-\frac12}))$.

\section{2D bound-preserving active flux methods}\label{sec:2d_limitings}
	In this paper, the admissible state set $\mathcal{G}$ is assumed to be convex.
	Two cases are considered.
	For the scalar conservation law \cref{eq:2d_scalar}, its solutions
	satisfy a strict maximum principle (MP) \cite{Dafermos_2000_Hyperbolic_book}, i.e.,
	\begin{equation}\label{eq:2d_scalar_g}
		\mathcal{G} = \left\{ u ~|~ m_0 \leqslant u \leqslant M_0 \right\},
		\quad m_0 = \min_{x,y} u_0(x,y), ~M_0 = \max_{x,y} u_0(x,y).
	\end{equation}
	For the compressible Euler equations, the admissible state set is
	\begin{equation}\label{eq:2d_euler_g}
		\mathcal{G} = \left\{\bU = \left(\rho, \rho\bv, E\right) ~\Big|~ \rho > 0,~ p = (\gamma-1)\left(E - \norm{\rho \bv}^2/(2\rho)\right) > 0 \right\},
	\end{equation}
	which is convex, see e.g. \cite{Zhang_2011_Positivity_JoCP}.
\begin{definition}
	An AF method is called \emph{bound-preserving} (BP) if starting from cell averages and point values in the admissible state set $\mathcal G$, the cell averages and point values remain in $\mathcal G$ at the next time step.
\end{definition}
Note that to avoid the effect of the round-off error, we need to choose the desired lower bounds for the density and pressure.
In the numerical tests, we will enforce $\rho\geqslant\varepsilon^{\rho}$, $p\geqslant\varepsilon^{p}$ with $\varepsilon^{\rho}, \varepsilon^{p}$ to be defined later.
Since the DoFs in the AF methods include both cell averages and point values,
it is necessary to design suitable BP limitings for both of them to achieve the BP property.
The limiting for the cell average has not been addressed much in the literature, except for a very recent work \cite{Abgrall_2023_Activea}.
The 1D limitings can be reduced from this Section, given in Section \ref{sec:1d_limitings} in Appendix.

\subsection{Convex limiting for the cell average}\label{sec:2d_limiting_average}
This section presents a convex limiting approach to achieve the BP property of the cell average update.
The basic idea of the convex limiting approaches \cite{Guermond_2018_Second_SJoSC, Hajduk_2021_Monolithic_C&MwA, Kuzmin_2020_Monolithic_CMiAMaE} is to enforce the preservation of local or global bounds by constraining individual numerical fluxes.
The BP or invariant domain-preserving (IDP) properties of flux-limited approximations are shown using representations in terms of intermediate states that stay in convex admissible state sets \cite{Guermond_2018_Second_SJoSC, Guermond_2019_Invariant_CMiAMaE}.
The low-order scheme is chosen as the first-order LLF scheme
\begin{equation*}
	\overline{\bU}^{\texttt{L}}_{i,j} = \overline{\bU}^{n}_{i,j}
	- \mu_{1,i}\left(\widehat{\bF}^{\texttt{L}}_{\xr,j} - \widehat{\bF}^{\texttt{L}}_{\xl,j}\right)
	- \mu_{2,j}\left(\widehat{\bF}^{\texttt{L}}_{i,\yr} - \widehat{\bF}^{\texttt{L}}_{i,\yl}\right),
\end{equation*}
where $\widehat{\bF}^{\texttt{L}}_{\xr,j}$ and $\widehat{\bF}^{\texttt{L}}_{i,\yr}$ are the LLF fluxes.
Take the $x$-direction as an example,
\begin{align}\label{eq:2d_x_llf_flux}
	\widehat{\bF}^{\texttt{L}}_{\xr,j} :=&\ 
	\widehat{\bF}_1^{\texttt{LLF}}(\overline{\bU}^{n}_{i,j}, \overline{\bU}^{n}_{i+1,j}) \\ =&\ \frac12\left(\bF_1(\overline{\bU}^{n}_{i,j}) + \bF_1(\overline{\bU}^{n}_{i+1,j})\right)
	- \frac{(\alpha_1)_{\xr,j}}{2}\left(\overline{\bU}^{n}_{i+1,j} - \overline{\bU}^{n}_{i,j}\right), \nonumber\\
	(\alpha_1)_{\xr,j} =&\ \max\{\varrho_1(\overline{\bU}^{n}_{i,j}), ~\varrho_1(\overline{\bU}^{n}_{i+1,j})\}, \nonumber\\
	\mu_{1,i} =&\  \Delta t^n / \Delta x_i, \nonumber
\end{align}
where $\varrho_1$ is the spectral radius of $\partial\bF_1/\partial\bU$.
Note that here $\alpha_{\xr,j}$ is not the same as the one in the LLF FVS.
Following \cite{Guermond_2016_Invariant_SJoNA}, the first-order LLF scheme can be rewritten as
\begin{align}\label{eq:2d_lo_decomp}
	\overline{\bU}^{\texttt{L}}_{i,j} =& \left[1-\mu_{1,i}\left((\alpha_1)_{\xl,j}+(\alpha_1)_{\xr,j}\right)
	-\mu_{2,j}\left((\alpha_2)_{i,\yl}+(\alpha_2)_{i,\yr}\right)\right]
	\overline{\bU}^{n}_{i,j} \nonumber\\
	&+ \mu_{1,i}(\alpha_1)_{\xl,j}\widetilde{\bU}_{\xl,j} + \mu_{1,i}(\alpha_1)_{\xr,j}\widetilde{\bU}_{\xr,j} \nonumber\\
	&+ \mu_{2,j}(\alpha_2)_{i,\yl}\widetilde{\bU}_{i,\yl} + \mu_{2,j}(\alpha_2)_{i,\yr}\widetilde{\bU}_{i,\yr},
\end{align}
with four intermediate states,
and the explicit expressions in the $x$-direction are
\begin{equation}\label{eq:2d_llf_inter_states}
			\widetilde{\bU}_{i\pm\frac12,j} = \frac12\left(\overline{\bU}^{n}_{i,j} + \overline{\bU}^{n}_{i\pm1,j}\right)
			\pm \frac{1}{2(\alpha_1)_{i\pm\frac12,j}}\left[\bF_1(\overline{\bU}^{n}_{i,j}) - \bF_1(\overline{\bU}^{n}_{i\pm1,j})\right].
\end{equation}
	The proofs of $\widetilde{\bU}_{i\pm\frac12,j}, \widetilde{\bU}_{i,j\pm\frac12}
  \in \mathcal{G}$ are given in Appendix Section \ref{sec:2d_llf_bp}, for the scalar case and Euler equations.
\begin{lemma}\label{lem:2d_llf_g}
	If the time step size $\Delta t^n$ satisfies
	\begin{equation}\label{eq:2d_convex_combination_dt}
		\Delta t^n \leqslant \frac12\min\left\{\dfrac{\Delta x_i}{(\alpha_1)_{\xl,j}+(\alpha_1)_{\xr,j}}, \dfrac{\Delta y_j}{(\alpha_2)_{i,\yl}+(\alpha_2)_{i,\yr}}\right\},
	\end{equation}
	then \cref{eq:2d_lo_decomp} is a convex combination,
	and the first-order LLF scheme is BP.
\end{lemma}
The proof (see e.g. \cite{Guermond_2016_Invariant_SJoNA,Perthame_1996_positivity_NM}) relies on $\overline{\bU}_{i,j}^n, \widetilde{\bU}_{i\pm\frac12,j}, \widetilde{\bU}_{i,j\pm\frac12} \in \mathcal{G}$ and the convexity of $\mathcal{G}$.

Upon defining the anti-diffusive flux $\Delta \widehat{\bF}_{i\pm\frac12,j} = \widehat{\bF}^{\texttt{H}}_{i\pm\frac12,j} - \widehat{\bF}^{\texttt{L}}_{i\pm\frac12,j}$,
and $\widehat{\bF}^{\texttt{H}}_{i\pm\frac12,j}$ is given in \cref{eq:2d_num_fluxes}, a forward-Euler step applied to the semi-discrete high-order scheme for the cell average \cref{eq:semi_av_2d} can be written as
\begin{align}
	\overline{\bU}^{\texttt{H}}_{i,j} 
	=&\ \overline{\bU}^{n}_{i,j}
	- \mu_{1,i}\left(\widehat{\bF}^{\texttt{L}}_{\xr,j} - \widehat{\bF}^{\texttt{L}}_{\xl,j}\right)
	- \mu_{2,j}\left(\widehat{\bF}^{\texttt{L}}_{i,\yr} - \widehat{\bF}^{\texttt{L}}_{i,\yl}\right) \nonumber\\
	&- \mu_{1,i}\left(\Delta\widehat{\bF}_{\xr,j} - \Delta\widehat{\bF}_{\xl,j}\right)
	- \mu_{2,j}\left(\Delta\widehat{\bF}_{i,\yr} - \Delta\widehat{\bF}_{i,\yl}\right) \nonumber\\
	=&\ \left[1-\mu_{1,i}\left((\alpha_1)_{\xl,j}+(\alpha_1)_{\xr,j}\right)
	-\mu_{2,j}\left((\alpha_2)_{i,\yl}+(\alpha_2)_{i,\yr}\right)\right]
	\overline{\bU}^{n}_{i,j} \nonumber\\
	&+ \mu_{1,i}(\alpha_1)_{\xl,j}\widetilde{\bU}_{\xl,j}^{\texttt{H},+}
	+ \mu_{1,i}(\alpha_1)_{\xr,j}\widetilde{\bU}_{\xr,j}^{\texttt{H},-} \nonumber\\
	&+ \mu_{2,j}(\alpha_2)_{i,\yl}\widetilde{\bU}_{i,\yl}^{\texttt{H},+}
	+ \mu_{2,j}(\alpha_2)_{i,\yr}\widetilde{\bU}_{i,\yr}^{\texttt{H},-} \label{eq:2d_ho_decomp},
\end{align}
with the high-order intermediate states
\begin{equation*}
	\widetilde{\bU}_{i\pm\frac12,j}^{\texttt{H},\mp} := \widetilde{\bU}_{i\pm\frac12,j} \mp \frac{\Delta\widehat{\bF}_{i\pm\frac12,j}}{(\alpha_1)_{i\pm\frac12,j}}, \quad
	\widetilde{\bU}_{i,j\pm\frac12}^{\texttt{H},\mp} := \widetilde{\bU}_{i,j\pm\frac12} \mp \frac{\Delta\widehat{\bF}_{i,j\pm\frac12}}{(\alpha_2)_{i,j\pm\frac12}}.
\end{equation*}
With the low-order scheme \cref{eq:2d_lo_decomp} and high-order scheme \cref{eq:2d_ho_decomp} having the same abstract form, one can blend them to define the limited scheme for the cell average as
\begin{align}
	\overline{\bU}^{\texttt{Lim}}_{i,j} =&\  \left[1-\mu_{1,i}\left((\alpha_1)_{\xl,j}+(\alpha_1)_{\xr,j}\right)
	-\mu_{2,j}\left((\alpha_2)_{i,\yl}+(\alpha_2)_{i,\yr}\right)\right]
	\overline{\bU}^{n}_{i,j} \nonumber\\
	&+ \mu_{1,i}(\alpha_1)_{\xl,j}\widetilde{\bU}_{\xl,j}^{\texttt{Lim},+}
	+ \mu_{1,i}(\alpha_1)_{\xr,j}\widetilde{\bU}_{\xr,j}^{\texttt{Lim},-} \nonumber\\
	&+ \mu_{2,j}(\alpha_2)_{i,\yl}\widetilde{\bU}_{i,\yl}^{\texttt{Lim},+}
	+ \mu_{2,j}(\alpha_2)_{i,\yr}\widetilde{\bU}_{i,\yr}^{\texttt{Lim},-} \label{eq:2d_limited_decomp},
\end{align}
where the limited intermediate states are
\begin{equation}\label{eq:2d_limited_inter_states}
	\begin{aligned}
	&\widetilde{\bU}_{i\pm\frac12,j}^{\texttt{Lim},\mp} = \widetilde{\bU}_{i\pm\frac12,j} \mp \frac{\Delta\widehat{\bF}^{\texttt{Lim}}_{i\pm\frac12,j}}{(\alpha_1)_{i\pm\frac12,j}}
	:= \widetilde{\bU}_{i\pm\frac12,j} \mp \frac{\theta_{i\pm\frac12,j}\Delta\widehat{\bF}_{i\pm\frac12,j}}{(\alpha_1)_{i\pm\frac12,j}},\\
	&\widetilde{\bU}_{i,j\pm\frac12}^{\texttt{Lim},\mp} = \widetilde{\bU}_{i,j\pm\frac12} \mp \frac{\Delta\widehat{\bF}^{\texttt{Lim}}_{i,j\pm\frac12}}{(\alpha_2)_{i,j\pm\frac12}}
	:= \widetilde{\bU}_{i,j\pm\frac12} \mp \frac{\theta_{i,j\pm\frac12}\Delta\widehat{\bF}_{i,j\pm\frac12}}{(\alpha_2)_{i,j\pm\frac12}},
	\end{aligned}
\end{equation}
	and $\theta_{i\pm \frac12,j}, \theta_{i,j\pm \frac12}\in[0,1]$ are the blending coefficients.
	The limited scheme \cref{eq:2d_limited_decomp} reduces to the first-order LLF scheme if $\theta_{i\pm\frac12,j}=\theta_{i,j\pm\frac12}=0$,
	and recovers the high-order AF scheme \cref{eq:semi_av_2d} when $\theta_{i\pm\frac12,j}=\theta_{i,j\pm\frac12}=1$.

\begin{proposition}
	If the cell average at the last time step $\overline{\bU}_{i,j}^n$ and the limited intermediate states $\widetilde{\bU}_{i\pm\frac12,j}^{\texttt{Lim},\mp}$, $\widetilde{\bU}_{i,j\pm\frac12}^{\texttt{Lim},\mp}$ belong to the admissible state set $\mathcal{G}$,
	then the limited average update \eqref{eq:2d_limited_decomp} is BP, i.e., $\overline{\bU}^{\texttt{Lim}}_{i,j} \in \mathcal{G}$, under the CFL condition \eqref{eq:2d_convex_combination_dt}.
	If the SSP-RK3 \eqref{eq:ssp_rk3} is used for the time integration,
	the high-order scheme is also BP.	
\end{proposition}

\begin{proof}
    Under the constraint \cref{eq:2d_convex_combination_dt}, the limited cell average update $\overline{\bU}^{\texttt{Lim}}_{i,j}$ is a convex combination of $\overline{\bU}_{i,j}^n$ $\widetilde{\bU}_{i\pm\frac12,j}^{\texttt{Lim},\mp}$, and $\widetilde{\bU}_{i,j\pm\frac12}^{\texttt{Lim},\mp}$,
	thus it belongs to $\mathcal{G}$ due to the convexity of $\mathcal{G}$.
	Because the SSP-RK3 is a convex combination of forward-Euler stages,
	the high-order scheme equipped with the SSP-RK3 is also BP according to the convexity.
\end{proof}

\begin{remark}
    The scheme \cref{eq:2d_limited_decomp} is conservative as it amounts to using the $x$-directional numerical flux
    $    \widehat{\bF}^{\texttt{L}}_{\xr,j} + \theta_{\xr,j}\Delta \widehat{\bF}_{\xr,j} = \theta_{\xr,j} \widehat{\bF}^{\texttt{H}}_{\xr,j} + (1 - \theta_{\xr,j}) \widehat{\bF}^{\texttt{L}}_{\xr,j}$,
    which is a convex combination of the high-order and low-order fluxes.
\end{remark}

\begin{remark}
	It should be noted that the time step size \cref{eq:2d_convex_combination_dt} is determined based on the solutions at $t^n$.
	If the constraint is not satisfied at the later stage of the SSP-RK3,
	the BP property may not be achieved because \cref{eq:2d_limited_decomp} is no longer a convex combination.
	In our implementation, we start from the usual CFL condition \cref{eq:cfl_condition}.
    Then, if the high-order AF solutions need BP limitings and \cref{eq:2d_llf_inter_states} is not BP or \cref{eq:2d_convex_combination_dt} is not satisfied,
	the numerical solutions are set back to the last time step,
	and we rerun with a halved time step size until \cref{eq:2d_llf_inter_states} is BP and the constraint \cref{eq:2d_convex_combination_dt} is satisfied.
	This is a typical implementation in other BP methods, e.g. \cite{Zhang_2010_positivity_JoCP}.
\end{remark}

The remaining task is to determine the coefficients at each interface $\theta_{i\pm\frac12,j}, \theta_{i,j\pm\frac12}$ such that $\widetilde{\bU}_{i\pm\frac12,j}^{\texttt{Lim},\mp}, \widetilde{\bU}_{i,j\pm\frac12}^{\texttt{Lim},\mp}\in\mathcal{G}$ and stay as close as possible to the high-order solutions $\widetilde{\bU}_{i\pm\frac12,j}^{\texttt{H}}, \widetilde{\bU}_{i,j\pm\frac12}^{\texttt{H}}$,
i.e., the goal is to find the largest $\theta_{i\pm\frac12,j}, \theta_{i,j\pm\frac12}\in[0,1]$ such that $\widetilde{\bU}_{i\pm\frac12,j}^{\texttt{Lim},\mp}, \widetilde{\bU}_{i,j\pm\frac12}^{\texttt{Lim},\mp}\in\mathcal{G}$.
The explanations will be given for the $x$-direction.

\subsubsection{Application to scalar conservation laws}\label{sec:2d_limiting_average_scalar}
This section is devoted to applying the convex limiting approach to scalar conservation laws \cref{eq:2d_scalar},
such that the limited cell averages \cref{eq:2d_limited_decomp} satisfy the MP $u^{\min}_{i,j} \leqslant \bar{u}_{i,j}^{\texttt{Lim}} \leqslant u^{\max}_{i,j}$,
where $u^{\min}_{i,j} = \min\mathcal{N}$, $u^{\max}_{i,j} = \max\mathcal{N}$,
and $\mathcal{N}$ will be defined later.
According to the convex decomposition, the blending coefficient $\theta_{i+\frac12,j}\in[0,1]$ or $\Delta \hat{f}_{\xr,j}^{\texttt{Lim}}$ should be chosen such that $u^{\min}_{i,j} \leqslant \tilde{u}_{i+\frac12,j}^{\texttt{Lim},-} \leqslant u^{\max}_{i,j}$,
$u^{\min}_{i+1,j} \leqslant \tilde{u}_{i+\frac12,j}^{\texttt{Lim},+} \leqslant u^{\max}_{i+1,j}$.
Solving the first condition, i.e. $u^{\min}_{i,j} \leqslant \tilde{u}_{i+\frac12,j} - \Delta\hat{f}_{i+\frac12,j}^{\texttt{Lim}}/\alpha_{i+\frac12,j}\leqslant u^{\max}_{i,j}$,
one has $\Delta \hat{f}_{\xr,j}^{\texttt{Lim}}\leqslant \alpha_{\xr,j}(\tilde{u}_{\xr,j}-u^{\min}_{i,j})$ if $\Delta \hat{f}_{\xr,j} \geqslant 0$,
or $\Delta \hat{f}_{\xr,j}^{\texttt{Lim}}\geqslant \alpha_{\xr,j}(\tilde{u}_{\xr,j}-u^{\max}_{i,j})$ if $\Delta \hat{f}_{\xr,j} < 0$.
Solving the second condition $u^{\min}_{i+1,j} \leqslant \tilde{u}_{i+\frac12,j}^{\texttt{Lim},+} \leqslant u^{\max}_{i+1,j}$ in the same way and combining the two sets of results yields
\begin{align*}
	&\Delta\hat{f}^{\texttt{Lim}}_{\xr,j} =
	\begin{cases}
		\min\big\{\Delta\hat{f}_{\xr,j}, \Delta\hat{f}^{+}_{\xr,j} \big\},
		&\text{if}~ \Delta \hat{f}_{\xr,j} \geqslant 0, \\
		\max\big\{\Delta\hat{f}_{\xr,j}, \Delta\hat{f}^{-}_{\xr,j} \big\}, &\text{otherwise}, \\
	\end{cases}\\
	&\Delta\hat{f}^{+}_{\xr,j} = (\alpha_1)_{\xr,j}\min\big\{ \tilde{u}_{\xr,j}-u^{\min}_{i,j},
	u^{\max}_{i+1,j}-\tilde{u}_{\xr,j} \big\}, \\
	&\Delta\hat{f}^{-}_{\xr,j} =
	(\alpha_1)_{\xr,j}\max\big\{ u^{\min}_{i+1,j}-\tilde{u}_{\xr,j},
	\tilde{u}_{\xr,j}-u^{\max}_{i,j} \big\}.
\end{align*}
Finally, the limited numerical flux is
\begin{equation}\label{eq:2d_flux_limited_scalar}
	\hat{f}^{\texttt{Lim}}_{\xr,j} = \hat{f}^{\texttt{L}}_{\xr,j} + \Delta\hat{f}^{\texttt{Lim}}_{\xr,j}.
\end{equation}
If considering the global MP, $\mathcal{N} = \bigcup_{i,j,\sigma}\{\bar{u}_{i,j}^n, u_{\sigma}^n\}$.
One can also enforce the local MP, which helps to suppress spurious oscillations \cite{Guermond_2018_Second_SJoSC, Kuzmin_2012_Flux_book, Guermond_2016_Invariant_SJoNA},
by choosing
\begin{align*}
	&\mathcal{N} = \left\{\bar{u}_{i,j}^n, ~\tilde{u}_{\xl,j}, ~\tilde{u}_{\xr,j}, ~\tilde{u}_{i,\yl}, ~\tilde{u}_{i,\yr}, \bar{u}_{i-1,j}^n, \bar{u}_{i+1,j}^n, \bar{u}_{i,j-1}^n, \bar{u}_{i,j+1}^n\right\},
\end{align*}
which includes the intermediate states and neighboring cell averages.

\subsubsection{Application to the compressible Euler equations}\label{sec:2d_limiting_average_euler}

This section aims at enforcing the positivity of density and pressure.
To avoid the effect of the round-off error, we need to choose the desired lower bounds.
Denote the lowest density and pressure in the domain by
\begin{equation}\label{eq:2d_lowest_rho_prs}
	\varepsilon^{\rho}:=\min_{i,j,\sigma}\{\overline{\bU}_{i,j}^{n,\rho}, \bU_{\sigma}^{n,\rho}\},~ \varepsilon^{p}:=\min_{i,j,\sigma}\{p(\overline{\bU}_{i,j}^{n}), p(\bU_{\sigma}^{n})\},
\end{equation}
where $\bU^{*,\rho}$ and $p(\bU^{*})$ denote the density component and pressure recovered from $\bU^{*}$, respectively,
and $\sigma$ denotes the locations of point values in the DoFs.
The limiting \cref{eq:2d_limited_inter_states} is feasible if the constraints are satisfied by the first-order LLF intermediate states \cref{eq:2d_llf_inter_states}, thus the lower bounds can be defined as
\begin{align*}
	&\varepsilon_{i,j}^\rho := \min\{10^{-13}, \varepsilon^{\rho}, \widetilde{\bU}_{\xl,j}^{\rho}, \widetilde{\bU}_{\xr,j}^{\rho}, \widetilde{\bU}_{i,\yl}^{\rho}, \widetilde{\bU}_{i,\yr}^{\rho}\},\\
	&\varepsilon_{i,j}^p := \min\{10^{-13}, \varepsilon^{p}, p(\widetilde{\bU}_{\xl,j}), p(\widetilde{\bU}_{\xr,j}), p(\widetilde{\bU}_{i,\yl}), p(\widetilde{\bU}_{i,\yr})\}.
\end{align*}

i) {\bfseries Positivity of density.}
The first step is to impose the density positivity
$\widetilde{\bU}_{\xr,j}^{\texttt{Lim},\pm,\rho} \geqslant \bar{\varepsilon}_{\xr,j}^{\rho}:=\min\{\varepsilon_{i,j}^\rho, \varepsilon_{i+1,j}^\rho\}$.
Similarly to the derivation of the scalar case,
the corresponding density component of the limited anti-diffusive flux is
\begin{equation*}
	\Delta\widehat{\bF}^{\texttt{Lim},\rho}_{\xr,j} = \begin{cases}
		\min\left\{\Delta\widehat{\bF}^{\rho}_{\xr,j}, ~(\alpha_1)_{\xr,j}\left(\widetilde{\bU}^{\rho}_{\xr,j}-\bar{\varepsilon}_{\xr,j}^{\rho}\right) \right\}, ~&\text{if}~ \Delta \widehat{\bF}^{\rho}_{\xr,j} \geqslant 0, \\
		\max\left\{\Delta\widehat{\bF}^{\rho}_{\xr,j}, ~(\alpha_1)_{\xr,j}\left(\bar{\varepsilon}_{\xr,j}^{\rho}-\widetilde{\bU}^{\rho}_{\xr,j}\right) \right\}, ~&\text{otherwise}. \\
	\end{cases}
\end{equation*}
Then the density component of the limited numerical flux is
$\widehat{\bF}_{\xr,j}^{\texttt{Lim}, *, \rho} = \widehat{\bF}_{\xr,j}^{\texttt{L}, \rho} + \Delta\widehat{\bF}_{\xr,j}^{\texttt{Lim}, \rho}$,
with the other components remaining the same as  $\widehat{\bF}_{\xr,j}^{\texttt{H}}$.

ii) {\bfseries Positivity of pressure.}
The second step is to enforce pressure positivity
	$p(\widetilde{\bU}_{\xr,j}^{\texttt{Lim},\pm}) \geqslant \bar{\varepsilon}_{\xr,j}^{p}:=\min\{\varepsilon_{i,j}^p, \varepsilon_{i+1,j}^p\}$.
Since
\begin{equation*}
	\widetilde{\bU}_{\xr,j}^{\texttt{Lim},\pm} = \widetilde{\bU}_{\xr,j} \pm \frac{\theta_{\xr,j}\Delta \widehat{\bF}_{\xr,j}^{\texttt{Lim}, *}}{\alpha_{\xr,j}},\quad
	\Delta \widehat{\bF}_{\xr,j}^{\texttt{Lim}, *} = \widehat{\bF}_{\xr,j}^{\texttt{Lim}, *} - \widehat{\bF}_{\xr,j}^{\texttt{L}},
\end{equation*}
the constraints lead to two inequalities after some algebraic operations
\begin{equation}\label{eq:pressure_inequality}
	A_{\xr,j}\theta^2_{\xr,j} \pm B_{\xr,j}\theta_{\xr,j} \leqslant C_{\xr,j},
\end{equation}
with the coefficients (the subscript $(\cdot)_{\xr,j}$ is omitted in the right-hand side)
\begin{align*}
	A_{\xr,j} &= \dfrac{1}{2} \norm{ \Delta \widehat\bF^{\texttt{Lim}, *, \rho \bv} }_2^2
	- \Delta \widehat\bF^{\texttt{Lim}, *, \rho} \Delta \widehat\bF^{\texttt{Lim}, *, E}, \\
	B_{\xr,j} &= \alpha_1\left(\Delta \widehat\bF^{\texttt{Lim}, *, \rho} \widetilde\bU^{E}
	+ \widetilde\bU^{\rho} \Delta \widehat\bF^{\texttt{Lim}, *, E}
	- \Delta \widehat\bF^{\texttt{Lim}, *, \rho \bv}\cdot \widetilde\bU^{\rho \bv}
	- \tilde{\varepsilon} \Delta \widehat\bF^{\texttt{Lim}, *, \rho}\right), \\
	C_{\xr,j} &= \alpha_1^2\left(\widetilde\bU^{\rho} \widetilde\bU^{E}
	- \dfrac{1}{2} \norm{ \widetilde\bU^{\rho \bv} }_2^2
	- \tilde{\varepsilon} \widetilde\bU^{\rho}\right),
	\quad
	\tilde{\varepsilon} = \bar{\varepsilon}^p_{\xr,j}/(\gamma-1).
\end{align*}
Following \cite{Kuzmin_2020_Monolithic_CMiAMaE}, the inequalities \cref{eq:pressure_inequality} hold under the linear sufficient condition
\begin{equation*}
	\left(\max\{0, A_{\xr,j}\} + \abs{B_{\xr,j}}\right)\theta_{\xr,j} \leqslant C_{\xr,j},
\end{equation*}
if making use of $\theta_{\xr,j}^2 \leqslant \theta_{\xr,j},~ \theta_{\xr,j}\in[0,1]$.
Thus the coefficient can be chosen as
\begin{equation*}
	\theta_{\xr,j} = \min\left\{1,~ \dfrac{C_{\xr,j}}{\max\{0, A_{\xr,j}\} + \abs{B_{\xr,j}}}\right\},
\end{equation*}
and the final limited numerical flux is
\begin{equation}\label{eq:2d_flux_limited_Euler}
	\widehat{\bF}_{\xr,j}^{\texttt{Lim},**} = \widehat{\bF}_{\xr,j}^{\texttt{L}} + \theta_{\xr,j}\Delta\widehat{\bF}_{\xr,j}^{\texttt{Lim}, *}.
\end{equation}

\subsubsection{Shock sensor-based limiting}\label{sec:2d_limiting_shock_sensor}
Spurious oscillations are observed, especially near strong shock waves, if only the BP limitings are employed, see \Cref{ex:2d_ffs}.
We propose to further limit the numerical fluxes using another parameter $\theta_{\xr,j}^{s}$ based on shock sensors.
Consider the Jameson's shock sensor in \cite{Jameson_1981_Solutions_AJ},
\begin{equation*}
	(\varphi_1)_{i,j} = \dfrac{\abs{\bar{p}_{i+1,j} - 2\bar{p}_{i,j} + \bar{p}_{i-1,j}}}{\abs{\bar{p}_{i+1,j} + 2\bar{p}_{i,j} + \bar{p}_{i-1,j}}},
\end{equation*}
and a modified Ducros' shock sensor \cite{Ducros_1999_Large_JoCP}
\begin{equation*}
	(\varphi_2)_{i,j} = \max\left\{\dfrac{-(\nabla\cdot\bar{\bv})_{i,j}}{\sqrt{(\nabla\cdot\bar{\bv})_{i,j}^2 + (\nabla\times\bar{\bv})_{i,j}^2 + 10^{-40}}}, ~0\right\},
\end{equation*}
where
\begin{align*}
	&(\nabla\cdot\bar{\bv})_{i,j}\approx \dfrac{2\left((\bar{v}_1)_{i+1,j} - (\bar{v}_1)_{i-1,j}\right)}{\Delta x_i + \Delta x_{i+1}}
	+ \dfrac{2\left((\bar{v}_2)_{i,j+1} - (\bar{v}_2)_{i,j-1}\right)}{\Delta y_j + \Delta y_{j+1}},\\
	&(\nabla\times\bar{\bv})_{i,j}\approx \dfrac{2\left((\bar{v}_2)_{i+1,j} - (\bar{v}_2)_{i-1,j}\right)}{\Delta x_i + \Delta x_{i+1}}
	- \dfrac{2\left((\bar{v}_1)_{i,j+1} - (\bar{v}_1)_{i,j-1}\right)}{\Delta y_j + \Delta y_{j+1}},
\end{align*}
with $\bar{v}_{i,j}$ and $\bar{p}_{i,j}$ the velocity and pressure recovered from the cell average $\overline{\bU}_{i,j}$.
We consider the sign of the velocity divergence, such that the shock waves can be located better.
The blending coefficient is designed as
\begin{align*}
	&\theta_{\xr,j}^{s} = \exp(-\kappa (\varphi_1)_{\xr,j} (\varphi_2)_{\xr,j})\in (0, 1],\\
	&(\varphi_s)_{\xr,j} = \max\left\{(\varphi_s)_{i,j}, (\varphi_s)_{i+1,j}\right\}, ~s=1,2,
\end{align*}
where the problem-dependent parameter $\kappa$ adjusts the strength of the limiting, and its optimal choice needs further investigation.
The final limited numerical flux is
\begin{equation}\label{eq:2d_flux_limited_Euler_ss}
	\widehat{\bF}_{\xr,j}^{\texttt{Lim}} = \widehat{\bF}_{\xr,j}^{\texttt{L}} + \theta_{\xr,j}^{s}\Delta\widehat{\bF}_{\xr,j}^{\texttt{Lim}, **},
\end{equation}
with $\Delta\widehat{\bF}_{\xr,j}^{\texttt{Lim},**} = \widehat{\bF}_{\xr,j}^{\texttt{Lim},**} - \widehat{\bF}_{\xr,j}^{\texttt{L}}$,
and $\widehat{\bF}_{\xr,j}^{\texttt{Lim},**}$ given in \cref{eq:2d_flux_limited_Euler}.

\subsection{Scaling limiter for point value}\label{sec:2d_limiting_point}
To achieve the BP property, it is also necessary to introduce BP limiting for the point value, because using the BP limiting for cell average alone cannot guarantee the bounds, see \Cref{ex:2d_advection_discontinuity}.
As there is no conservation requirement on the point value update,
a simple scaling limiter \cite{Liu_1996_Nonoscillatory_SJoNA} is directly performed on the high-order solution rather than on the flux for the cell average.

The first step is to define suitable first-order LLF schemes.
The stencils are shown in \cref{fig:2d_af_llf}.

\begin{figure}[hptb!]
	\centering
	\includegraphics[width=0.8\linewidth]{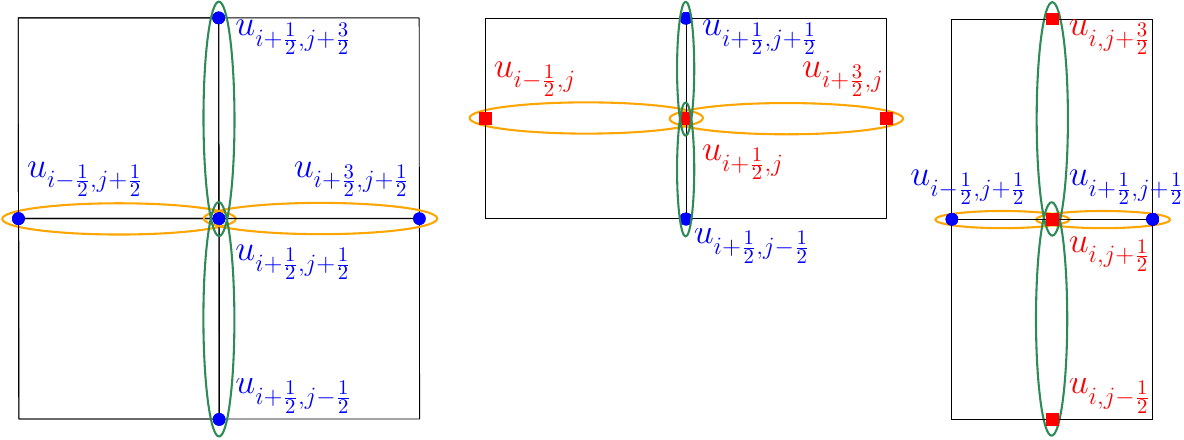}
	\caption{The stencils for the first-order LLF schemes.}
	\label{fig:2d_af_llf}
\end{figure}

For the point value at the corner, one can choose
\begin{align}\label{eq:2d_llf_node}
	\bU_{\xr,\yr}^{\texttt{L}} = \bU_{\xr,\yr}^{n}
	&- \dfrac{2\Delta t^n}{\Delta x_i+\Delta x_{i+1}}
	\left(\widehat{\bF}^{\texttt{L}}_{i+1,\yr}
	- \widehat{\bF}^{\texttt{L}}_{i,\yr}\right) \nonumber\\
	&- \dfrac{2\Delta t^n}{\Delta y_j+\Delta y_{j+1}}
	\left(\widehat{\bF}^{\texttt{L}}_{\xr,j+1}
	- \widehat{\bF}^{\texttt{L}}_{\xr,j}\right),
\end{align}
with the LLF numerical fluxes
\begin{equation*}
	\widehat{\bF}^{\texttt{L}}_{i+1,\yr}:=\widehat{\bF}_1^{\texttt{LLF}}(\bU_{\xr,\yr}^n, \bU_{i+\frac32,\yr}^n),~
	\widehat{\bF}^{\texttt{L}}_{\xr,j+1}:=\widehat{\bF}_2^{\texttt{LLF}}(\bU_{\xr,\yr}^n, \bU_{\xr,j+\frac32}^n).
\end{equation*}
Note that the $x$-directional LLF flux has been used in \cref{eq:2d_x_llf_flux}.
For the vertical face-centered point value, we choose the first-order LLF scheme as
\begin{equation}\label{eq:2d_llf_facex}
	\bU_{\xr,j}^{\texttt{L}} =\  \bU_{\xr,j}^{n}
	- \dfrac{2\Delta t^n}{\Delta x_i+\Delta x_{i+1}}
	\left(\widehat{\bF}^{\texttt{L}}_{i+1,j}
	- \widehat{\bF}^{\texttt{L}}_{i,j}\right) 
	- \dfrac{\Delta t^n}{\Delta y_j}
	\left(\widehat{\bF}^{\texttt{L}}_{\xr,j+\frac12}
	- \widehat{\bF}^{\texttt{L}}_{\xr,j-\frac12}\right),
\end{equation}
with the LLF numerical fluxes
\begin{equation*}
	\widehat{\bF}^{\texttt{L}}_{i+1,j}:=\widehat{\bF}_1^{\texttt{LLF}}(\bU_{\xr,j}^n, \bU_{i+\frac32,j}^n),~
	\widehat{\bF}^{\texttt{L}}_{\xr,j+\frac12}:=\widehat{\bF}_2^{\texttt{LLF}}(\bU_{\xr,j}^n, \bU_{\xr,\yr}^n).
\end{equation*}
The LLF scheme for the face-centered value on the horizontal face can be chosen as
\begin{equation}\label{eq:2d_llf_facey}
	\bU_{i,\yr}^{\texttt{L}} = \bU_{i,\yr}^{n}
	- \dfrac{\Delta t^n}{\Delta x_i}
	\left(\widehat{\bF}^{\texttt{L}}_{i+\frac12,\yr}
	- \widehat{\bF}^{\texttt{L}}_{i-\frac12,\yr}\right)
	- \dfrac{2\Delta t^n}{\Delta y_j+\Delta y_{j+1}}
	\left(\widehat{\bF}^{\texttt{L}}_{i,j+1}
	- \widehat{\bF}^{\texttt{L}}_{i,j}\right),
\end{equation}
with similarly defined LLF numerical fluxes as for the vertical face.

Similarly to \Cref{lem:2d_llf_g}, it is straightforward to obtain the following Lemma.
\begin{lemma}
	The LLF schemes \eqref{eq:2d_llf_node}-\eqref{eq:2d_llf_facey} for the point value update are BP under the following time step size constraint
	\begin{align}
		\Delta t^n \leqslant \frac12\min\Bigg\{
		&\dfrac{\Delta x_i + \Delta x_{i+1}}{2\left((\alpha_1)_{i,j+\frac12}+ (\alpha_1)_{i+1,j+\frac12}\right)},
		\dfrac{\Delta y_j + \Delta y_{j+1}}{2\left((\alpha_2)_{i+\frac12,j}+ (\alpha_2)_{i+\frac12,j+1}\right)}, \nonumber\\
		&\dfrac{\Delta x_i + \Delta x_{i+1}}{2\left((\alpha_1)_{i,j}+ (\alpha_1)_{i+1,j}\right)},
		\dfrac{\Delta y_j}{(\alpha_2)_{i+\frac12,j+\frac12}+ (\alpha_2)_{i+\frac12,j-\frac12}}, \nonumber\\
		&\dfrac{\Delta x_i}{(\alpha_1)_{i+\frac12,j+\frac12}+ (\alpha_1)_{i-\frac12,j+\frac12}},
		\dfrac{\Delta y_j + \Delta y_{j+1}}{2\left((\alpha_2)_{i,j} + (\alpha_2)_{i,j+1}\right)} 
		\Bigg\}, \label{eq:2d_pnt_llf_dt}
	\end{align}
	where $(\alpha_1)_*$ and $(\alpha_2)_*$ are the viscosity coefficients in the LLF schemes.
\end{lemma}

The limited solution is obtained by blending the high-order AF scheme \cref{eq:semi_pnt_2d} with the forward-Euler scheme and the LLF schemes \cref{eq:2d_llf_node}-\cref{eq:2d_llf_facey} as $\bU_{\sigma}^{\texttt{Lim}} = \theta_{\sigma} \bU_{\sigma}^{\texttt{H}} + (1-\theta_{\sigma}) \bU_{\sigma}^{\texttt{L}}$,
such that $\bU_{\sigma}^{\texttt{Lim}}\in\mathcal{G}$.

\begin{remark}
	In the FVS, the cell-centered value obtained based on Simpson's rule $\bU_i = (- \bU_{\xl} + 6\overline{\bU}_i - \bU_{\xr})/4$ in 1D or \cref{eq:2d_cell_center_parabolic} in 2D is not a convex combination, thus it is possible that $\bU_i, \bU_{i,j}\notin\mathcal{G}$.
	For the scalar case, it does not affect the BP property.
	However, for the Euler equations, the computation of $\bF_i$ (resp. $(\bF_\ell)_{i,j}$) requires that $\bU_i\in\mathcal{G}$ (resp. $\bU_{i,j}\in\mathcal{G}$), thus the scaling limiter \cite{Zhang_2010_positivity_JoCP} is applied in the cell $I_i$ (resp. $I_{i,j}$), a procedure also mentioned in \cite{Chudzik_2021_Cartesian_AMaC}.
	See more details in \Cref{rmk:cell_center}.
\end{remark}

\subsubsection{Application to scalar conservation laws}\label{sec:2d_limiting_point_scalar}
This section enforces the MP $u_{\sigma}^{\min} \leqslant u_{\sigma}^{\texttt{Lim}} \leqslant u_{\sigma}^{\max}$ using the scaling limiter \cite{Zhang_2010_maximum_JoCP}.
The limited solution is
\begin{equation}\label{eq:2d_pnt_limited_state_scalar}
	u_{\sigma}^{\texttt{Lim}} = \theta_{\sigma} u_{\sigma}^{\texttt{H}} + \left(1-\theta_{\sigma}\right) u_{\sigma}^{\texttt{L}},
\end{equation}
with the coefficient
\begin{equation*}
	\theta_{\sigma} = \min\left\{ 1,~ \left|\dfrac{u_{\sigma}^{\texttt{L}}-m_0}{u_{\sigma}^{\texttt{L}}-u_{\sigma}^{\texttt{H}}} \right|, ~\left| \dfrac{M_0-u_{\sigma}^{\texttt{L}}}{u_{\sigma}^{\texttt{H}}-u_{\sigma}^{\texttt{L}}} \right|
	\right\}.
\end{equation*}
%

	The bounds are determined by $u^{\min}_{\sigma} = \min \mathcal{N}$, $u^{\max}_{\sigma} = \max \mathcal{N}$,
	where the set $\mathcal{N}$ consists of all the DoFs in the domain,
	i.e., $\mathcal{N} = \bigcup_{i,j,\sigma}\{\bar{u}_{i,j}^n, u_{\sigma}^n\}$ for the global MP.
	One can also consider the neighboring DoFs for the local MP.
	For the point value at the corner $(x_{\xr}, y_{\yr})$, we choose
	\begin{align*}
		\mathcal{N} = \left\{u_{\xr,\yr}^n, u_{\xl,\yr}^n, u_{i+\frac32,\yr}^n, u_{\xr, \yl}^n, u_{\xr,j+\frac32}^n\right\},
	\end{align*}
	which should at least include all the DoFs that appeared in the first-order LLF scheme \cref{eq:2d_llf_node}.
	For the point value at the vertical face center $(x_{\xr}, y_{j})$, similarly we choose
	\begin{equation*}
		\mathcal{N} = \left\{u_{\xr,j}^n, u_{\xl, j}^n, u_{i+\frac32,j}^n,
		u_{\xr,\yl}^n, u_{\xr,\yr}^n\right\}.
	\end{equation*}
	For the point value at the horizontal face center $(x_{i}, y_{\yr})$, we choose
	\begin{equation*}
		\mathcal{N} = \left\{u_{i,\yr}^n, u_{i, \yl}^n, u_{i,j+\frac32}^n,
		u_{\xl,\yr}^n, u_{\xr,\yr}^n\right\}.
	\end{equation*}

\subsubsection{Application to the compressible Euler equations}\label{sec:2d_limiting_point_euler}
The limiting consists of two steps.

{\bfseries i) Positivity of density.}
First, the high-order solution $\bU_{\sigma}^{\texttt{H}}$ is modified as $\bU_{\sigma}^{\texttt{Lim}, *}$,
such that its density component satisfies $\bU_{\sigma}^{\texttt{Lim}, *, \rho} \geqslant \varepsilon^\rho_{\sigma}:=\min\{10^{-13}, \varepsilon^{\rho}, \bU_{\sigma}^{\texttt{L},\rho}\}$ with $\varepsilon^{\rho}$ given in \cref{eq:2d_lowest_rho_prs}.
Solving the inequality yields
\begin{equation*}
	\theta^{*}_{\sigma} = \begin{cases}
		\dfrac{\bU_{\sigma}^{\texttt{L}, \rho}-\varepsilon^\rho_{\sigma}}{\bU_{\sigma}^{\texttt{L}, \rho}-\bU_{\sigma}^{\texttt{H}, \rho}},&\text{if}~~ \bU_{\sigma}^{\texttt{H}, \rho} < \varepsilon^\rho_{\sigma}, \\
		1, &\text{otherwise}. \\
	\end{cases}
\end{equation*}
Then the density component of the limited solution is $\bU_{\sigma}^{\texttt{Lim}, *, \rho} = \theta^{*}_{\sigma} \bU_{\sigma}^{\texttt{H}, \rho} + (1-\theta^{*}_{\sigma}) \bU_{\xr}^{\texttt{L}, \rho}$,
with the other components remaining the same as $\bU_{\sigma}^{\texttt{H}}$.

{\bfseries ii) Positivity of pressure.}
Then the limited solution $\bU_{\sigma}^{\texttt{Lim}, *}$ is modified as $\bU_{\sigma}^{\texttt{Lim}}$,
such that it gives positive pressure, i.e., $p(\bU_{\sigma}^{\texttt{Lim}}) \geqslant \varepsilon^p_{\sigma}:=\min\{10^{-13}, \varepsilon^{p}, p(\bU_{\sigma}^{\texttt{L}})\}$, with $\varepsilon^{p}$ given in \cref{eq:2d_lowest_rho_prs}.
Let the final limited solution be
\begin{equation}\label{eq:2d_pnt_limited_state_Euler}
	\bU_{\sigma}^{\texttt{Lim}} = \theta^{**}_{\sigma} \bU_{\sigma}^{\texttt{Lim}, *} + \left(1-\theta^{**}_{\sigma}\right) \bU_{\sigma}^{\texttt{L}}.
\end{equation}
The pressure is a concave function of the conservative variables (see e.g. \cite{Zhang_2011_Maximum_PotRSAMPaES}), so that 
$p(\bU_{\sigma}^{\texttt{Lim}}) \geqslant \theta^{**}_{\sigma}p(\bU_{\sigma}^{\texttt{Lim}, *}) + \left(1-\theta^{**}_{\sigma}\right)p(\bU_{\sigma}^{\texttt{L}})$
based on Jensen's inequality and $\bU_{\sigma}^{\texttt{Lim}, *, \rho} > 0$, $\bU_{\sigma}^{\texttt{L}, \rho} > 0$, $\theta_{\sigma}^{**} \in [0,1]$.
Thus the coefficient can be chosen as
\begin{equation*}
	\theta^{**}_{\sigma} = \begin{cases}
		\dfrac{p(\bU_{\sigma}^{\texttt{L}}) - \varepsilon^p_{\sigma}}{p(\bU_{\sigma}^{\texttt{L}}) - p(\bU_{\sigma}^{\texttt{Lim}, *})},&\text{if}~~ p(\bU_{\sigma}^{\texttt{Lim}, *}) < \varepsilon^p_{\sigma}, \\
		1, &\text{otherwise}. \\
	\end{cases}
\end{equation*}

\begin{remark}\label{rmk:cell_center}
To compute the high-order FVS-based point value update, we should limit the cell-centered value $\bU_i$ in 1D (resp. $\bU_{i,j}$ in 2D) at the beginning of each Runge-Kutta stage.
For example, in 2D, we modify $\bU_{i,j}$ as $\bU_{i,j}^{\texttt{Lim}} = \theta_{i,j}\bU_{i,j} + (1-\theta_{i,j})\overline{\bU}_{i,j}$ such that
\begin{equation*}
	\bU_{i,j}^{\texttt{Lim},\rho} \geqslant \min\{10^{-13}, \overline{\bU}_{i,j}^{\rho}\},~
	p(\bU_{i,j}^{\texttt{Lim}}) \geqslant \min\{10^{-13}, p(\overline{\bU}_{i,j})\}.
\end{equation*}
The computation of $\theta_{i,j}$ is similar to the procedure in this section.
\end{remark}

Let us summarize the main results of the BP AF methods in this paper.
\begin{proposition}
	If the initial numerical solution $\overline{\bU}_{i,j}^0, \bU_{\sigma}^0\in\mathcal{G}$ for all $i,j,\sigma$,
	and the time step size satisfies \eqref{eq:2d_convex_combination_dt} and \eqref{eq:2d_pnt_llf_dt},
	then the AF methods \eqref{eq:semi_av_2d}-\eqref{eq:semi_pnt_2d} equipped with the SSP-RK3 \eqref{eq:ssp_rk3} and the BP limitings
	\begin{itemize}[leftmargin=*]
		\item \eqref{eq:2d_flux_limited_scalar} and \eqref{eq:2d_pnt_limited_state_scalar} preserve the maximum principle for scalar case;
		\item \eqref{eq:2d_flux_limited_Euler} and \eqref{eq:2d_pnt_limited_state_Euler} preserve positive density and pressure for the Euler equations.
	\end{itemize}
\end{proposition}

\begin{remark}
		For uniform meshes, and if taking the maximal spectral radius of $\partial\bF_1/\partial\bU$ and $\partial\bF_2/\partial\bU$ in the domain as $\norm{\varrho_1}_\infty$ and $\norm{\varrho_2}_\infty$,
		the following CFL condition
		\begin{equation*}
			\Delta t^n \leqslant \frac14\min\left\{\dfrac{\Delta x}{\norm{\varrho_1}_\infty},
			\dfrac{\Delta y}{\norm{\varrho_2}_\infty}
			\right\}
		\end{equation*}
		fulfills the time step size constraints \eqref{eq:2d_convex_combination_dt} and \eqref{eq:2d_pnt_llf_dt}.
\end{remark}

\section{Numerical results}\label{sec:results}
This section presents some numerical tests to verify the accuracy, BP property, and shock-capturing ability of the proposed BP AF methods.
The adiabatic index is $\gamma=1.4$ for the Euler equations except for \cref{ex:2d_jet}, where it is $5/3$.
In the 2D plots, the numerical solutions are visualized on a refined mesh with half the mesh size, where the values at the grid points are the cell averages or point values on the original mesh.
Note that the BP limitings naturally reduce some oscillations.
Some additional tests are provided in Section \ref{sec:results_supp} in Appendix,
including a 1D accuracy test for the Euler equations, double rarefaction problem, blast wave interaction problem using the power law reconstruction \cite{Barsukow_2021_active_JoSC}, and double Mach reflection problem.


\begin{example}[Self-steepening shock]\label{ex:1d_burgers}
	Consider the 1D Burgers' equation
$u_t + \left(\frac12 u^2\right)_x = 0$
	on the domain $[-1,1]$ with periodic boundary conditions.
	The initial condition is a square wave, $u_0(x)=2$ if $\abs{x}<0.2$, otherwise $u_0(x)=-1$.

\Cref{fig:1d_burgers_shock} shows the cell averages and point values at $T=0.5$ based on different point value updates with $200$ cells.
The CFL number is $0.2$.
The spike generation when using the JS has been observed in \cite{Helzel_2019_New_JoSC},
and the reason is also discussed in \Cref{sec:1d_fvs}.
Such an issue cannot be eliminated by our BP limitings alone,
but can be cured by additionally using the FVS for the point value update.
The numerical solutions based on the FVS agree well with the reference solution when the limitings are activated.

\begin{figure}[htb]
	\centering
	\begin{subfigure}[b]{0.31\textwidth}
		\centering
		\includegraphics[width=1.0\linewidth]{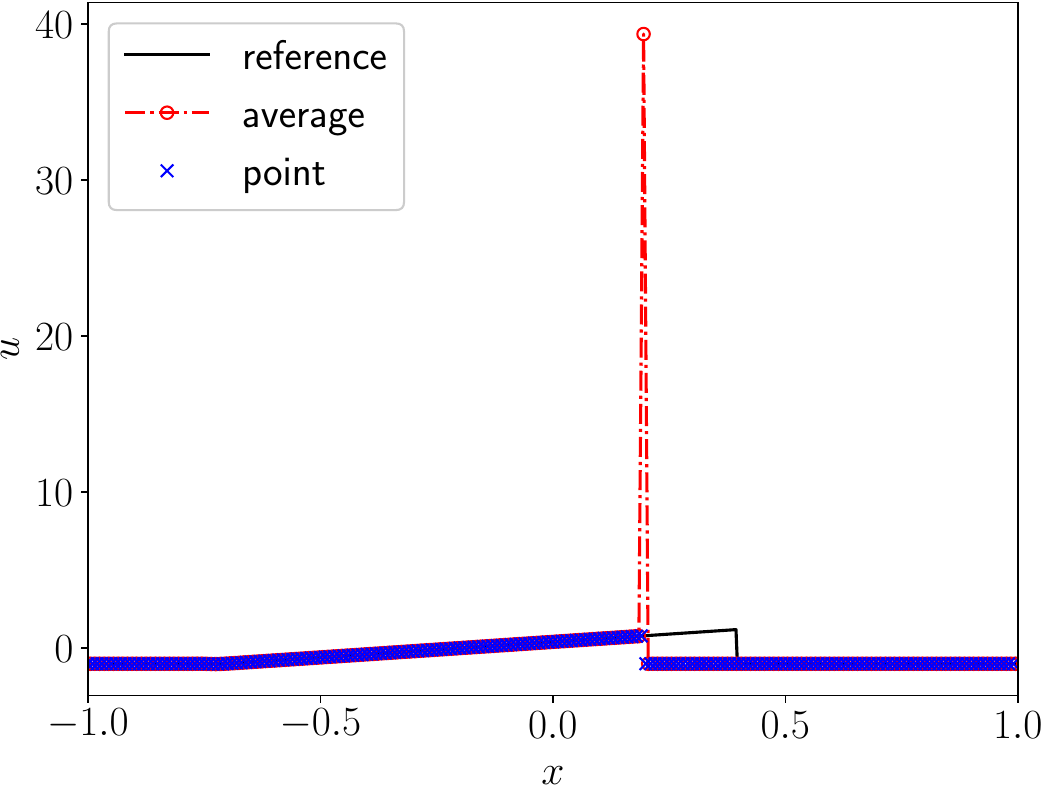}
	\end{subfigure}
	\begin{subfigure}[b]{0.32\textwidth}
		\centering
		\includegraphics[width=0.98\linewidth]{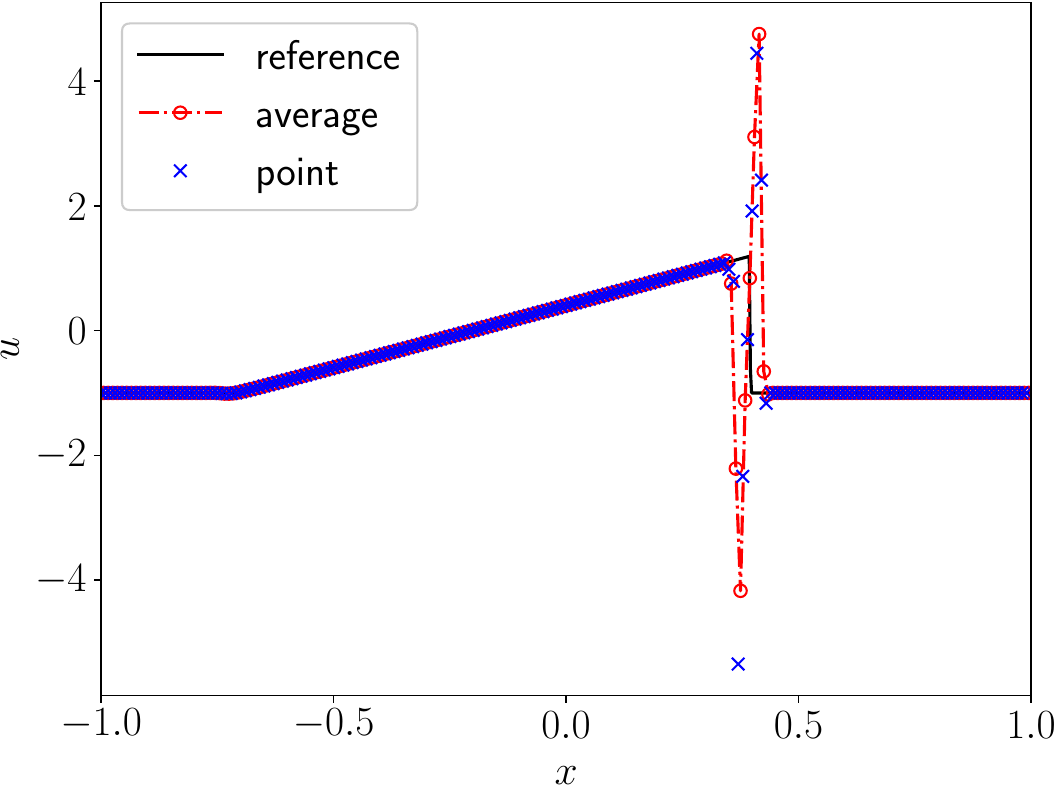}
	\end{subfigure}
	\begin{subfigure}[b]{0.32\textwidth}
		\centering
		\includegraphics[width=0.98\linewidth]{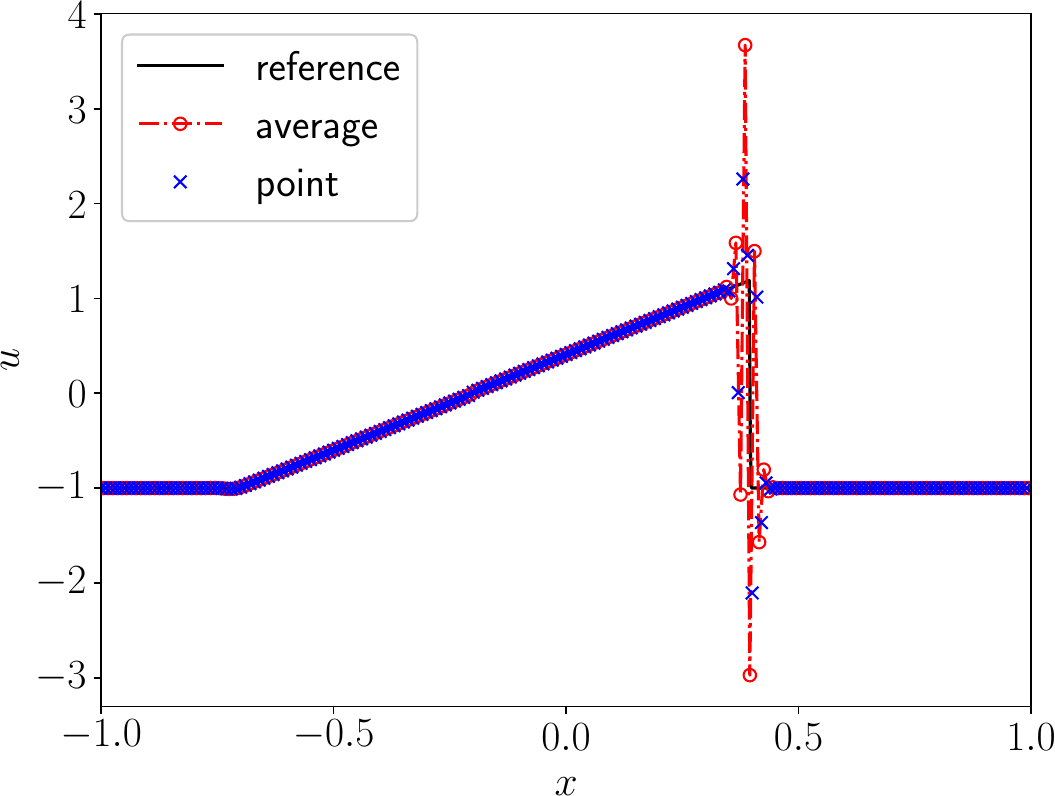}
	\end{subfigure}
	
	\begin{subfigure}[b]{0.32\textwidth}
		\centering
		\includegraphics[width=1.0\linewidth]{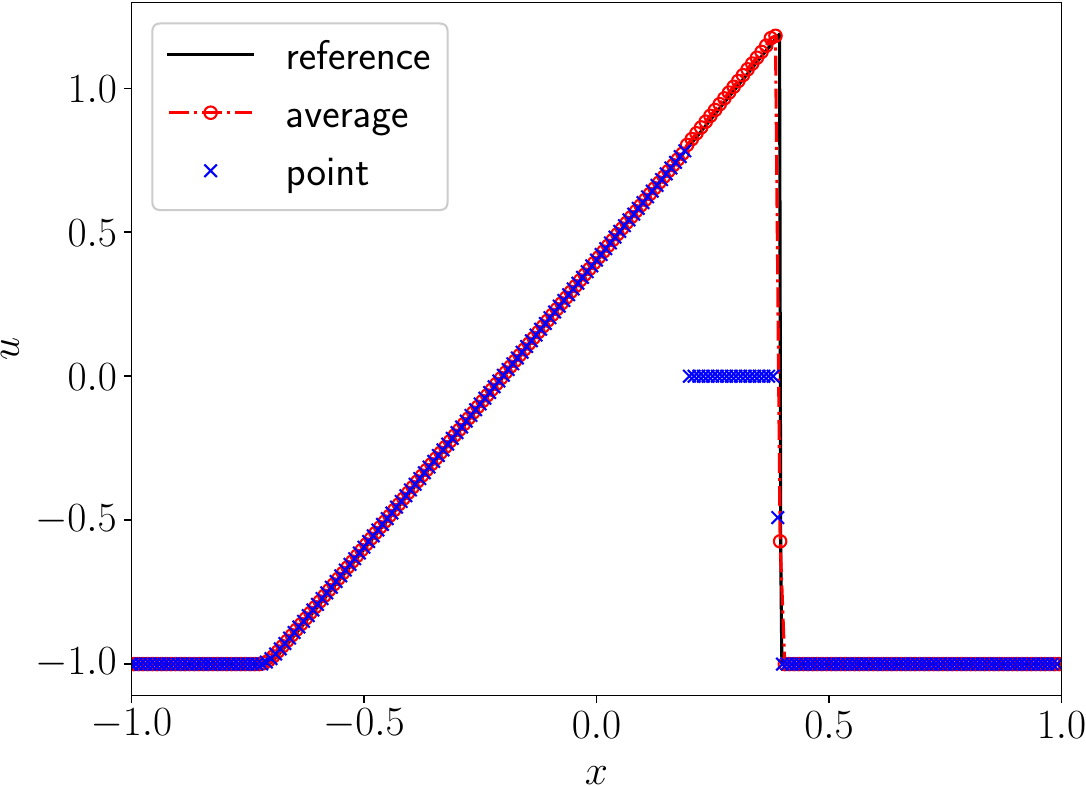}
	\end{subfigure}
	\begin{subfigure}[b]{0.32\textwidth}
		\centering
		\includegraphics[width=1.0\linewidth]{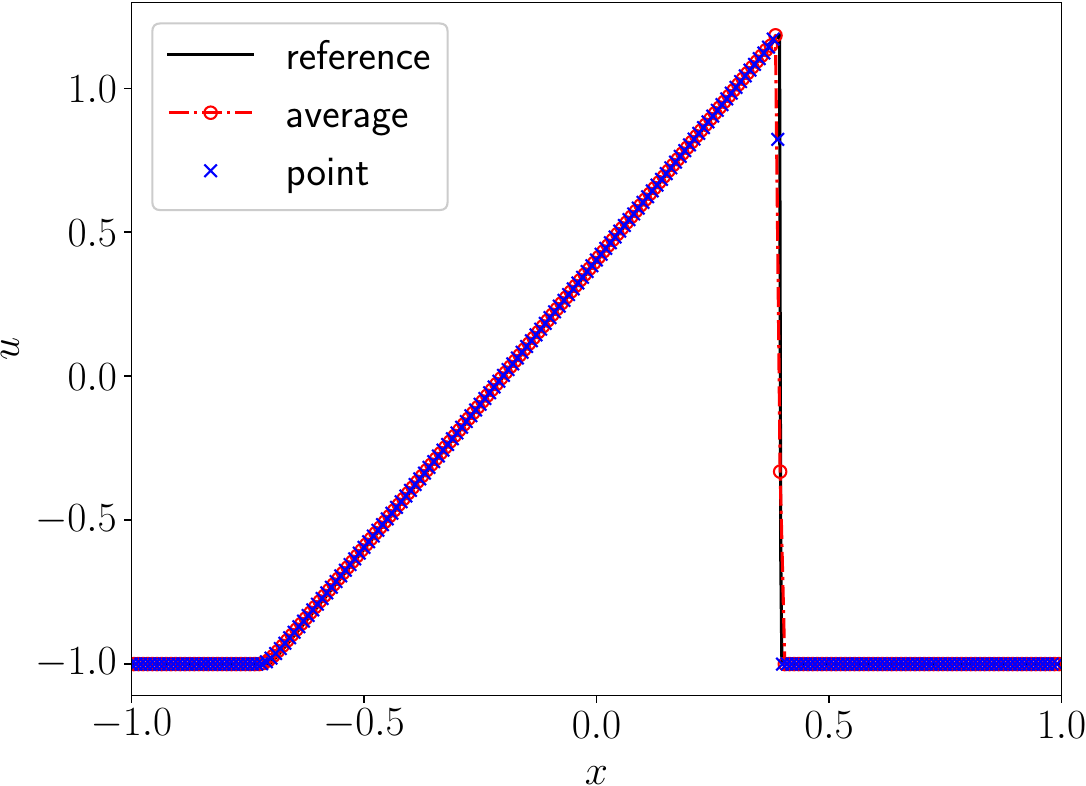}
	\end{subfigure}
	\begin{subfigure}[b]{0.32\textwidth}
		\centering
		\includegraphics[width=1.0\linewidth]{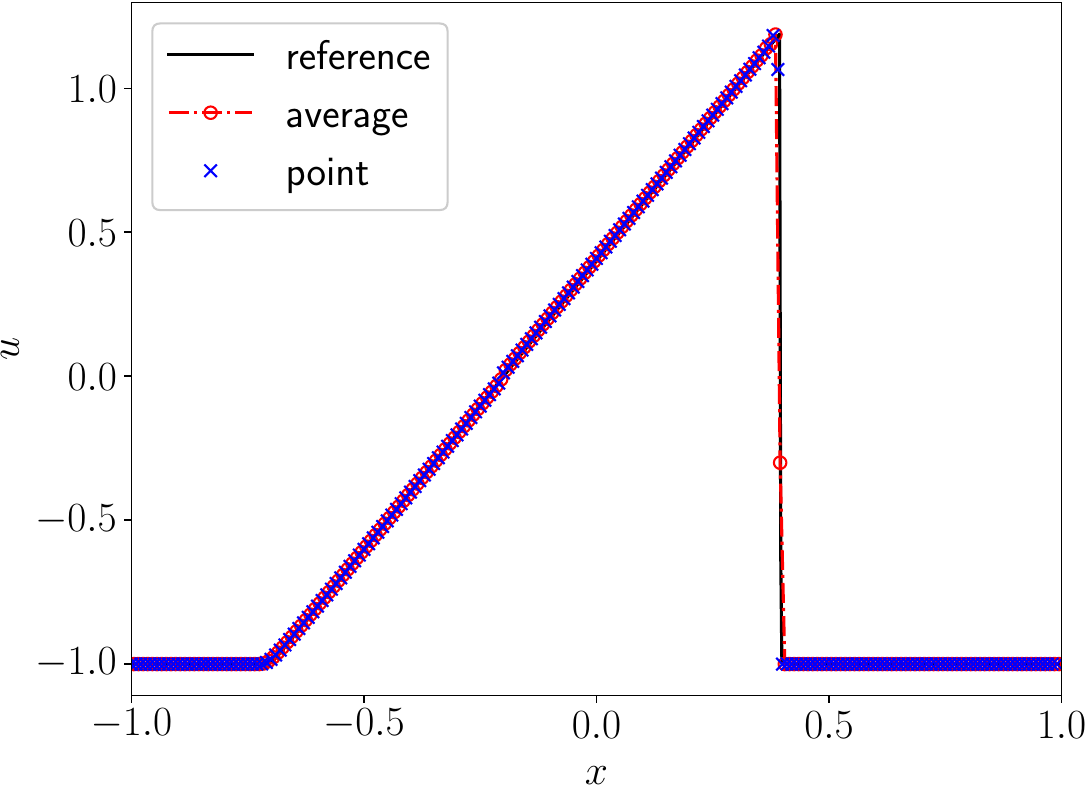}
	\end{subfigure}
	\caption{\Cref{ex:1d_burgers}, self-steepening shock for the Burgers' equation.
	The numerical solutions computed without limiting (top row) and with the BP limitings imposing the local MP (bottom row). From left to right: JS, LLF, and upwind FVS.
	}
	\label{fig:1d_burgers_shock}
\end{figure}
\end{example}

\begin{example}[LeBlanc shock tube]\label{ex:1d_leblanc}
	This is a Riemann problem with an extremely large initial pressure ratio.
	This test is solved until $T=5\times 10^{-6}$ on the domain $[0,1]$ with the initial data $(\rho, v, p)=(2, 0, 10^9)$ if $x<0.5$, otherwise $(\rho, v, p)=(10^{-3}, 0, 1)$.

Without the BP limitings, the simulation will stop due to negative density or pressure.
\Cref{fig:1d_leblanc_fine_mesh_shock_sensor} shows the density computed on a uniform mesh of $6000$ cells with the BP limitings and shock sensor-based limiting.
Note that, the numerical methods typically need a small mesh size to accurately obtain the right location of the shock wave.
The CFL number is $0.4$ for the JS, LLF, and SW FVS,
and $0.1$ for the VH FVS for stability.
The numerical solutions agree well with the exact solution with only a few undershoots at the discontinuities.

\begin{figure}[htb]
	\centering
	\begin{subfigure}[b]{0.24\textwidth}
		\centering
		\includegraphics[width=\linewidth]{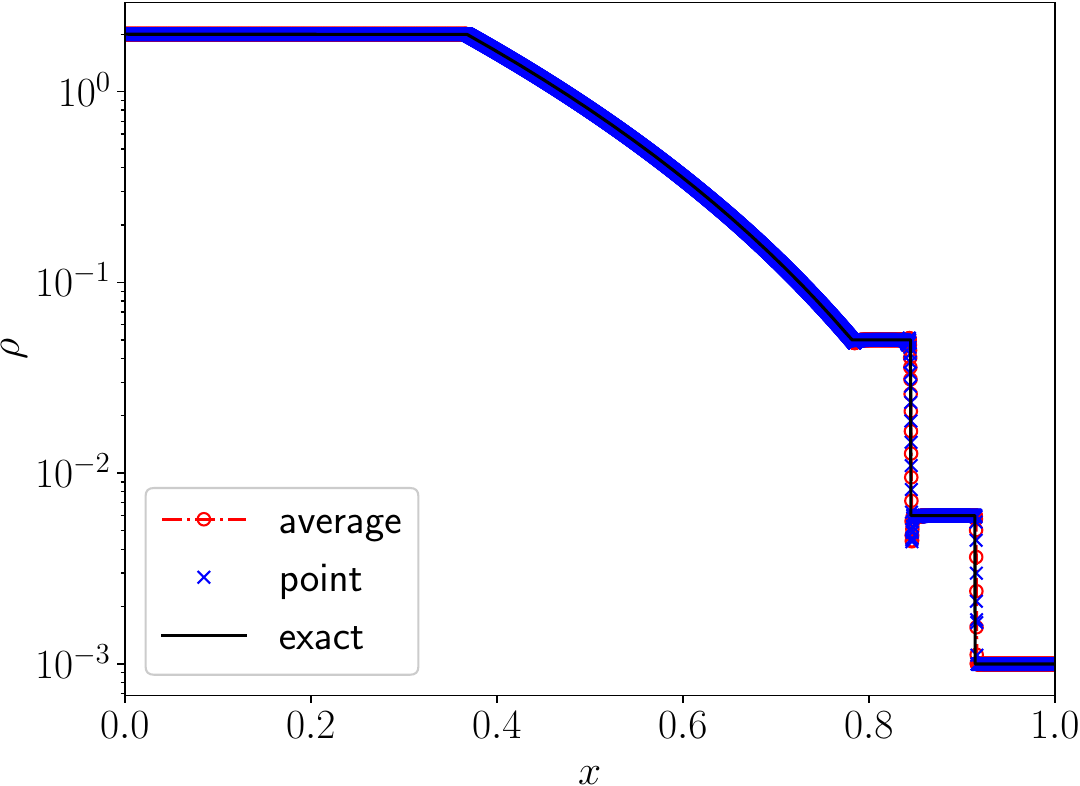}
	\end{subfigure}
	\begin{subfigure}[b]{0.24\textwidth}
		\centering
		\includegraphics[width=\linewidth]{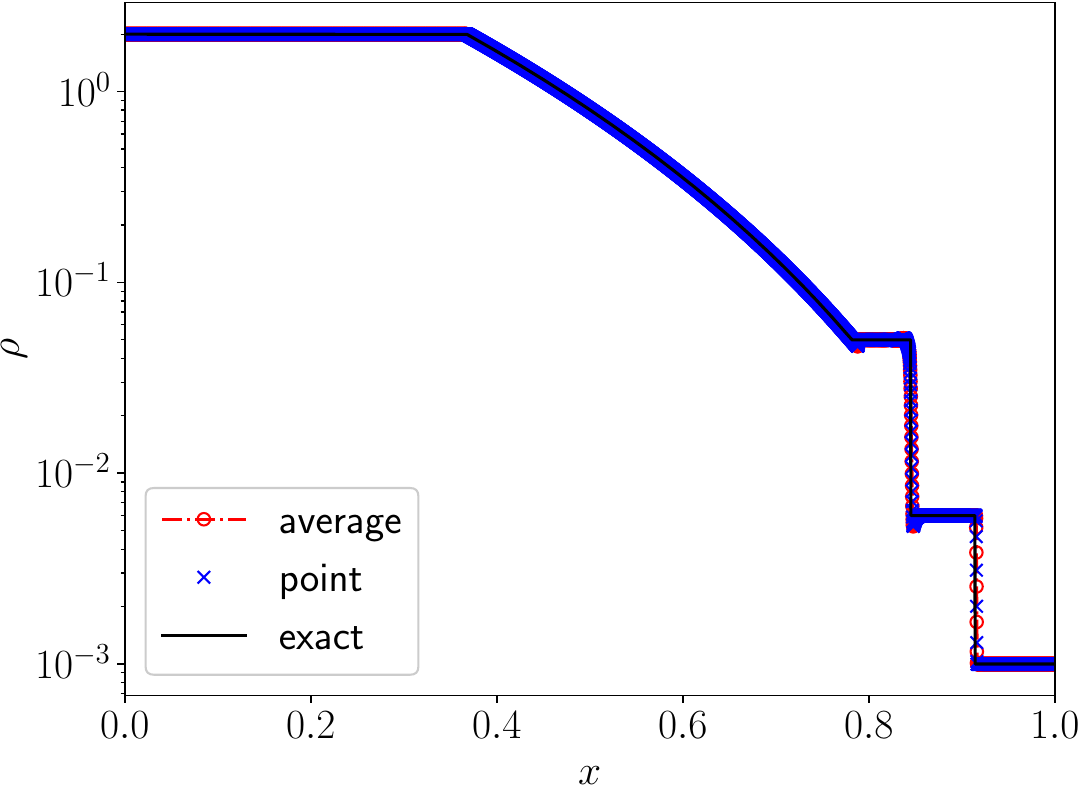}
	\end{subfigure}
	\begin{subfigure}[b]{0.24\textwidth}
		\centering
		\includegraphics[width=\linewidth]{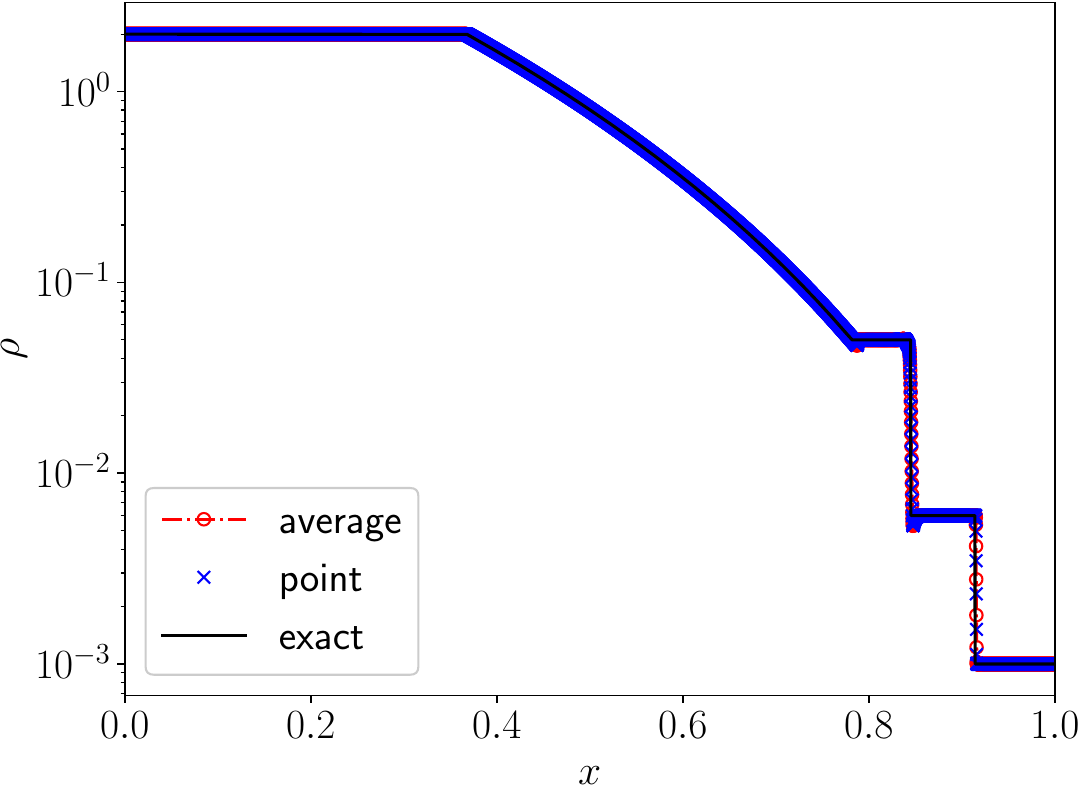}
	\end{subfigure}
	\begin{subfigure}[b]{0.24\textwidth}
		\centering
		\includegraphics[width=\linewidth]{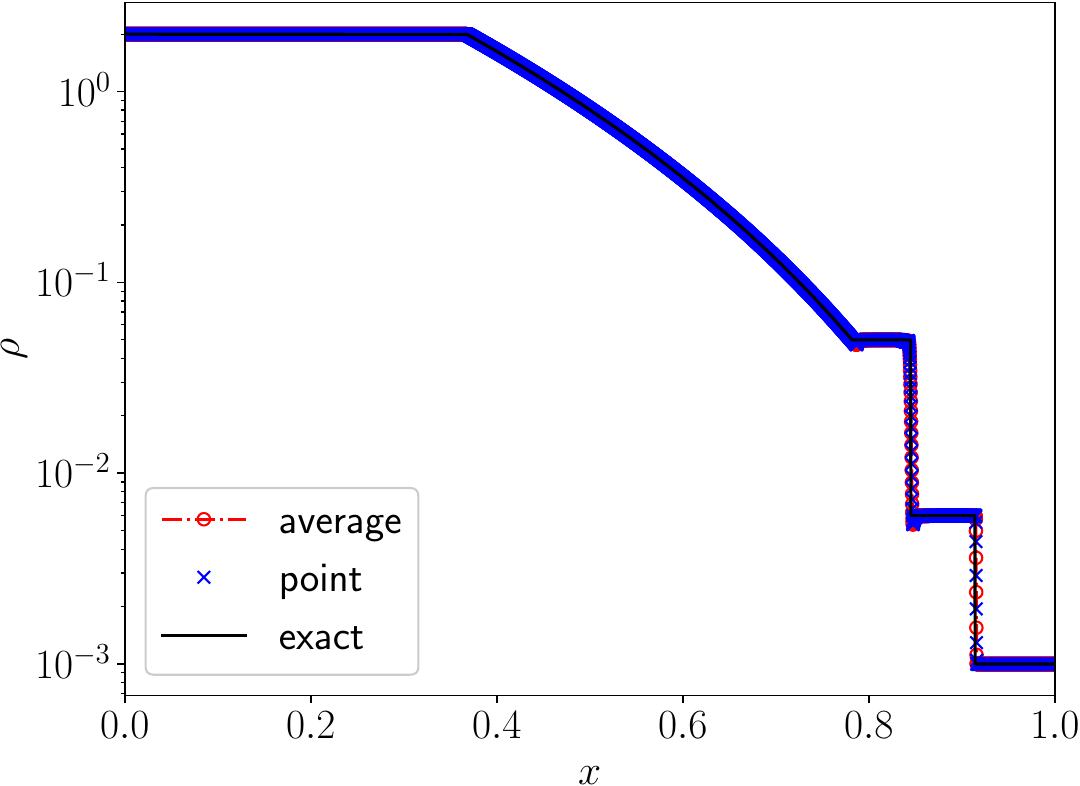}
	\end{subfigure}
	\caption{\Cref{ex:1d_leblanc}, LeBlanc Riemann problem.
		The density computed with the BP limitings and the shock sensor-based limiting ($\kappa=10$) on a uniform mesh of $6000$ cells.
		From left to right: JS, LLF, SW, and VH FVS.
	}
	\label{fig:1d_leblanc_fine_mesh_shock_sensor}
\end{figure}
\end{example}

\begin{example}[Blast wave interaction \cite{Woodward_1984_numerical_JoCP}]\label{ex:1d_blast_wave}
	This test describes the interaction of two strong shocks in the domain $[0,1]$ with reflective boundary conditions.
	The test is solved until $T=0.038$.

Due to the low-pressure region, the schemes blow up without the BP limitings.
\Cref{fig:1d_blast_wave_coarse_mesh_kappa=1} shows the density plots obtained by using the BP limitings and shock sensor-based limiting on a uniform mesh of $800$ cells.
The CFL number is $0.4$ for the JS, LLF, and SW FVS, and $0.36$ for the VH FVS.
The numerical solutions agree well the reference solution with a few overshoots/undershoots.

\begin{figure}[htb]
	\centering
	\begin{subfigure}[b]{0.24\textwidth}
		\centering
		\includegraphics[width=1.0\linewidth]{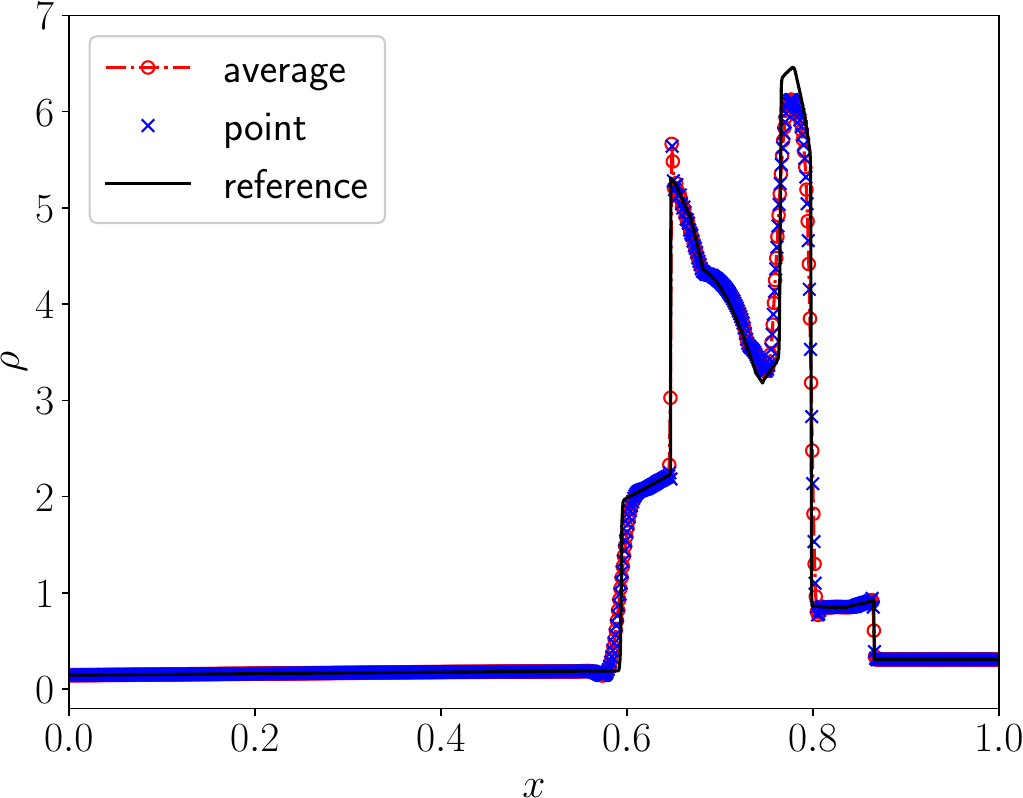}
	\end{subfigure}
	\begin{subfigure}[b]{0.24\textwidth}
		\centering
		\includegraphics[width=1.0\linewidth]{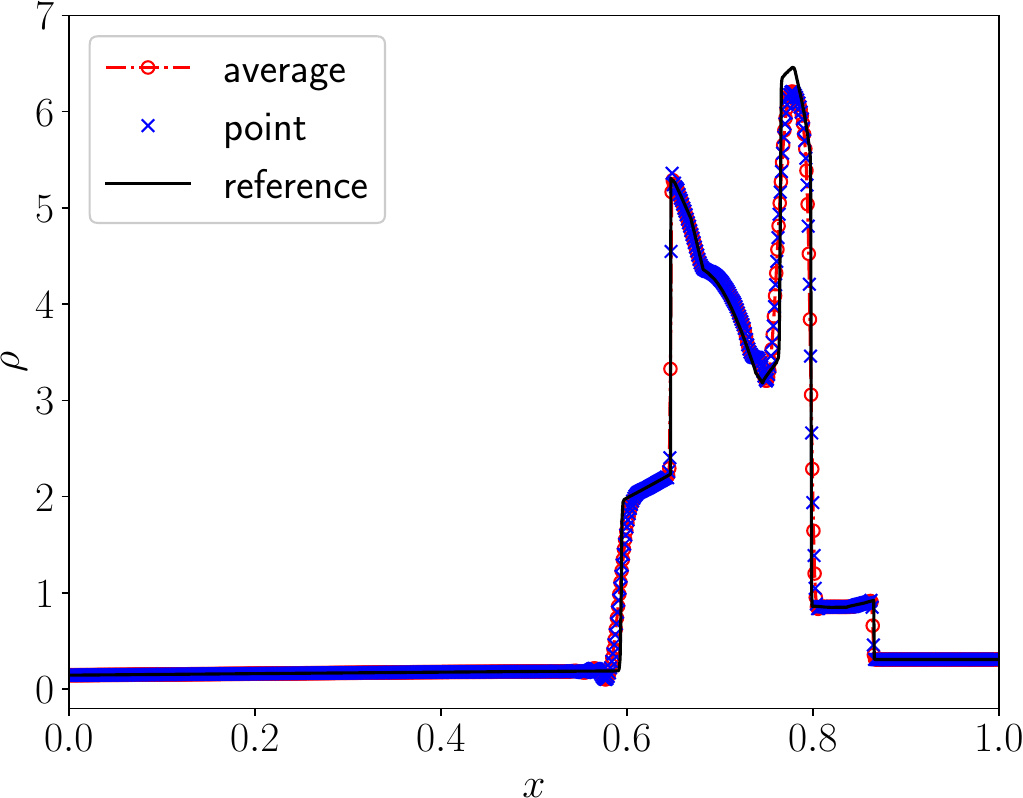}
	\end{subfigure}
	\begin{subfigure}[b]{0.24\textwidth}
		\centering
		\includegraphics[width=1.0\linewidth]{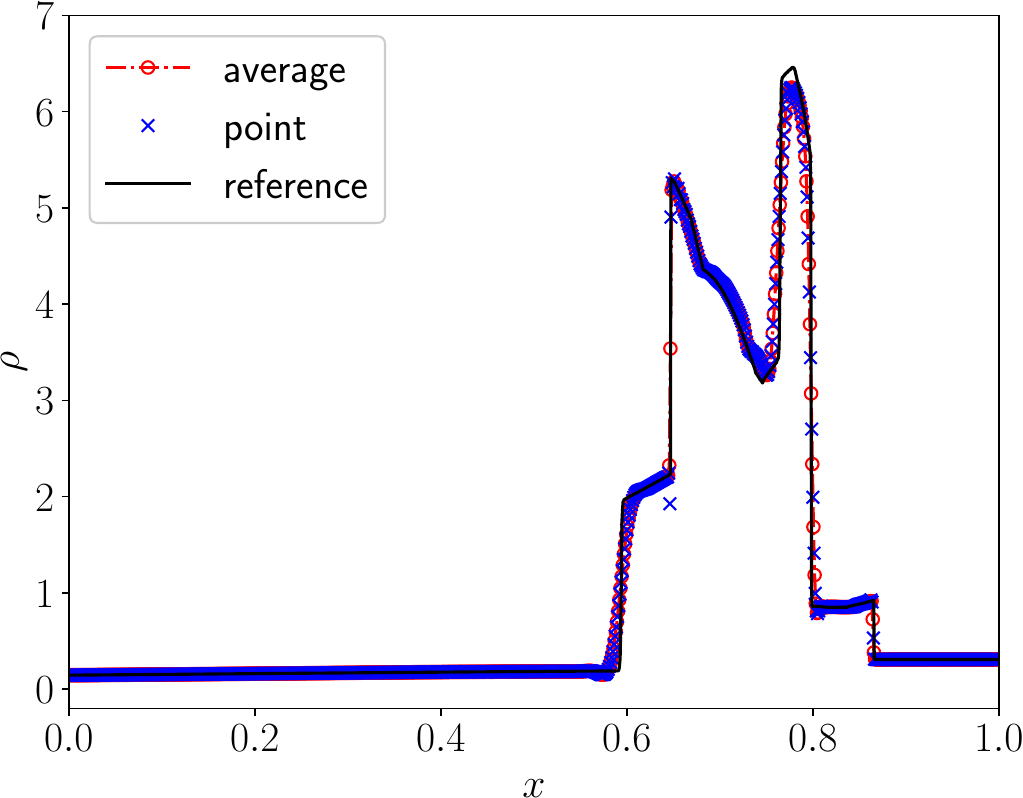}
	\end{subfigure}
	\begin{subfigure}[b]{0.24\textwidth}
		\centering
		\includegraphics[width=1.0\linewidth]{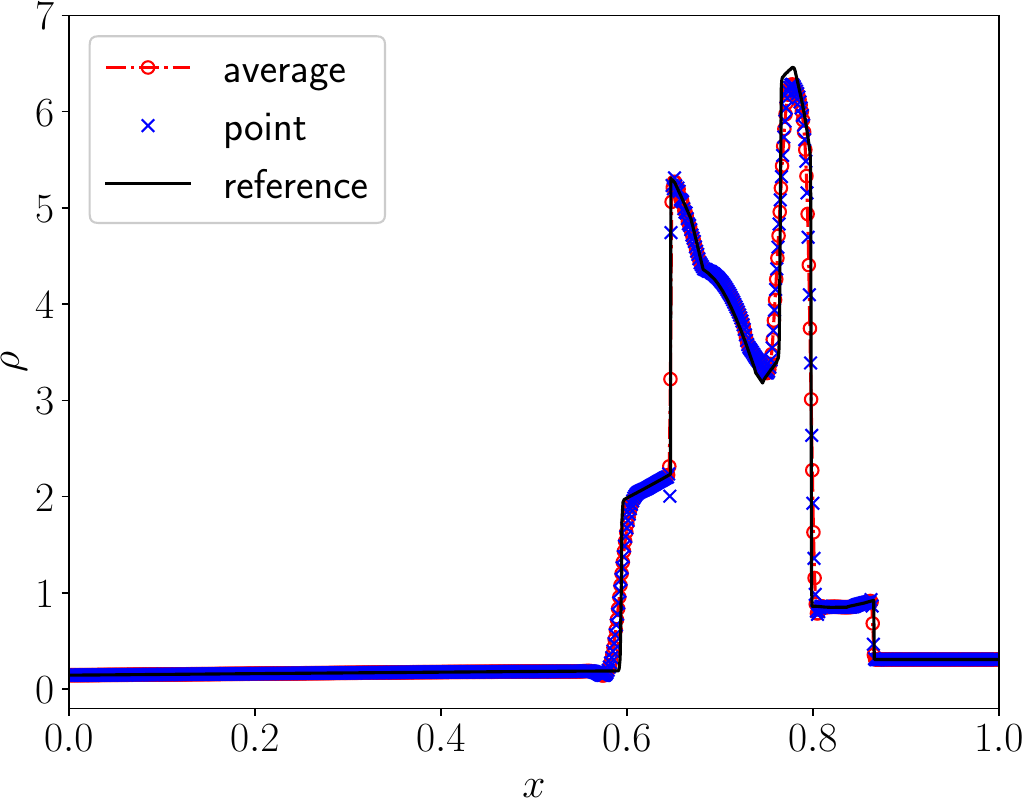}
	\end{subfigure}
	
	\begin{subfigure}[b]{0.24\textwidth}
		\centering
		\includegraphics[width=1.0\linewidth]{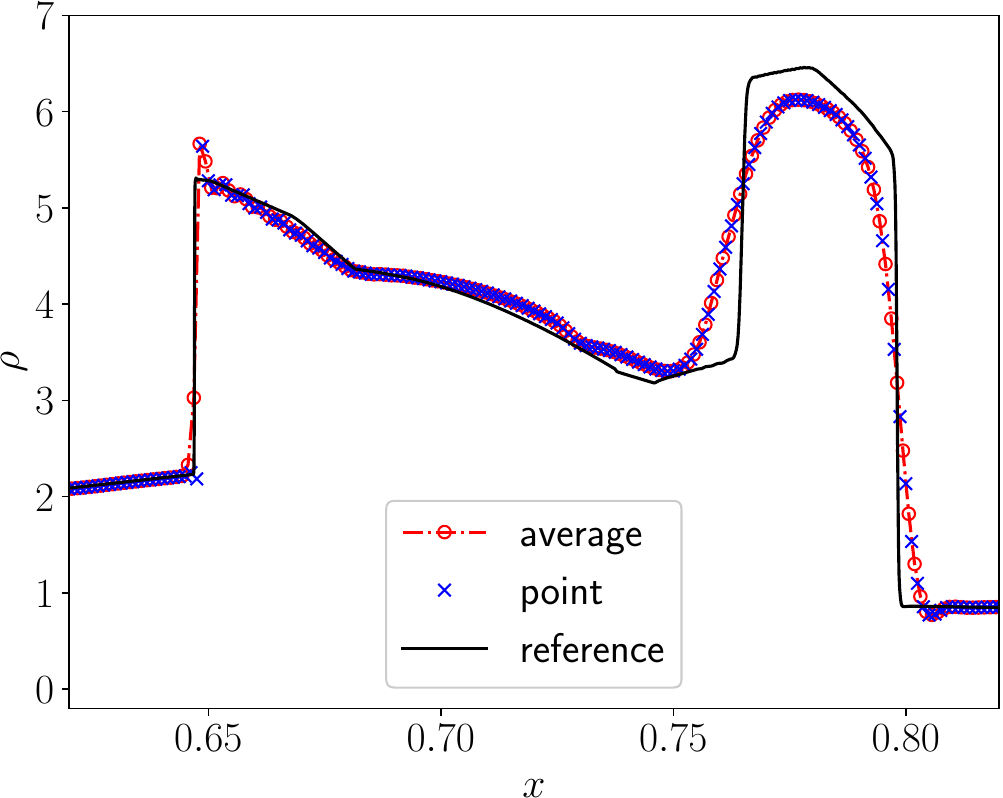}
	\end{subfigure}
	\begin{subfigure}[b]{0.24\textwidth}
		\centering
		\includegraphics[width=1.0\linewidth]{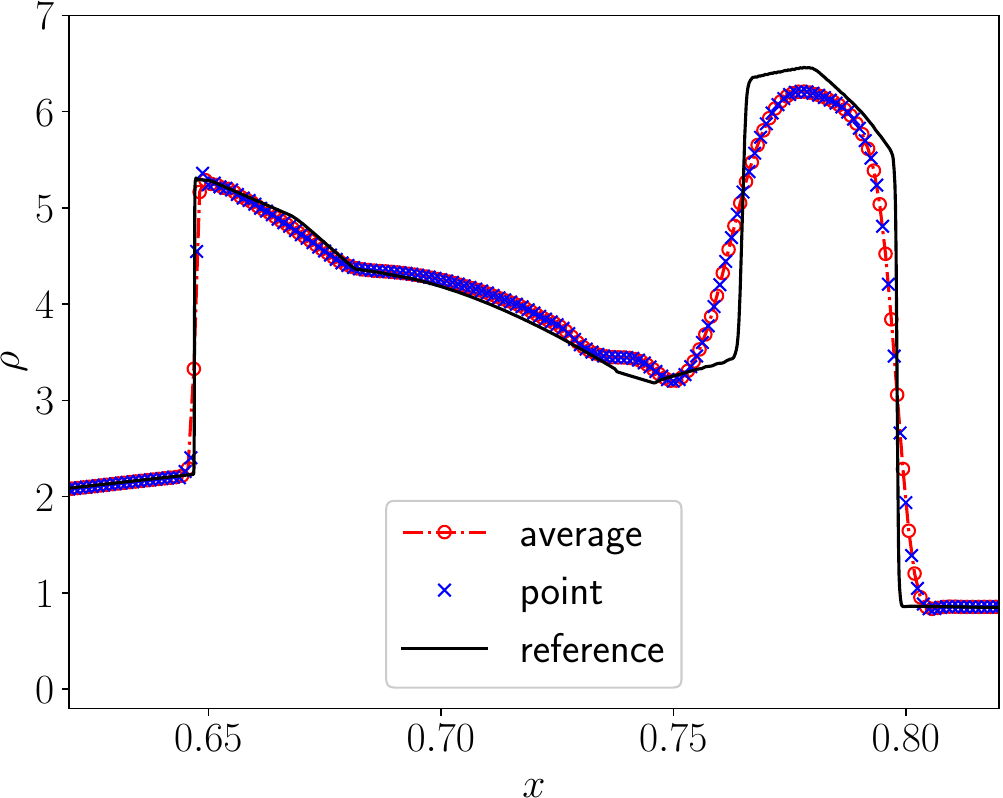}
	\end{subfigure}
	\begin{subfigure}[b]{0.24\textwidth}
		\centering
		\includegraphics[width=1.0\linewidth]{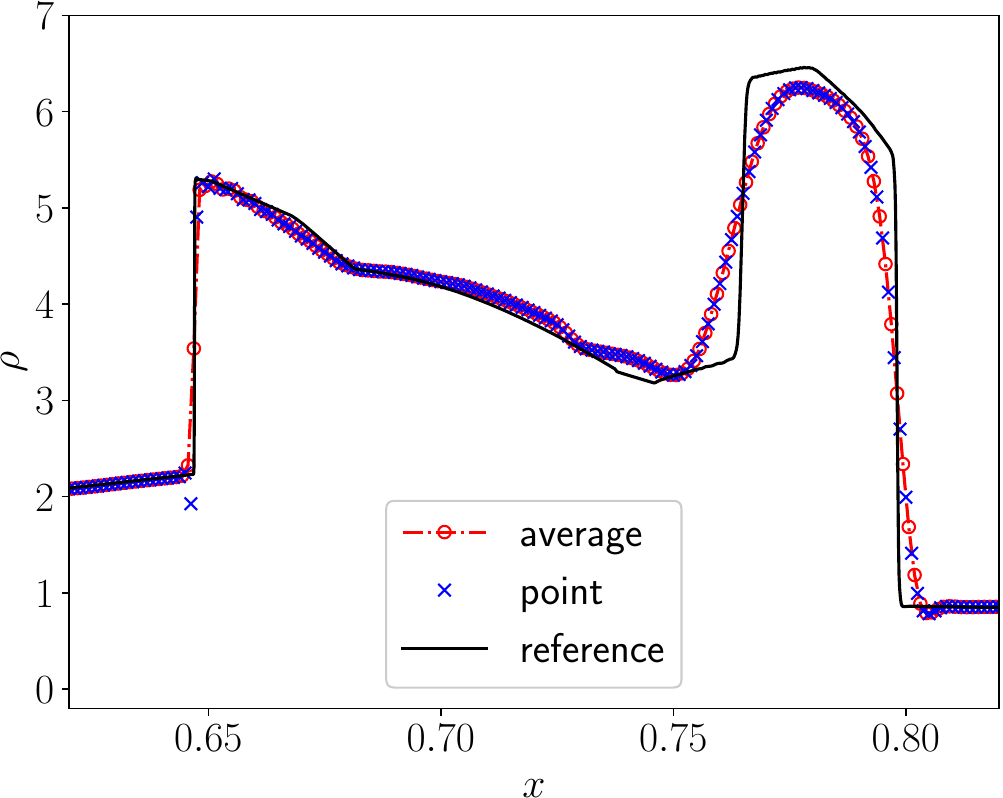}
	\end{subfigure}
	\begin{subfigure}[b]{0.24\textwidth}
		\centering
		\includegraphics[width=1.0\linewidth]{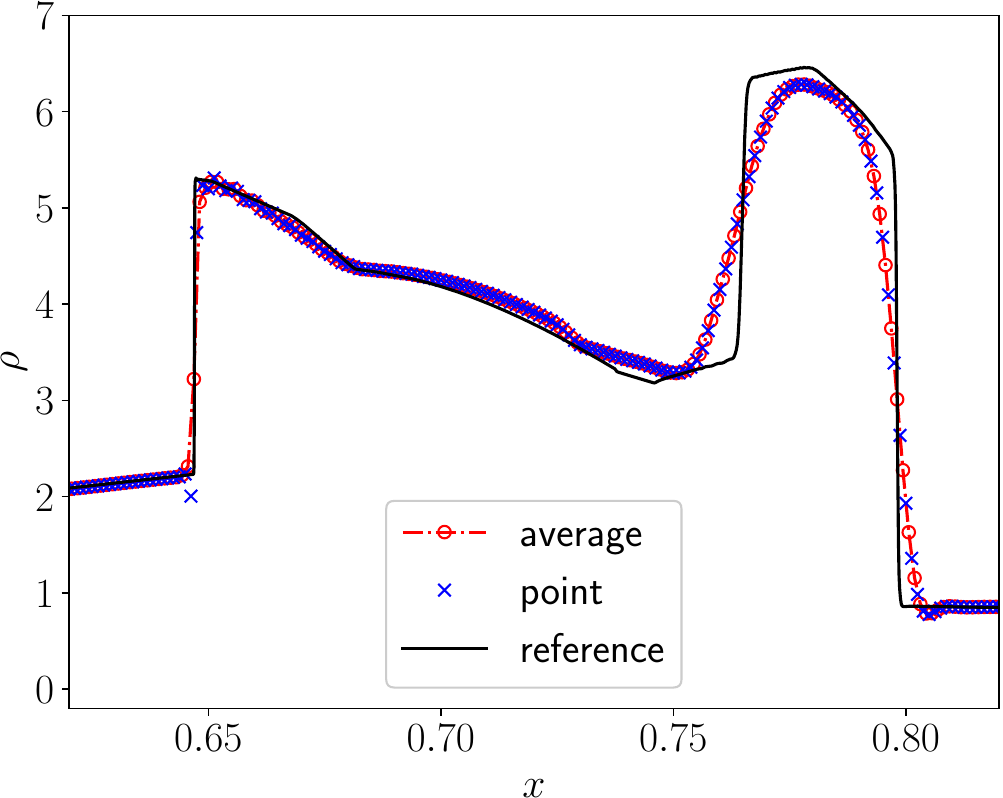}
	\end{subfigure}
	\caption{\Cref{ex:1d_blast_wave}, blast wave interaction.
		The density computed with the BP limitings and the shock sensor-based limiting ($\kappa=1$).
		The corresponding enlarged views in $x\in[0.62, 0.82]$ are shown in the bottom row.}
	\label{fig:1d_blast_wave_coarse_mesh_kappa=1}
\end{figure}
\end{example}

\begin{remark}
   In the numerical tests, the maximal CFL numbers for stability are obtained experimentally.
   Note that the reduction of the CFL numbers is due to different stability bounds for different point value updates, and is not related to the BP property.
   The study of such an issue is beyond the scope of this paper, which will be explored in the future.
\end{remark}


\begin{example}[2D advection equation]\label{ex:2d_advection_discontinuity}
	This test solves $u_t + u_x + u_y = 0$,
	on the periodic domain $[0,1]\times[0,1]$ with the following initial data
	\begin{equation*}
		u_0(x,y) = \begin{cases}
			1-\abs{5r}, &\text{if}\quad r=\sqrt{(x-0.3)^2+(y-0.3)^2}<0.2, \\
			1, &\text{if}\quad \max\{\abs{x-0.7}, ~\abs{y-0.7}\} < 0.2, \\
			0, &\text{otherwise}. \\
		\end{cases}
	\end{equation*}

For the advection equation, the JS and LLF FVS are equivalent.
The results on the uniform $100\times 100$ mesh obtained without and with BP limitings at $T=2$ are presented in \Cref{fig:2d_advection_discontinuity_surface}.
The BP limitings suppress the overshoots and undershoots well near the discontinuities.
\Cref{tab:2d_advection_discontinuity_bounds} lists whether the numerical solutions stay in the bound $[0,1]$.
The bound is only preserved when both the BP limitings for the cell average and point value are activated,
demonstrating that it is necessary to use the two kinds of BP limitings simultaneously.

\begin{figure}[htb]
	\centering
	\begin{subfigure}[b]{0.31\textwidth}
		\centering
		\includegraphics[width=1.0\linewidth]{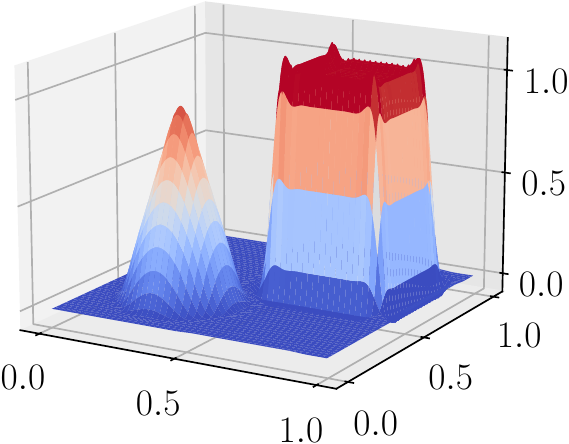}
	\end{subfigure}
	\begin{subfigure}[b]{0.31\textwidth}
		\centering
		\includegraphics[width=1.0\linewidth]{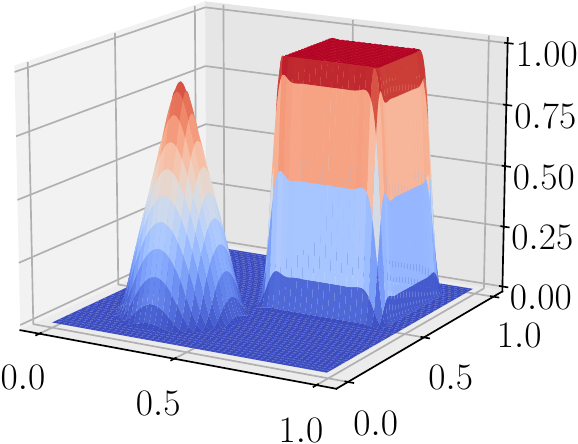}
	\end{subfigure}
        \begin{subfigure}[b]{0.35\textwidth}
		\centering
		\includegraphics[width=1.0\linewidth]{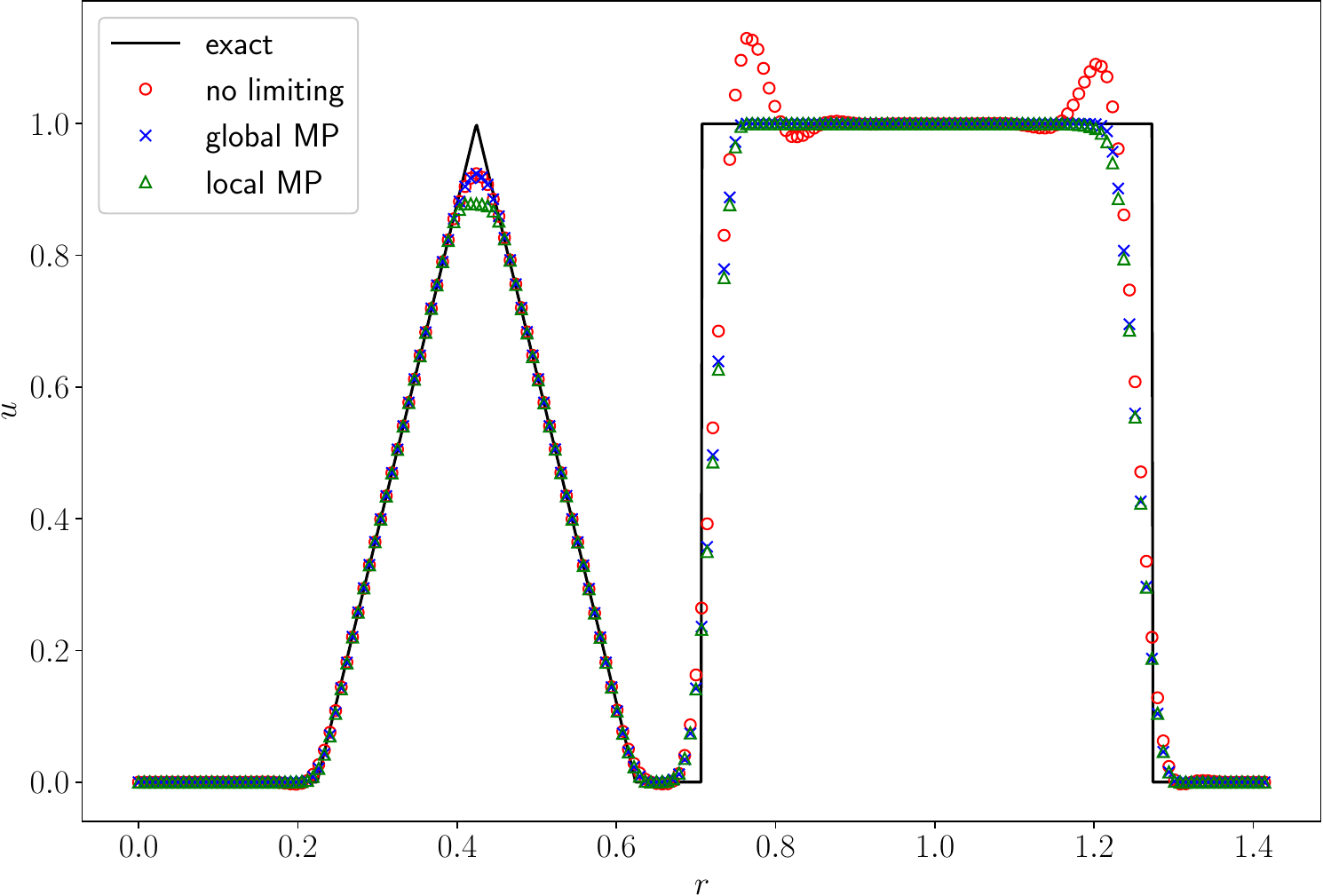}
	\end{subfigure}
	\caption{\Cref{ex:2d_advection_discontinuity}, 2D advection equation.
        From left to right: without any limiting, with BP limitings imposing the global MP, cut-line along $y=x$.}
	\label{fig:2d_advection_discontinuity_surface}
\end{figure}

\begin{table}[htb]
	\centering
	\begin{tabular}{c|c|c|c}
		\hline\hline
	\diagbox{cell average}{point value}  & no limiting & global MP & local MP \\ \hline
	no limiting   &  {\color{red}\xmark}   &  {\color{red}\xmark}   &      {\color{red}\xmark}     \\ \hline
	global MP   &  {\color{red}\xmark}   &   {\color{blue}\cmark}   &   {\color{blue}\cmark}     \\ \hline
	local MP   &  {\color{red}\xmark}   &    {\color{blue}\cmark}   &     {\color{blue}\cmark}     \\
		\hline\hline
	\end{tabular}
	\caption{\Cref{ex:2d_advection_discontinuity}, 2D advection equation.
	We list whether the numerical solutions stay in the bound $[0,1]$ with different limitings.}
	\label{tab:2d_advection_discontinuity_bounds}
\end{table}
\end{example}

\begin{example}[2D Burgers' equation]\label{ex:2d_burgers}
	We solve $u_t + \left(\frac12 u^2\right)_x + \left(\frac12 u^2\right)_y = 0$
	on the periodic domain $[0,1]\times[0,1]$,
	with the initial condition $u_0(x, y) = 0.5 + \sin(2\pi(x+y))$.
	This test is solved until $T=0.3$, when the shock waves have emerged.

\Cref{fig:2d_burgers_shock_surface} plots the solutions using the LLF FVS on the uniform $100\times100$ mesh without and with limitings.
The oscillations near the shock waves are suppressed well when the limitings are activated, and the numerical solutions agree well with the reference solution.
The blending coefficients $\theta_{\xr,j}, \theta_{i,\yr}$ for the cell average and $\theta_{\sigma}$ for the point value when using the global MP are also presented in \Cref{fig:2d_burgers_shock_limiting}, verifying that the limitings are only locally activated near the shock waves.
 
\begin{figure}[htb]
	\centering
	\begin{subfigure}[b]{0.31\textwidth}
		\centering
		\includegraphics[width=\linewidth]{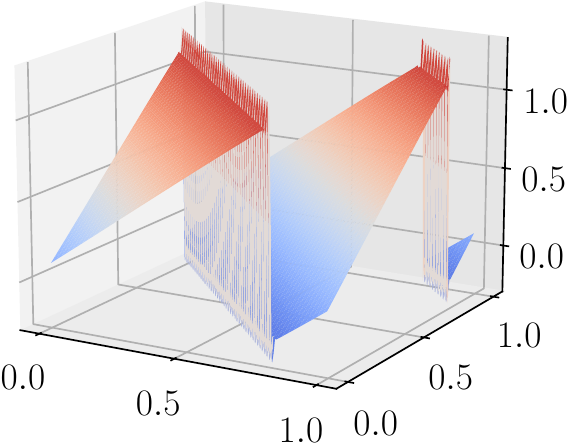}
	\end{subfigure}
	\begin{subfigure}[b]{0.31\textwidth}
		\centering
		\includegraphics[width=\linewidth]{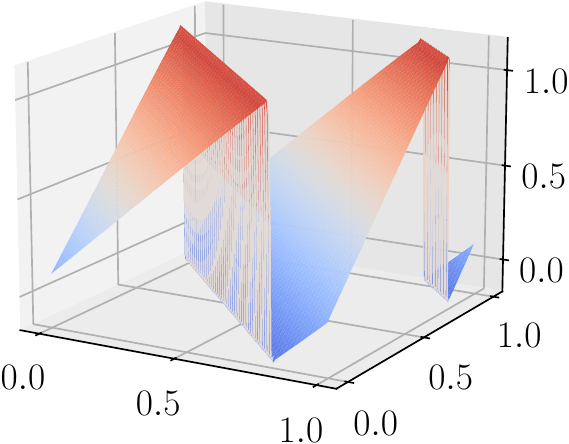}
	\end{subfigure}
        \begin{subfigure}[b]{0.32\textwidth}
		\centering
		\includegraphics[width=\linewidth]{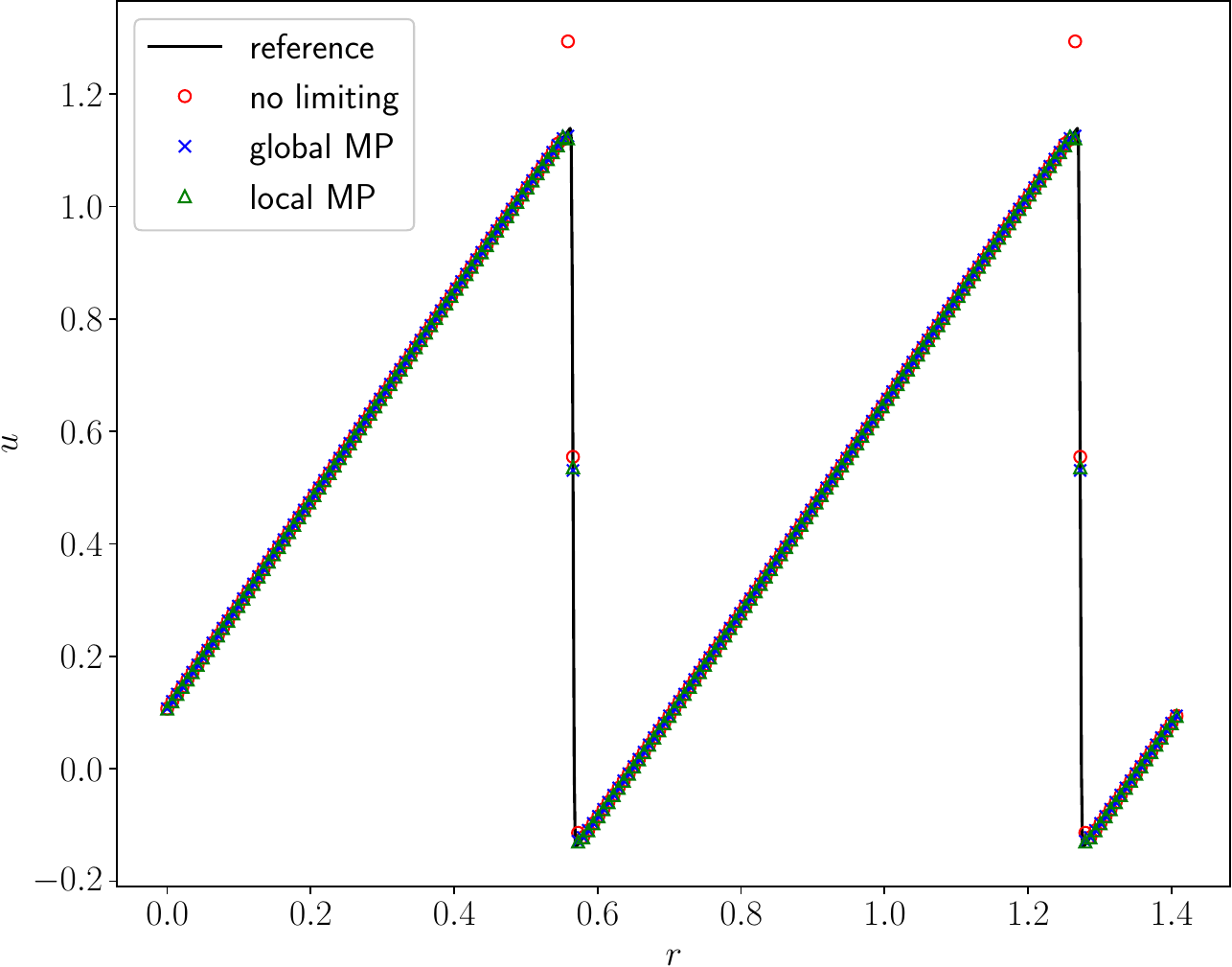}
	\end{subfigure}
	\caption{\Cref{ex:2d_burgers}, 2D Burgers' equation.
 	From left to right: without limiting, with BP limitings imposing the global MP, cut-line along $y=x$.}
	\label{fig:2d_burgers_shock_surface}
\end{figure}

\begin{figure}[htb]
	\centering
	\begin{subfigure}[b]{0.32\textwidth}
		\centering
		\includegraphics[width=\linewidth]{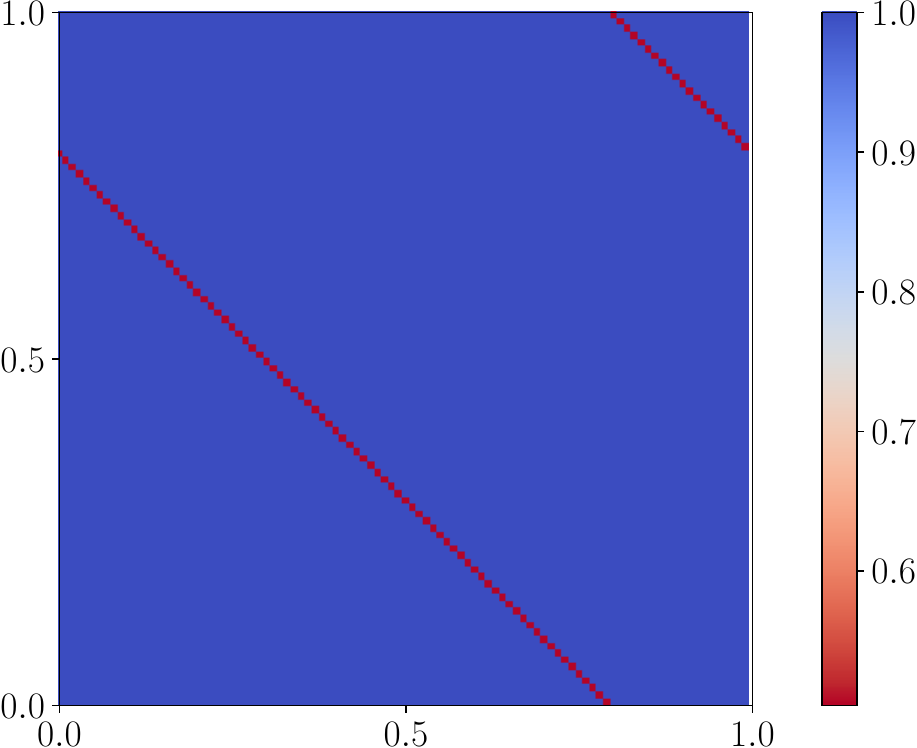}
	\end{subfigure}
	\begin{subfigure}[b]{0.32\textwidth}
		\centering
		\includegraphics[width=\linewidth]{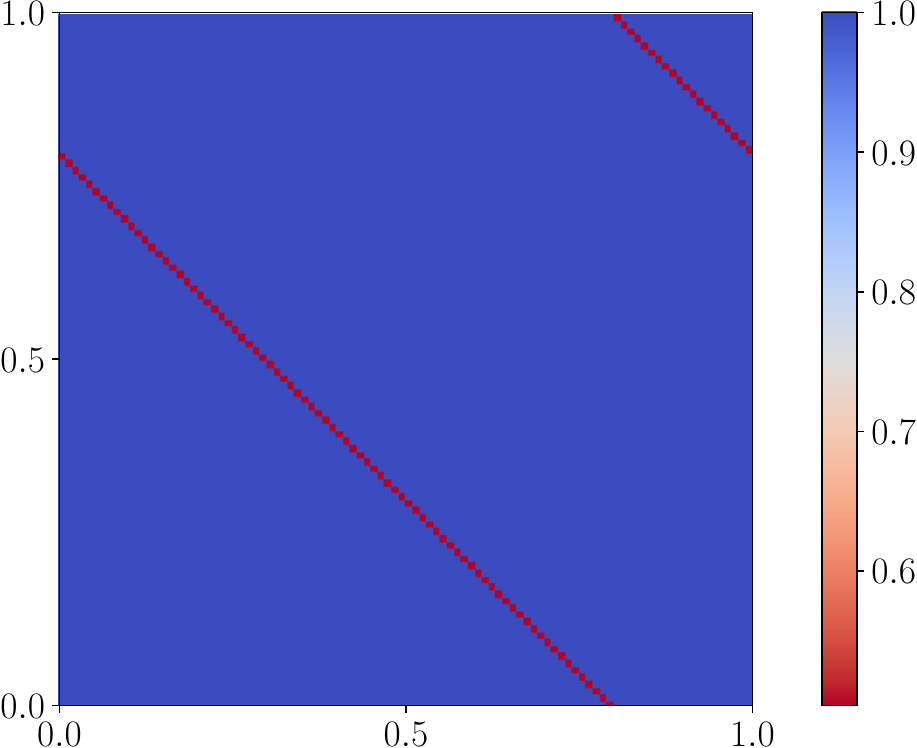}
	\end{subfigure}
	\begin{subfigure}[b]{0.32\textwidth}
		\centering
		\includegraphics[width=\linewidth]{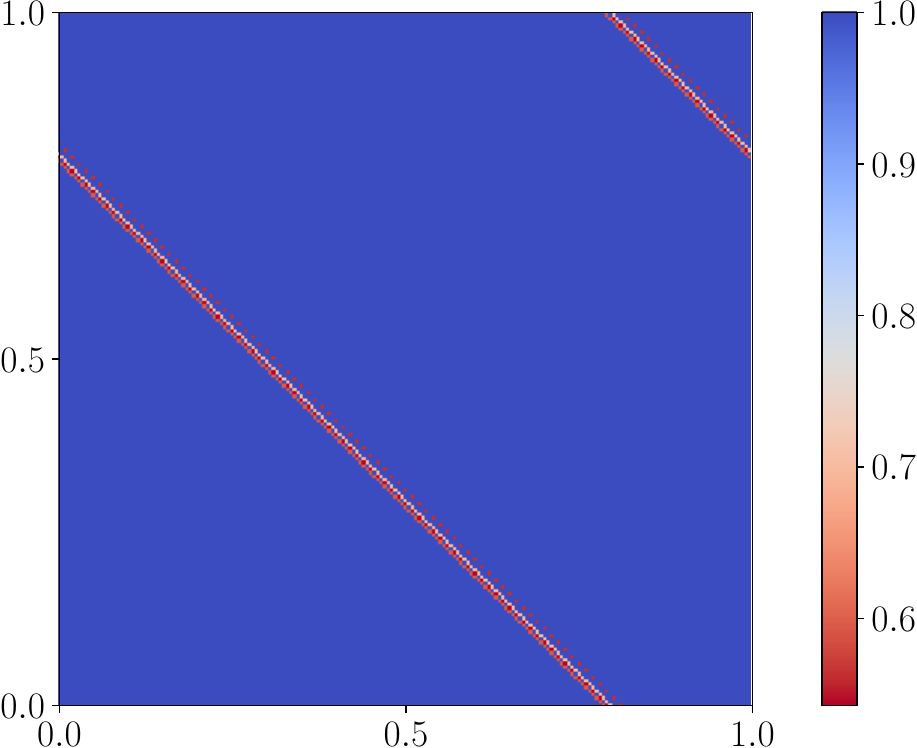}
	\end{subfigure}
	\caption{\Cref{ex:2d_burgers}, 2D Burgers' equation.
	The blending coefficients in the limitings.
	From left to right: $\theta_{\xr,j}$ and $\theta_{i,\yr}$ for the cell average, $\theta_{\sigma}$ for the point value.}
	\label{fig:2d_burgers_shock_limiting}
\end{figure}
\end{example}

\begin{example}[2D isentropic vortex]\label{ex:2d_vortex}
	The domain is $[-5,5]\times[-5,5]$ with periodic boundary conditions and the initial condition is
	\begin{equation*}
		\rho = T_0^{\frac{1}{\gamma-1}},~
		(v_1, v_2) = (1, 1) + k_0(y, -x),~
		p = T_0\rho,~
		k_0=\frac{\epsilon}{2\pi}e^{0.5(1-r^2)},~
		T_0 = 1 - \frac{\gamma-1}{2\gamma}k_0^2,
	\end{equation*}
	where $r^2 = x^2+y^2$, and $\epsilon=10.0828$ is the vortex strength.
	The lowest initial density and pressure are around $7.83\times 10^{-15}$ and $1.78 \times 10^{-20}$, respectively,
	so that the BP limitings are necessary to run this test case.
	The problem is solved until $T = 1$.

Figure \ref{fig:2d_vortex_accuracy} shows the errors and corresponding convergence rates of the conservative variables in the $\ell^1$ norm with the CFL number $0.2$.
The BP AF methods based on the JS, LLF, and VH FVS achieve the third-order accuracy,
which is not affected by the BP limitings.
The convergence rate based on the SW FVS reduces to around $2$,
due to the non-smoothness of the SW FVS as mentioned in \Cref{rmk:sw_fvs}.

\begin{figure}[htb]
	\centering
	\begin{subfigure}[b]{0.5\textwidth}
		\centering
		\includegraphics[width=1.0\linewidth]{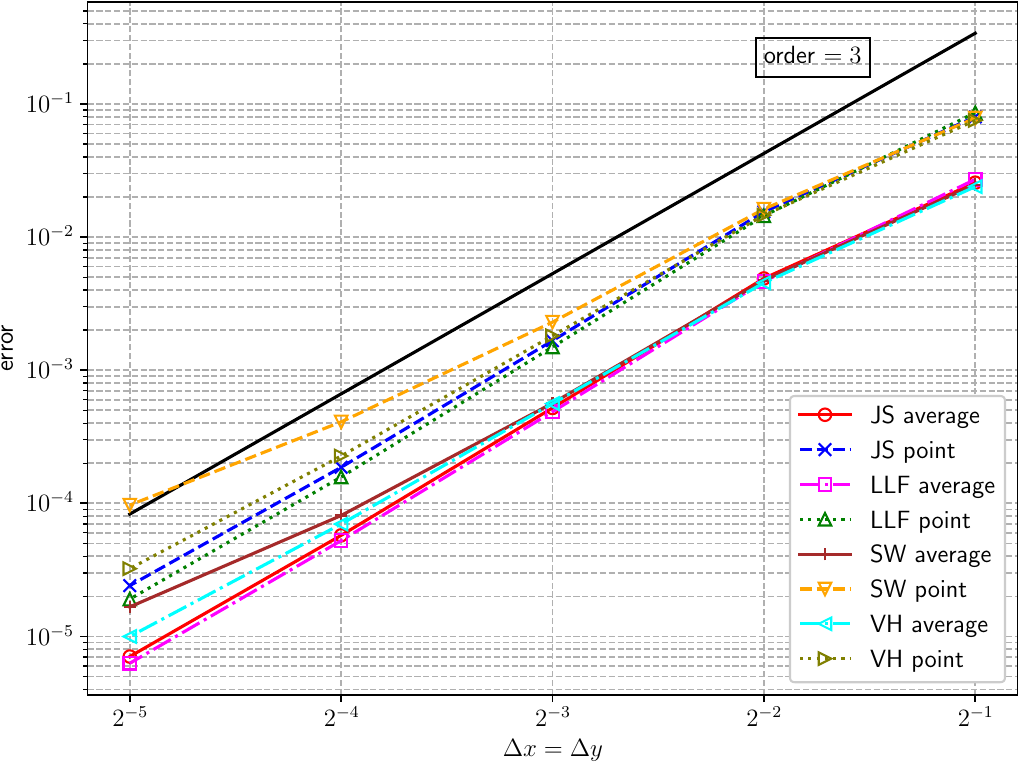}
	\end{subfigure}
	\caption{\Cref{ex:2d_vortex}, 2D isentropic vortex problem.
		The errors and convergence rates.
	}
	\label{fig:2d_vortex_accuracy}
\end{figure}
\end{example}

\begin{example}[Quasi-2D Sod shock tube]\label{ex:2d_sod}
	This test solves the Sod shock tube problem along the $x$-direction on the domain $[0,1]\times[0,1]$ with a $100\times 2$ uniform mesh until $T=0.2$.
	The initial condition is $(\rho, v_1, v_2, p)=(1, 0, 0, 1)$ if $x < 0.5$, otherwise $(\rho, v_1, v_2, p)=(0.125, 0, 0, 0.1)$.
	
	The density plots obtained by using different ways for the point value update without and with the shock sensor ($\kappa=1$) are shown in \Cref{fig:2d_sod_density}.
	The density based on the JS shows large deviations between the contact discontinuity and shock wave, which cannot be reduced by the limiting.
	Seen from \Cref{fig:2d_sod_density_decoupled}, the solutions belonging to the DoFs for different point values are decoupled,
	known as the mesh alignment issue, and has been explained in \Cref{sec:mesh_alignment}.
	The results of all the FVS-based methods agree well with the exact solution when the limiting is activated.
	The FVS-based AF methods are more advantageous in simulations since they can cure both the stagnation and mesh alignment issues.
	To save space, in the following tests, we only show the results obtained using the LLF FVS.
		
	\begin{figure}[htb]
		\centering		
		\begin{subfigure}[b]{0.24\textwidth}
			\includegraphics[width=\linewidth]{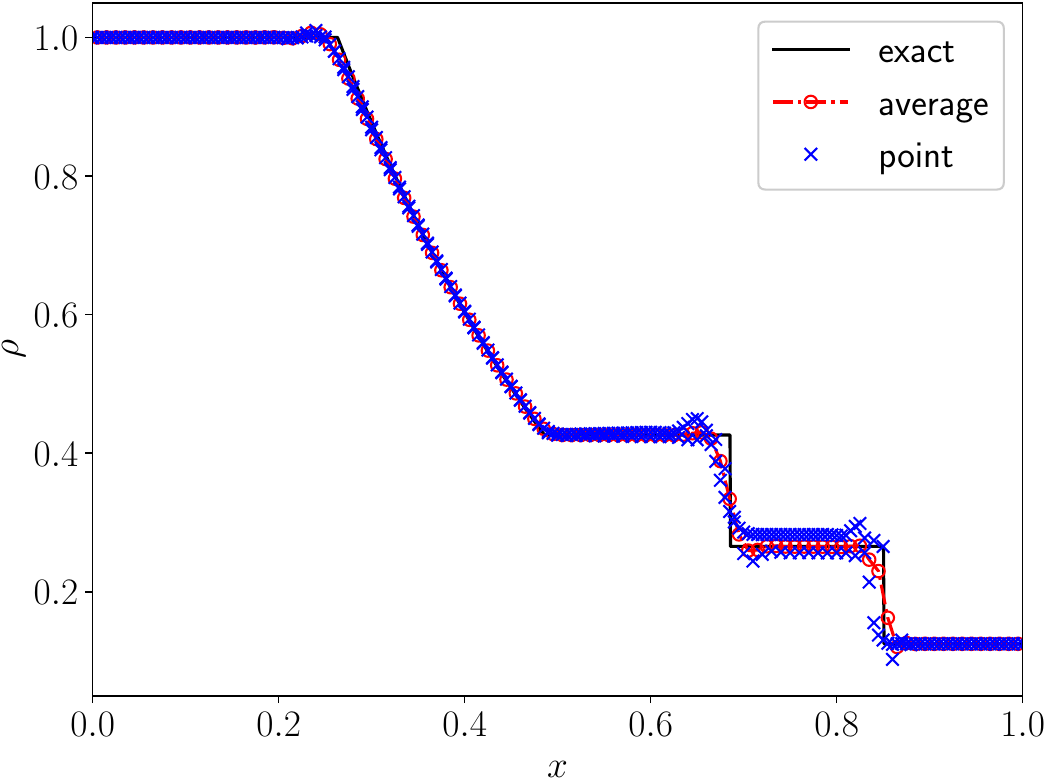}
		\end{subfigure}
		\begin{subfigure}[b]{0.24\textwidth}
			\centering
			\includegraphics[width=\linewidth]{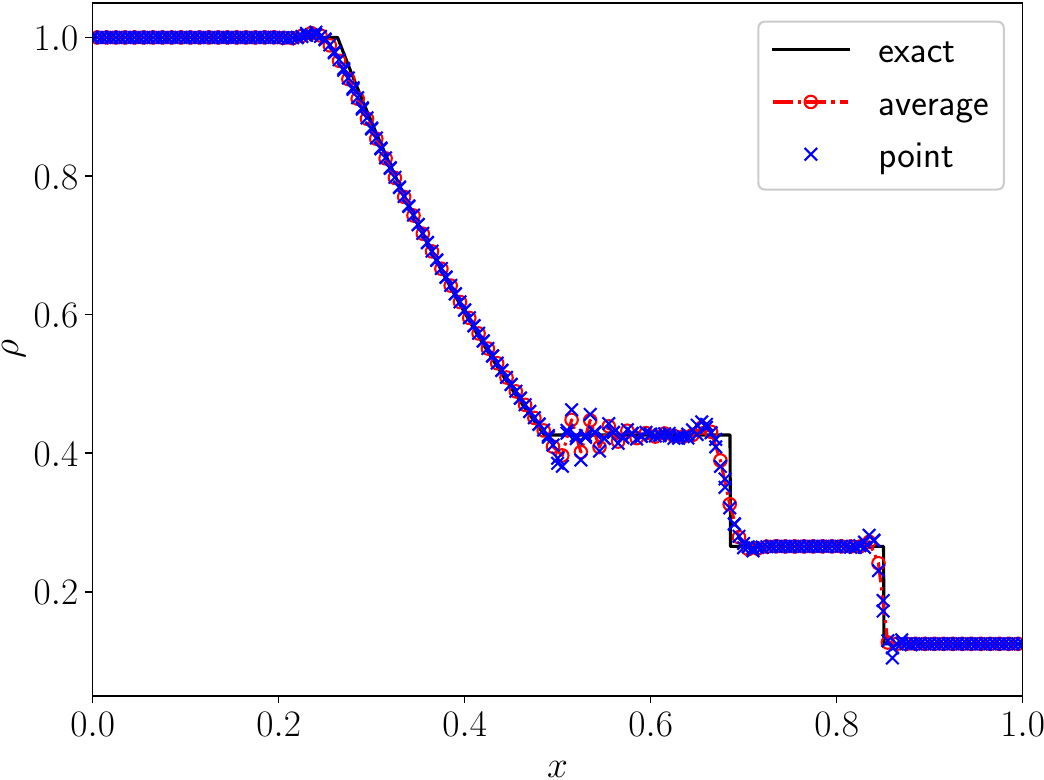}
		\end{subfigure}
		\begin{subfigure}[b]{0.24\textwidth}
			\centering
			\includegraphics[width=\linewidth]{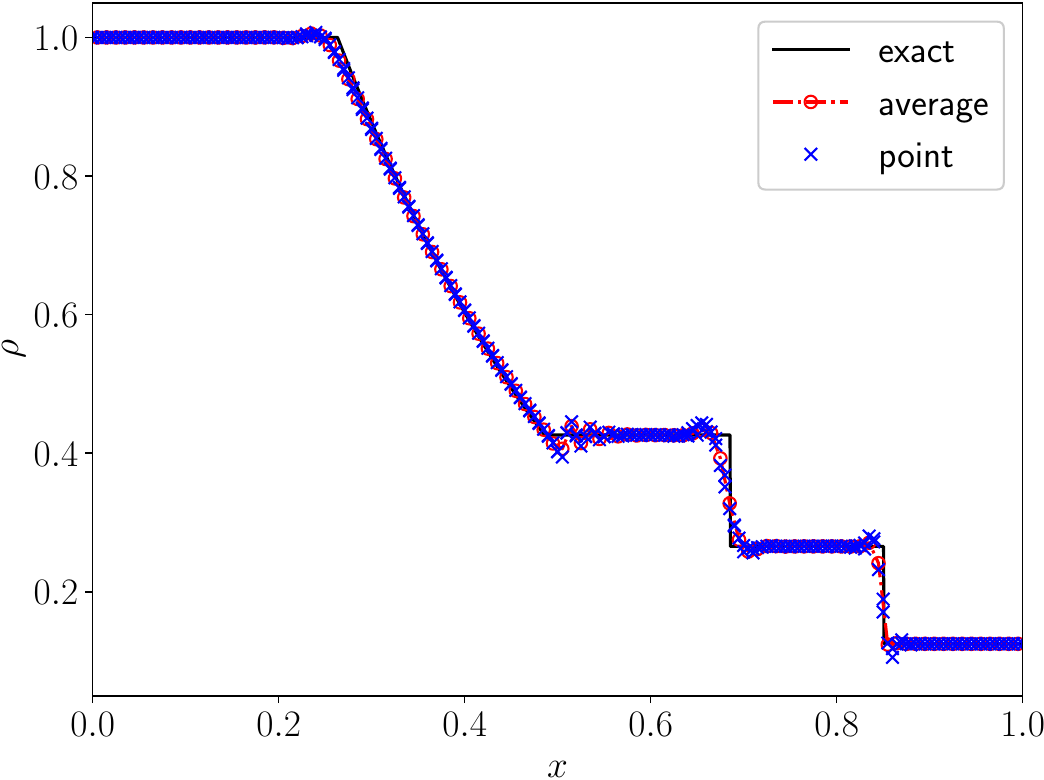}
		\end{subfigure}
		\begin{subfigure}[b]{0.24\textwidth}
			\centering
			\includegraphics[width=\linewidth]{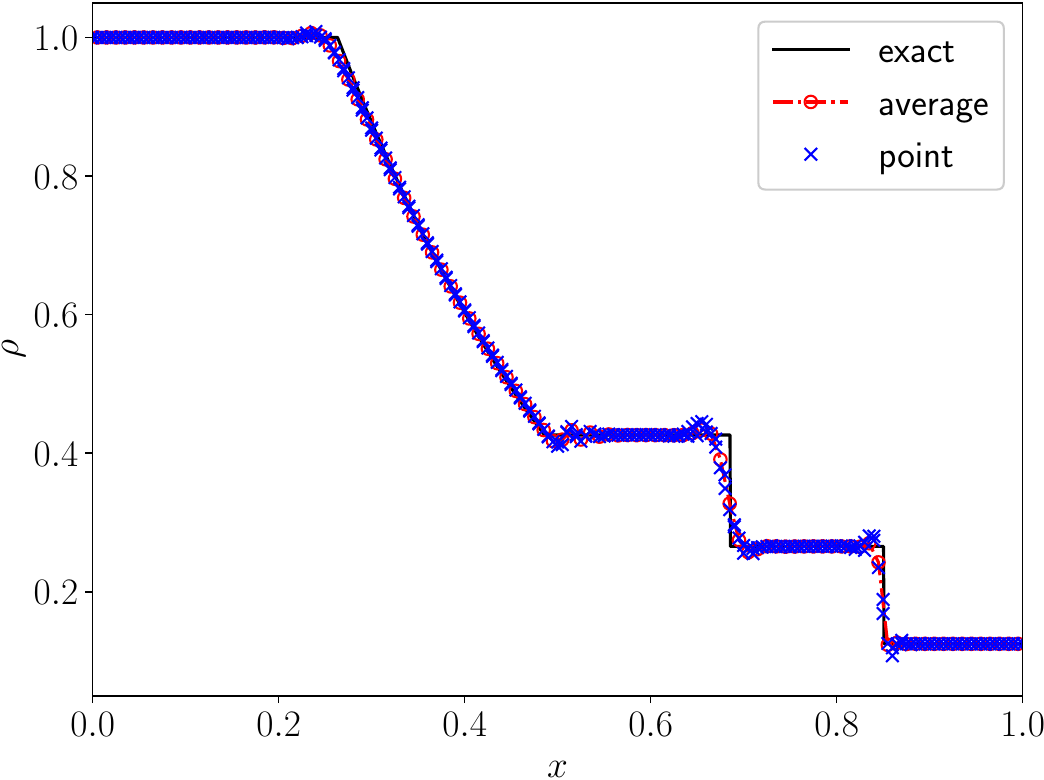}
		\end{subfigure}
		\vspace{5pt}
		
		\begin{subfigure}[b]{0.24\textwidth}
			\includegraphics[width=\linewidth]{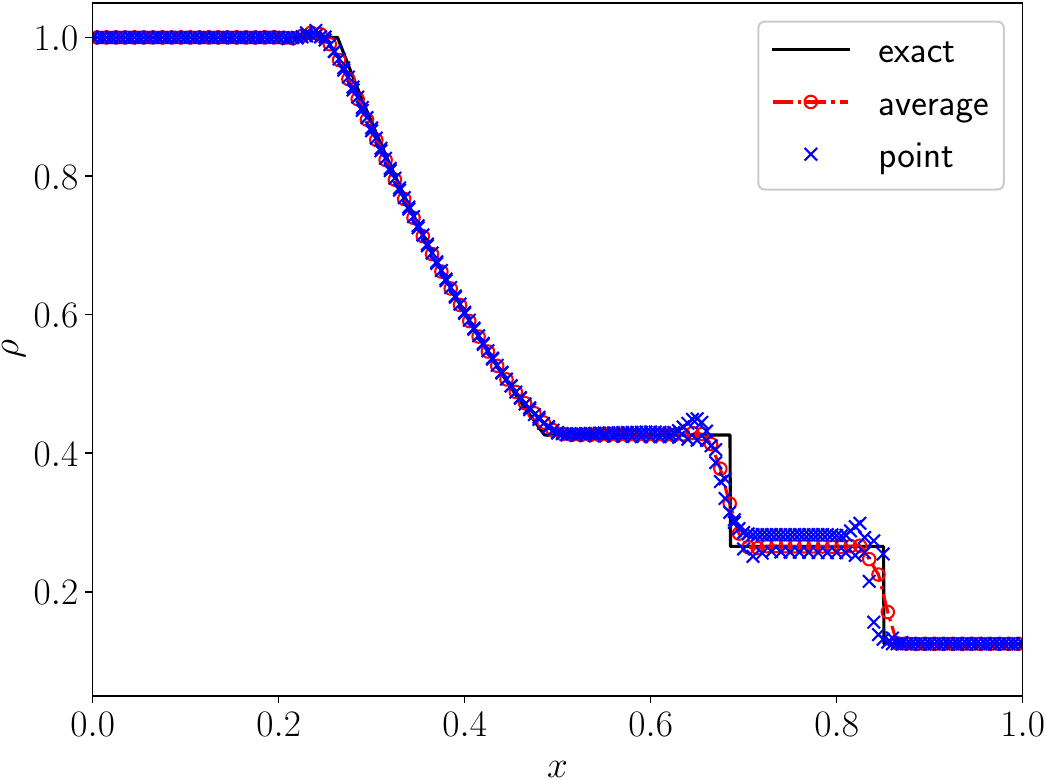}
		\end{subfigure}
		\begin{subfigure}[b]{0.24\textwidth}
			\centering
			\includegraphics[width=\linewidth]{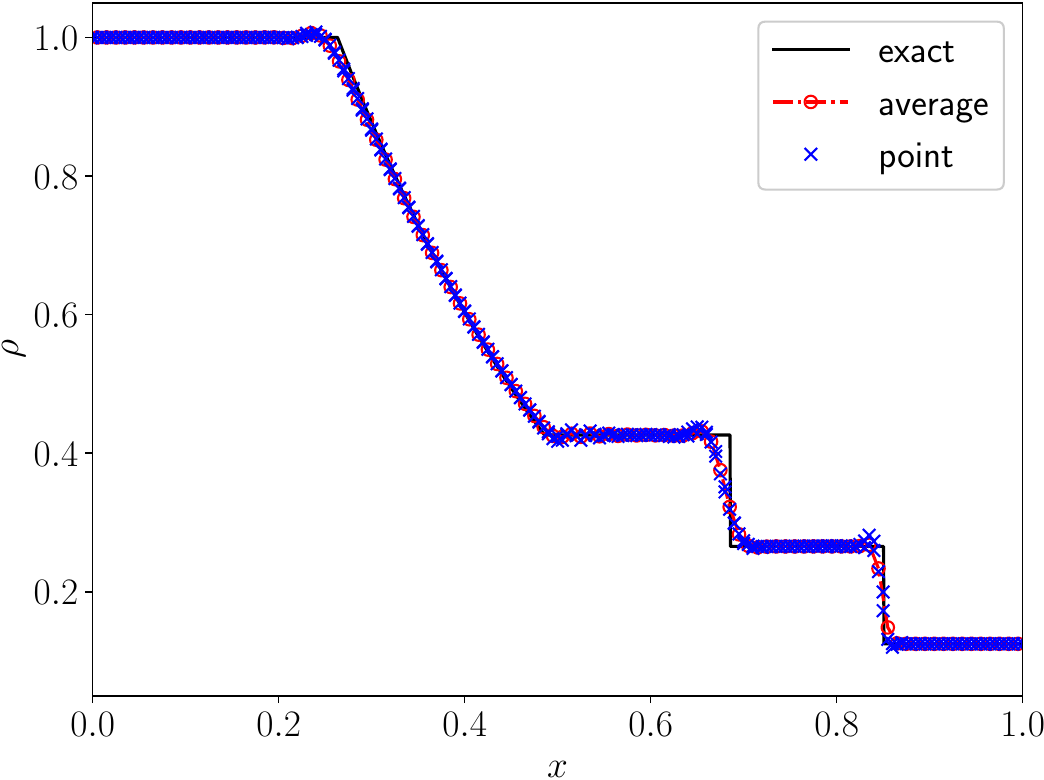}
		\end{subfigure}
		\begin{subfigure}[b]{0.24\textwidth}
			\centering
			\includegraphics[width=\linewidth]{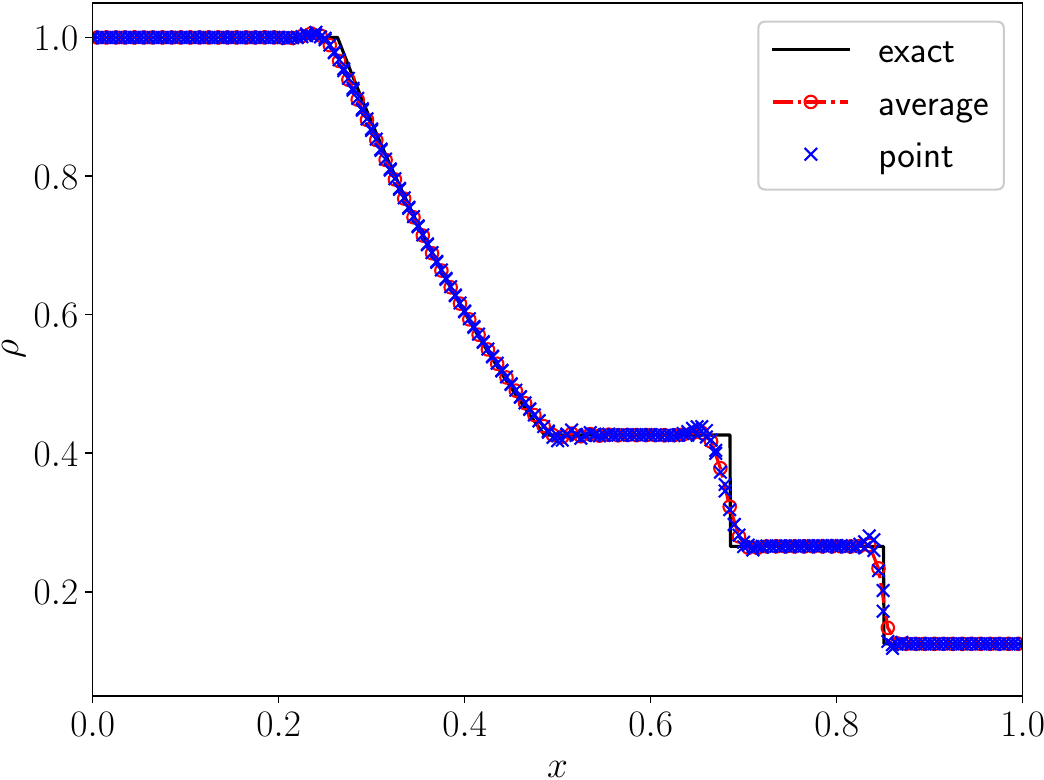}
		\end{subfigure}
		\begin{subfigure}[b]{0.24\textwidth}
			\centering
			\includegraphics[width=\linewidth]{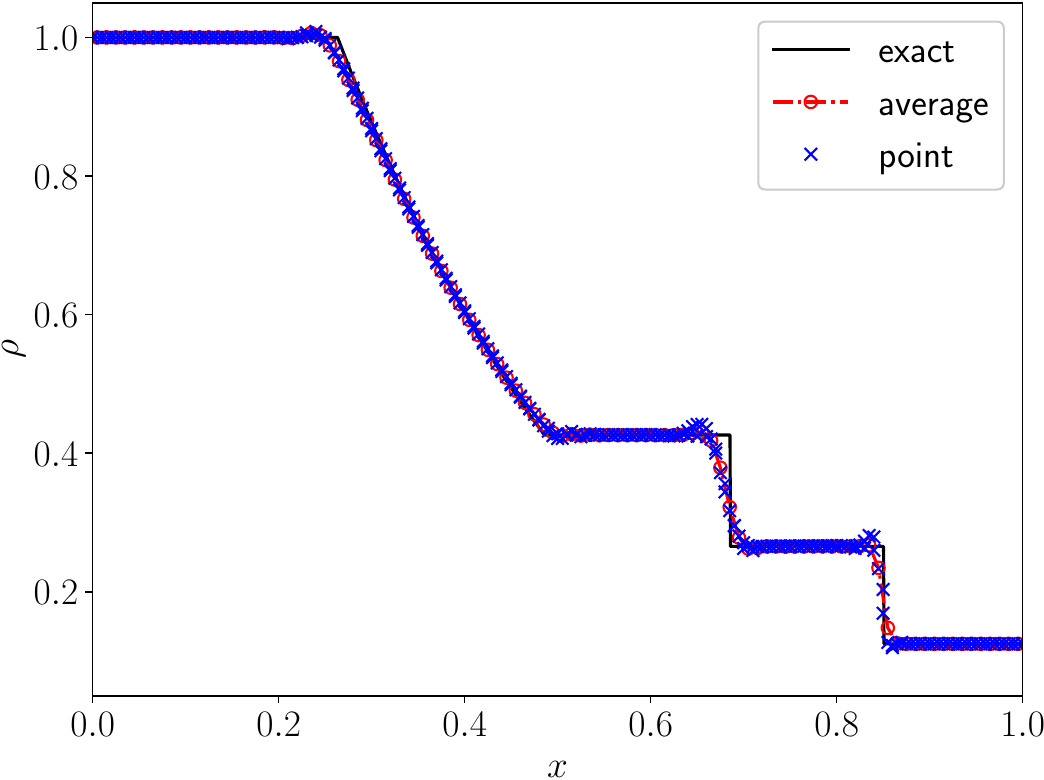}
		\end{subfigure}
		\caption{\Cref{ex:2d_sod}, quasi-2D Sod shock tube.
			The density are computed without (top row) and with the shock sensor-based limiting ($\kappa=1$, bottom row).
			From left to tight: JS, LLF, SW, and VH FVS.}
		\label{fig:2d_sod_density}
	\end{figure}
	
	\begin{figure}[htb]
		\centering		
		\includegraphics[width=0.4\linewidth]{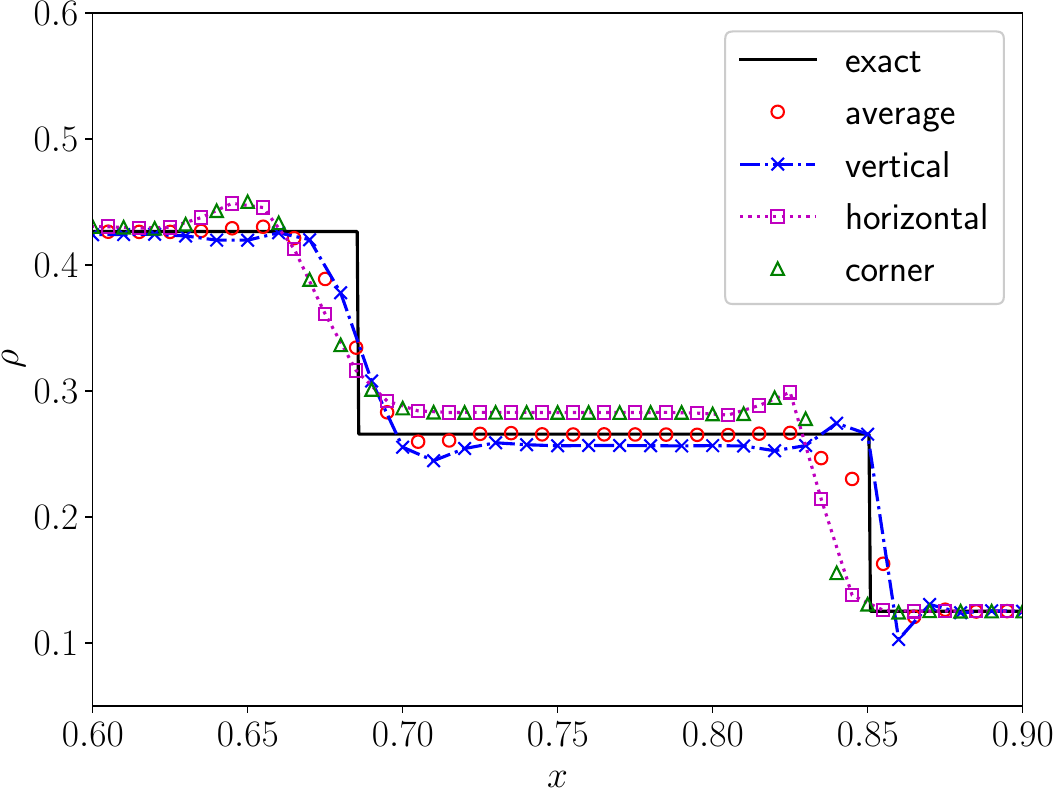}
		\caption{\Cref{ex:2d_sod}, quasi-2D Sod shock tube.
		Based on the JS, the solutions that belong to different kinds of DoFs are decoupled.}
		\label{fig:2d_sod_density_decoupled}
	\end{figure}
			
\end{example}

\begin{example}[Sedov blast wave]\label{ex:2d_sedov}
	The domain is $[-1.1,1.1]\times[-1.1,1.1]$ with outflow boundary conditions.
	The initial density is one, velocity is zero, and the total energy is $10^{-12}$ everywhere except that for the centered cell, the total energy of the cell average and the point values on its faces are $\frac{0.979264}{\Delta x\Delta y}$ with $\Delta x = 2.2/N_1, \Delta y = 2.2/N_2$, which is used to emulate a $\delta$-function at the center.
	
	This test is solved until $T=1$ and the BP limitings are necessary, otherwise, the simulation fails due to negative pressure.
	The density plots obtained with the shock sensor ($\kappa=0.5$) are shown in \Cref{fig:2d_sedov_density}.
	The circular shock wave is well-captured and the numerical solutions converge to the exact solution without spurious oscillations.
	The blending coefficients based on the shock sensor are presented in \Cref{fig:2d_sedov_ss}, indicating that the limiting is locally activated.
	
\begin{figure}[htb]
	\centering
	\begin{subfigure}[b]{0.3\textwidth}
		\centering
		\includegraphics[width=\linewidth]{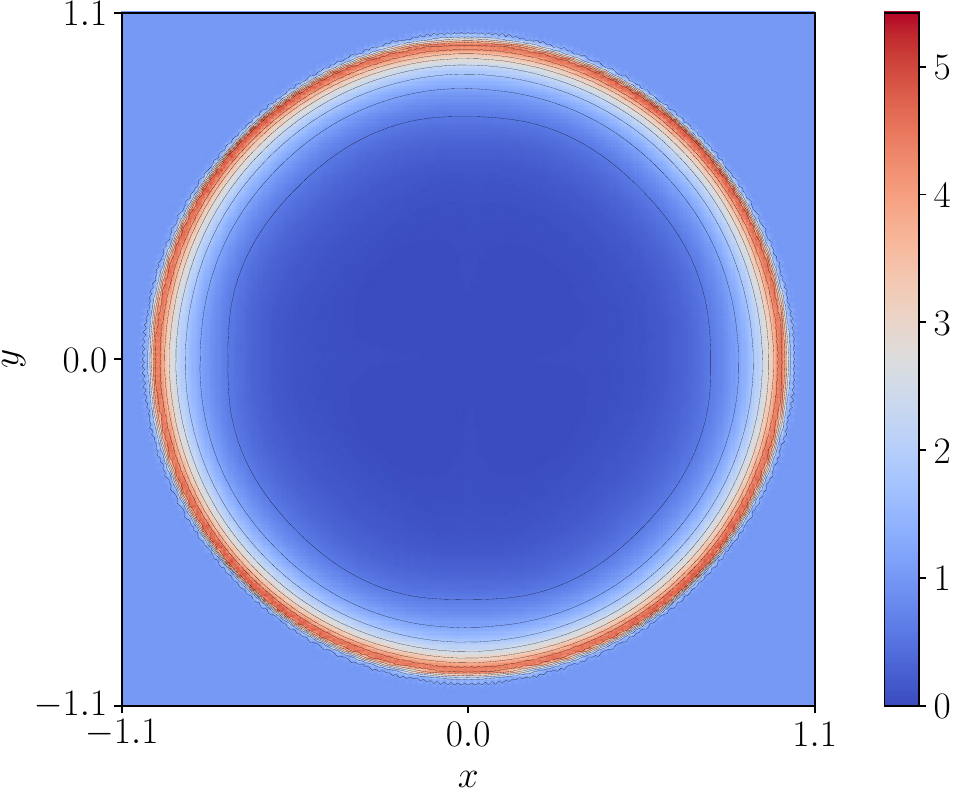}
	\end{subfigure}
	\begin{subfigure}[b]{0.3\textwidth}
		\centering
		\includegraphics[width=\linewidth]{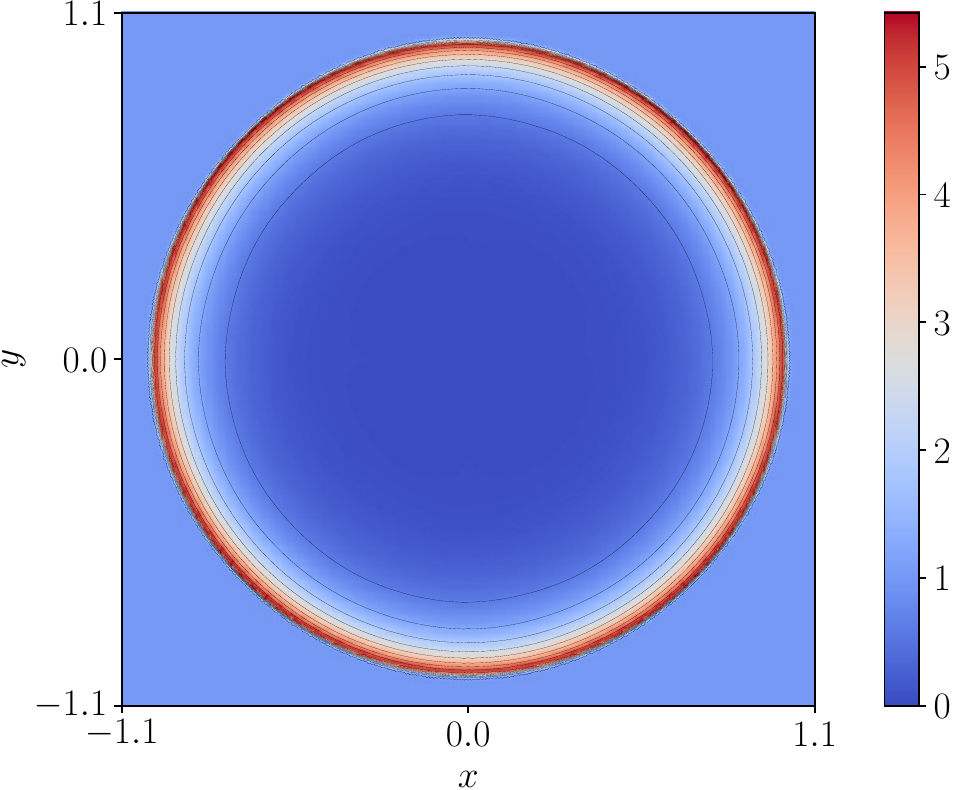}
	\end{subfigure}
        \begin{subfigure}[b]{0.3\textwidth}
		\centering
		\includegraphics[width=0.81\linewidth]{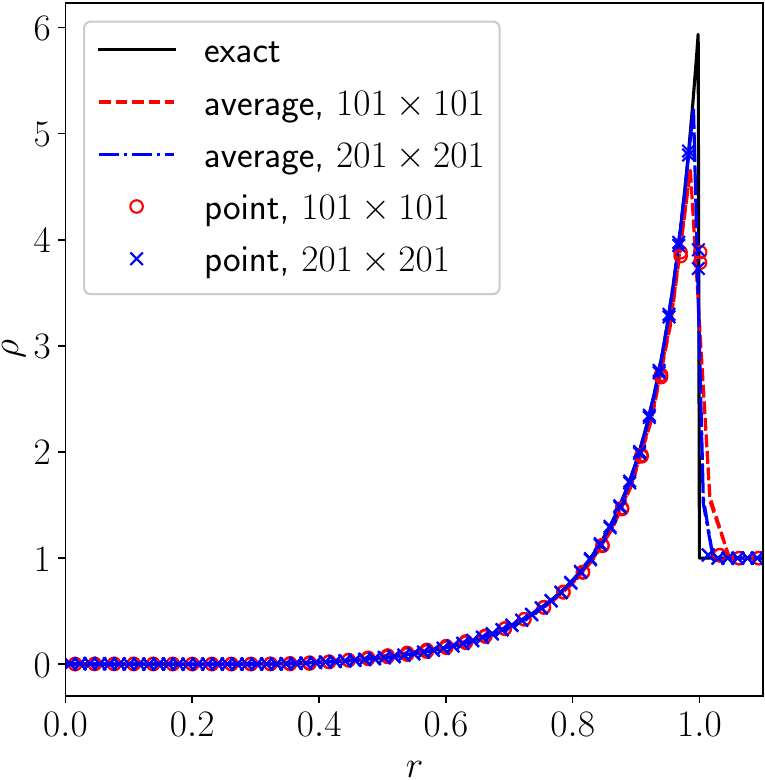}
	\end{subfigure}
	\caption{\Cref{ex:2d_sedov}, 2D Sedov blast wave.
		The density plots computed by the BP AF method.
		From left to right:	$10$ equally spaced contour lines from $0$ to $5.423$ on the uniform $101\times101$ and $201\times201$ meshes, respectively, cut-line along $y=x$. }
	\label{fig:2d_sedov_density}
\end{figure}

\begin{figure}[htb]
	\centering
	\begin{subfigure}[b]{0.3\textwidth}
		\centering
		\includegraphics[width=\linewidth]{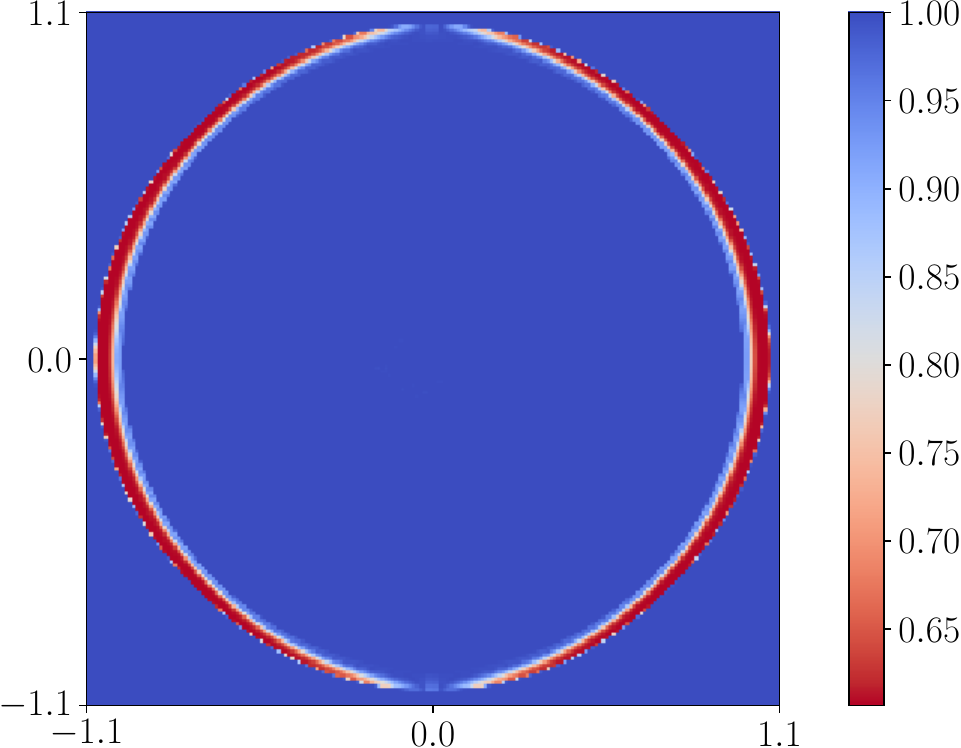}
	\end{subfigure}
	\qquad\qquad\qquad
	\begin{subfigure}[b]{0.3\textwidth}
		\centering
		\includegraphics[width=\linewidth]{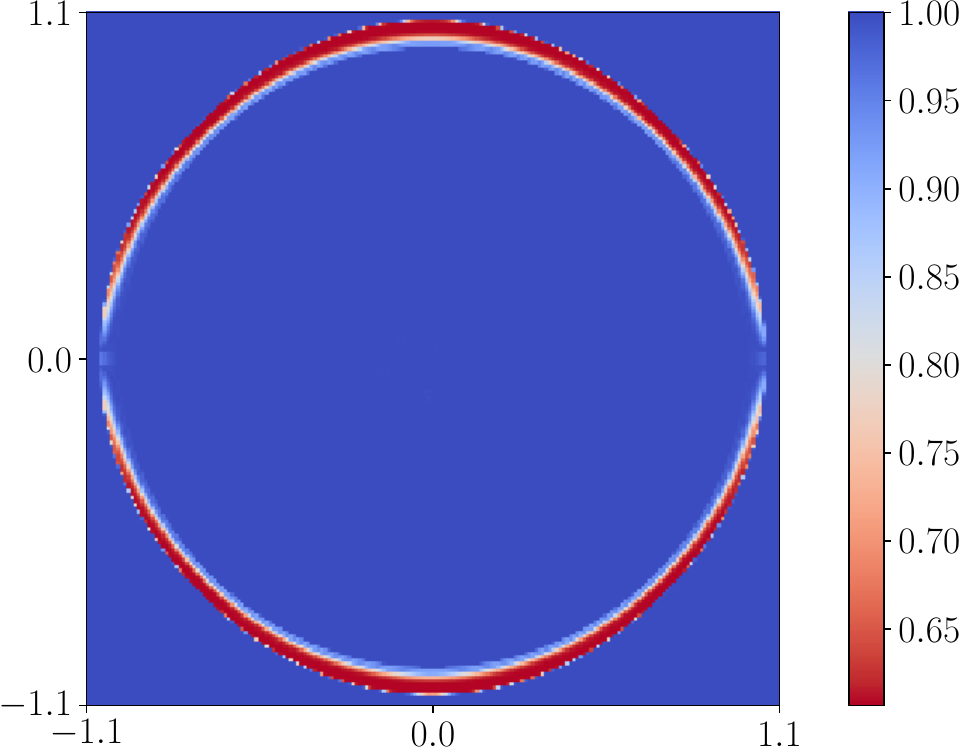}
	\end{subfigure}
	\caption{\Cref{ex:2d_sedov}, 2D Sedov blast wave.
		The shock sensor-based blending coefficients $\theta_{\xr,j}^{s}$ (left) and $\theta_{i,\yr}^{s}$ (right) on the $201\times201$ uniform mesh. }
	\label{fig:2d_sedov_ss}
\end{figure}
\end{example}

\begin{example}[A Mach 3 wind tunnel with a forward-facing step]\label{ex:2d_ffs}
	The initial condition is a Mach 3 flow $(\rho, v_1, v_2, p) = (1.4, 3, 0, 1)$.
	The computational domain is $[0,3]\times[0,1]$ and the step is of height $0.2$ located from $x=0.6$ to $x=3$.
	The inflow and outflow boundary conditions are applied at the entrance ($x=0$) and exit ($x=3$), respectively, and the reflective boundary conditions are imposed at other boundaries.
	
	The density computed by the BP AF method without and with the shock sensor-based limiting at $T=4$ are shown in \Cref{fig:2d_ffs_density_dx}, and the blending coefficients $\theta_{\xr,j}^{s}$, $\theta_{i,\yr}^{s}$ are presented in \Cref{fig:2d_ffs_ss}.
			If only the BP limitings are used, there are oscillations in the numerical solutions, but the BP property is not violated.
			The numerical solutions can be improved by our shock sensor-based limiting.
	Our BP AF method can capture the main features and well-developed Kelvin–Helmholtz roll-ups that originate from the triple point.
	The noise after the shock waves is reduced by the shock sensor-based limiting, while the roll-ups are preserved well.
	Compared to the results obtained by the third-order $P^2$ DG method with the TVB limiter \cite{Cockburn_2001_Runge_JoSC}, the vortices are better captured with the same mesh size $\Delta x=\Delta y=1/160, 1/320$.
	Note that the AF method uses fewer DoFs, showing its efficiency and potential for high Mach number flows.
	
	\begin{figure}[htb]
		\centering
		\begin{subfigure}[b]{0.6\textwidth}
			\centering
			\includegraphics[width=\linewidth]{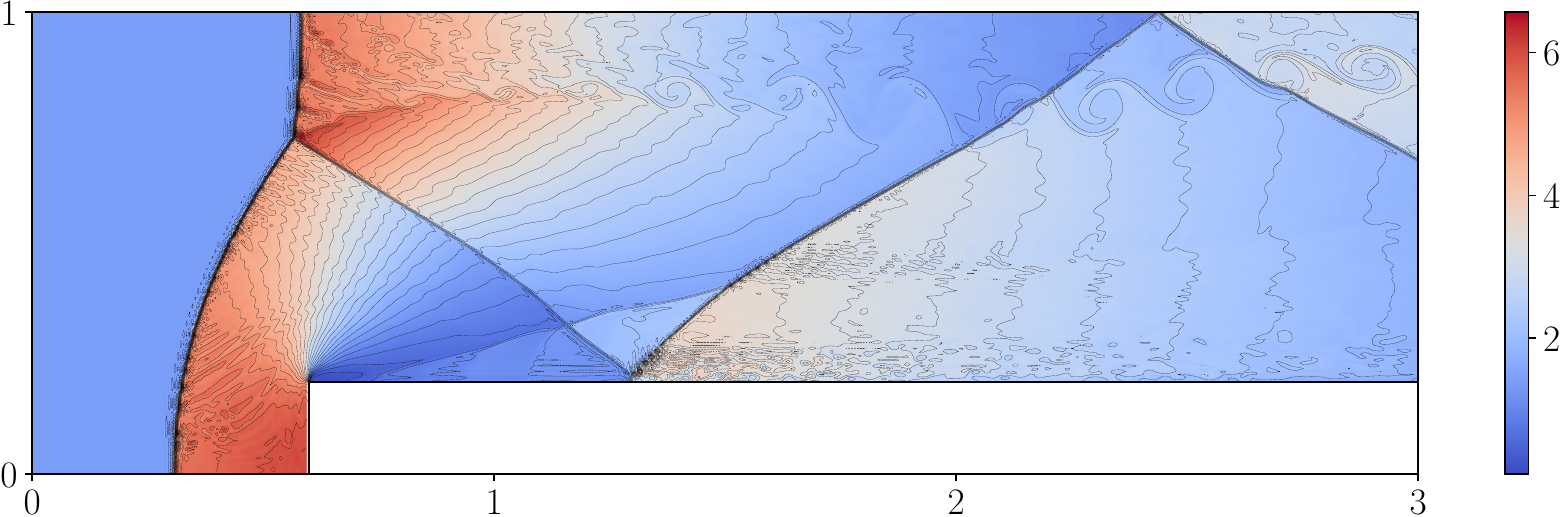}
		\end{subfigure}
		
		\begin{subfigure}[b]{0.6\textwidth}
			\centering
			\includegraphics[width=\linewidth]{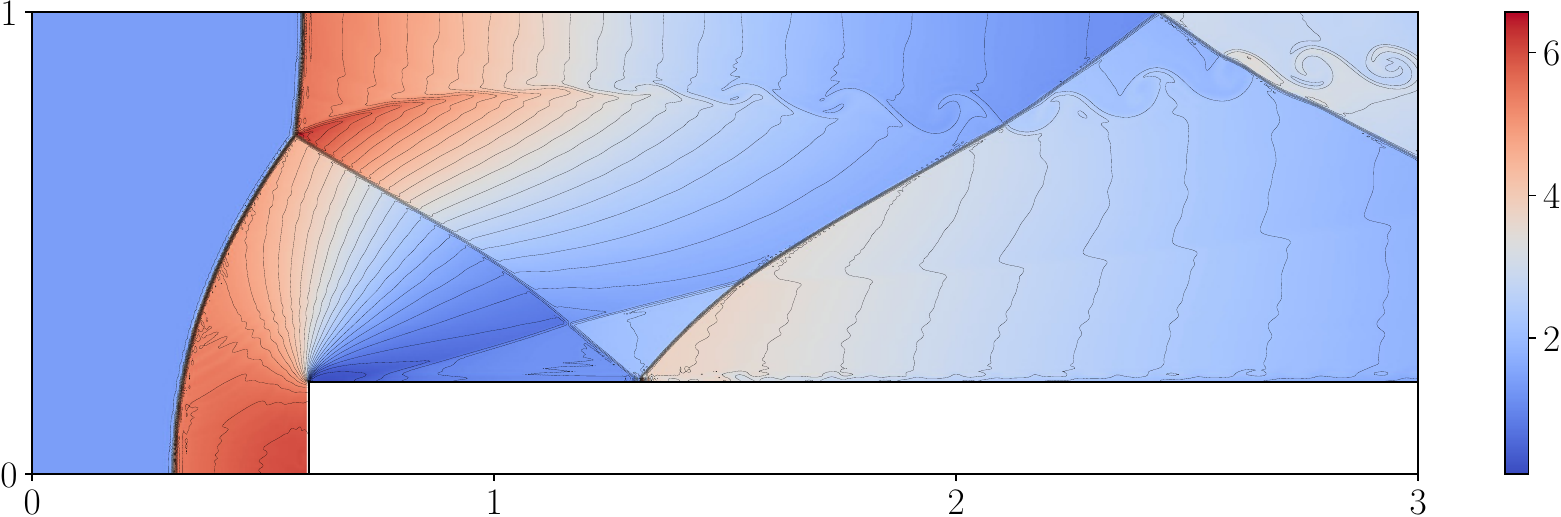}
		\end{subfigure}

		\begin{subfigure}[b]{0.6\textwidth}
			\centering
			\includegraphics[width=\linewidth]{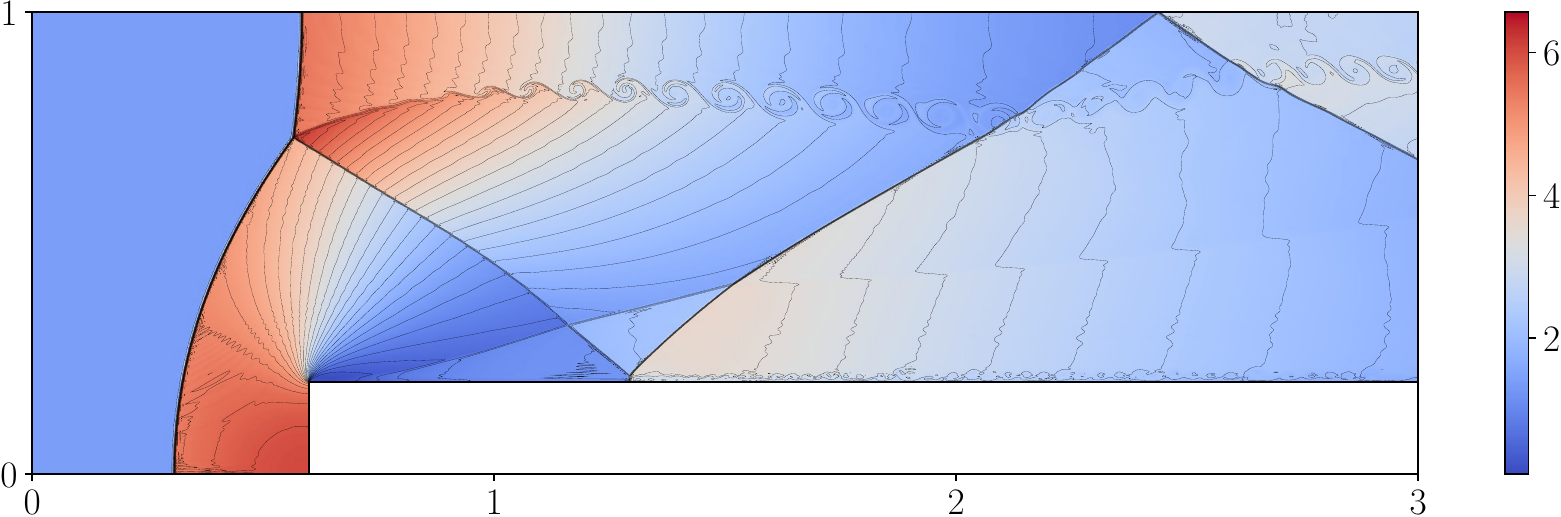}
		\end{subfigure}
		\caption{\Cref{ex:2d_ffs}, forward-facing step problem.
			$30$ equally spaced contour lines of the density from $0.098$ to $6.566$.
			From top to bottom: $480\times160$ mesh without shock sensor, $480\times160$ mesh with $\kappa=1$, 	$960\times320$ mesh with $\kappa=1$.
			}
		\label{fig:2d_ffs_density_dx}
	\end{figure}
	
	\begin{figure}[htb]
		\centering		
		\begin{subfigure}[b]{0.48\textwidth}
			\centering
			\includegraphics[width=\linewidth]{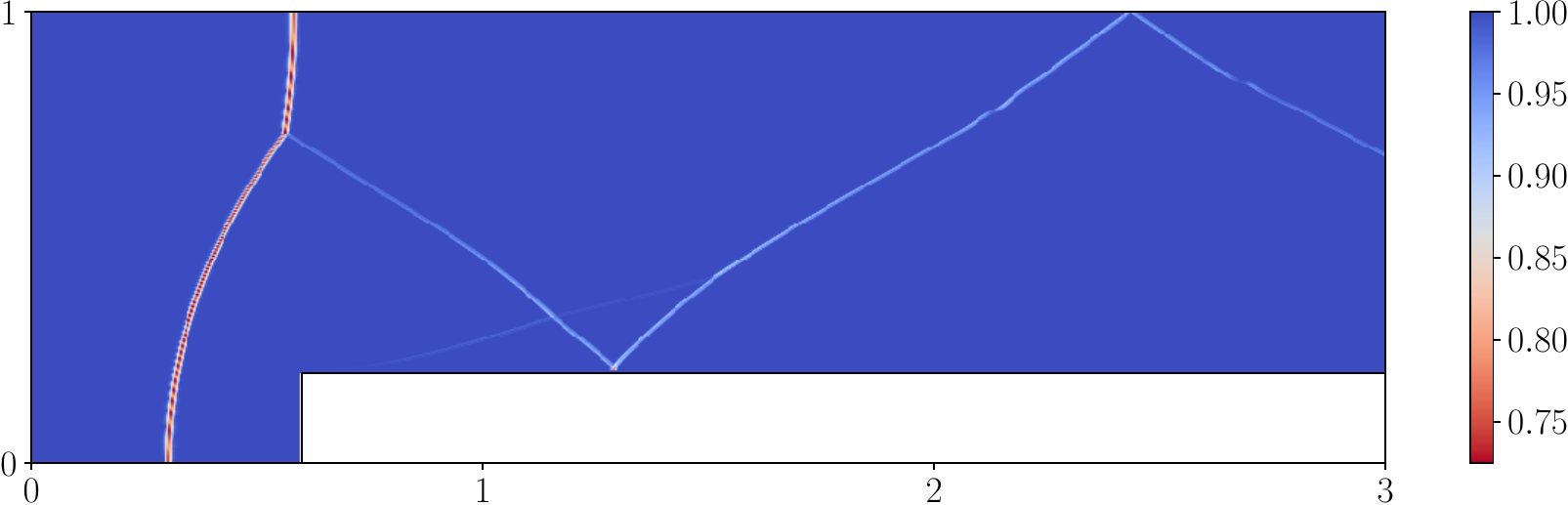}
		\end{subfigure}
		\quad
		\begin{subfigure}[b]{0.48\textwidth}
			\centering
			\includegraphics[width=\linewidth]{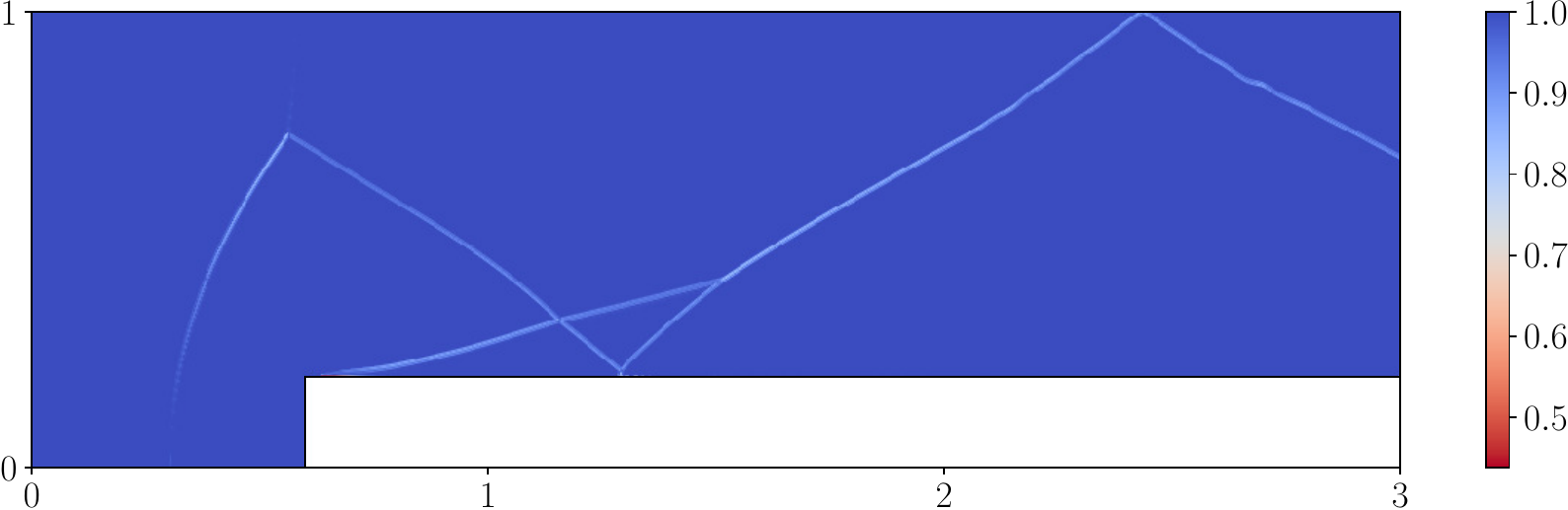}
		\end{subfigure}
		\caption{\Cref{ex:2d_ffs}, forward-facing step problem.
			The blending coefficients $\theta_{\xr,j}^{s}$ (left) and $\theta_{i,\yr}^{s}$ (right) based on the shock sensor with $\kappa=1$ on the $960\times 320$ mesh.}
		\label{fig:2d_ffs_ss}
	\end{figure}
\end{example}

\begin{example}[High Mach number astrophysical jets]\label{ex:2d_jet}
	This test follows the setup in \cite{Zhang_2010_positivity_JoCP}.
        The first case considers a Mach $80$ jet on a computational domain $[0,2]\times[-0.5,0.5]$, initially filled with ambient gas with $(\rho, v_1, v_2, p) = (0.5, 0, 0, 0.4127)$.
        A jet is injected into the domain with $(\rho, v_1, v_2, p) = (5, 30, 0, 0.4127)$ at the left boundary when $\abs{y}<0.05$.
		The free boundary conditions are applied on other boundaries.
        The second case considers a Mach $2000$ jet on a computational domain $[0,1]\times[-0.25,0.25]$.
        The initial condition and boundary conditions are the same as the first case except that the state of the jet is $(\rho, v_1, v_2, p) = (5, 800, 0, 0.4127)$.
        The adiabatic index is $\gamma=5/3$, and 
        the output time is $0.07$ and $0.001$ for the two cases, respectively.
	
	The numerical solutions obtained by the BP AF methods with the shock sensor on the uniform $400\times200$ mesh are shown in \Cref{fig:2d_jet}.
	The main flow structures and small-scale features are captured well, comparable to those in \cite{Zhang_2010_positivity_JoCP}.
	
	\begin{figure}[htb]
		\centering
		\begin{subfigure}[b]{0.48\textwidth}
			\centering
			\includegraphics[width=\linewidth]{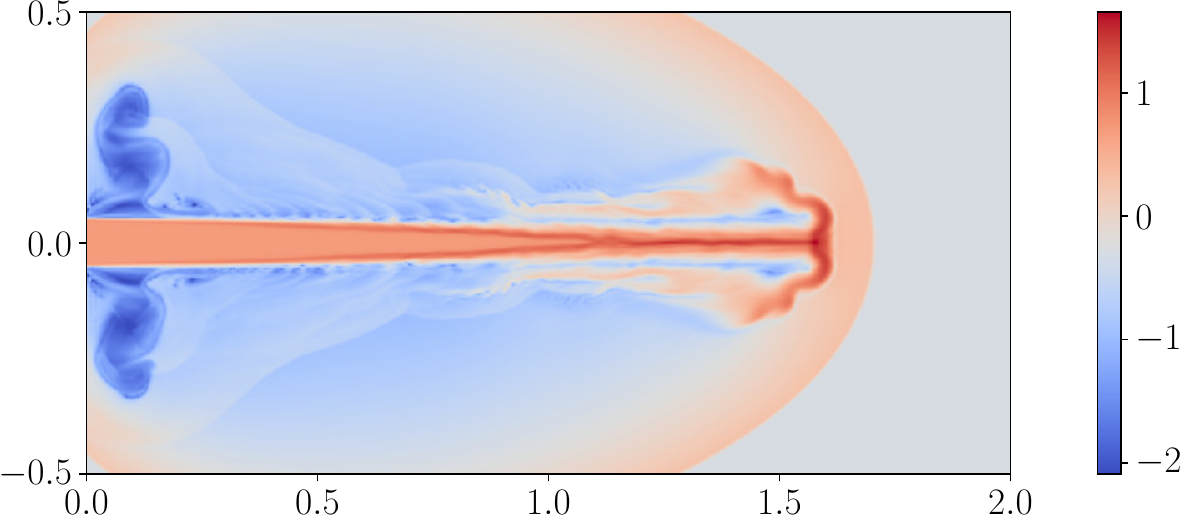}
		\end{subfigure}
		\begin{subfigure}[b]{0.48\textwidth}
			\centering
			\includegraphics[width=\linewidth]{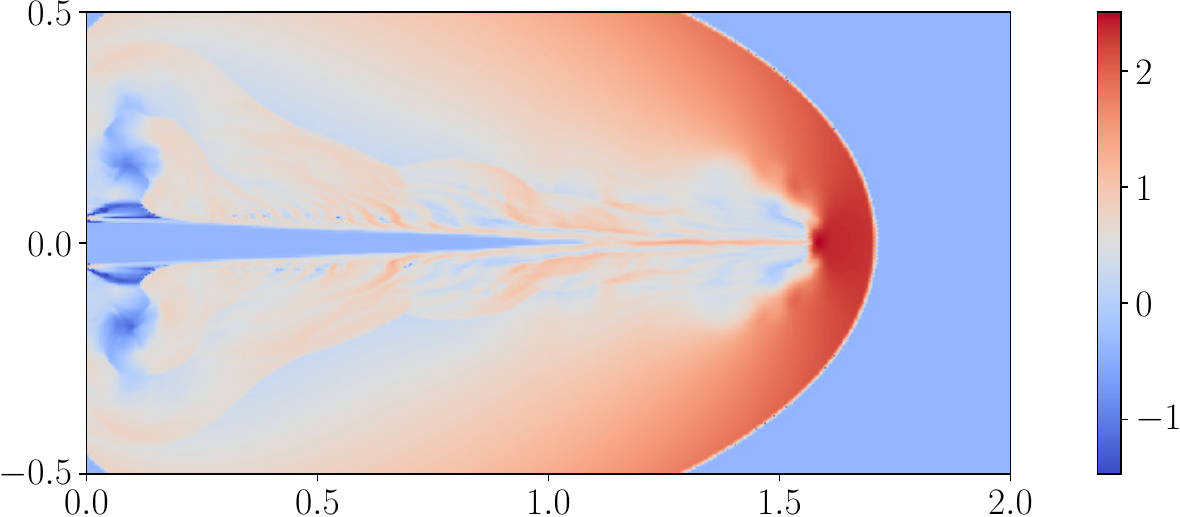}
		\end{subfigure}
		\vspace{5pt}
		
		\begin{subfigure}[b]{0.48\textwidth}
			\centering
			\includegraphics[width=\linewidth]{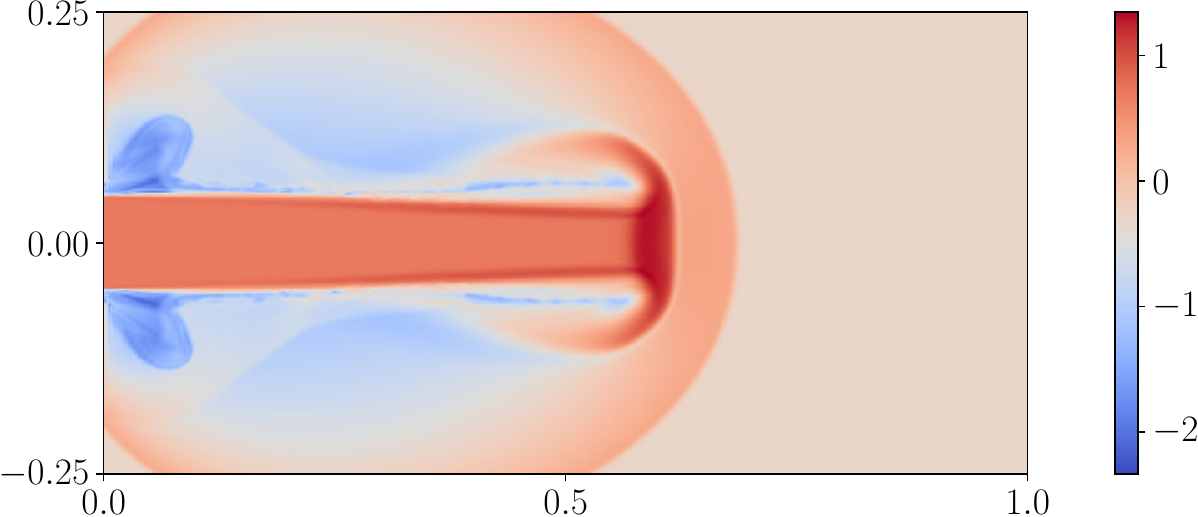}
		\end{subfigure}
		\begin{subfigure}[b]{0.48\textwidth}
			\centering
			\includegraphics[width=\linewidth]{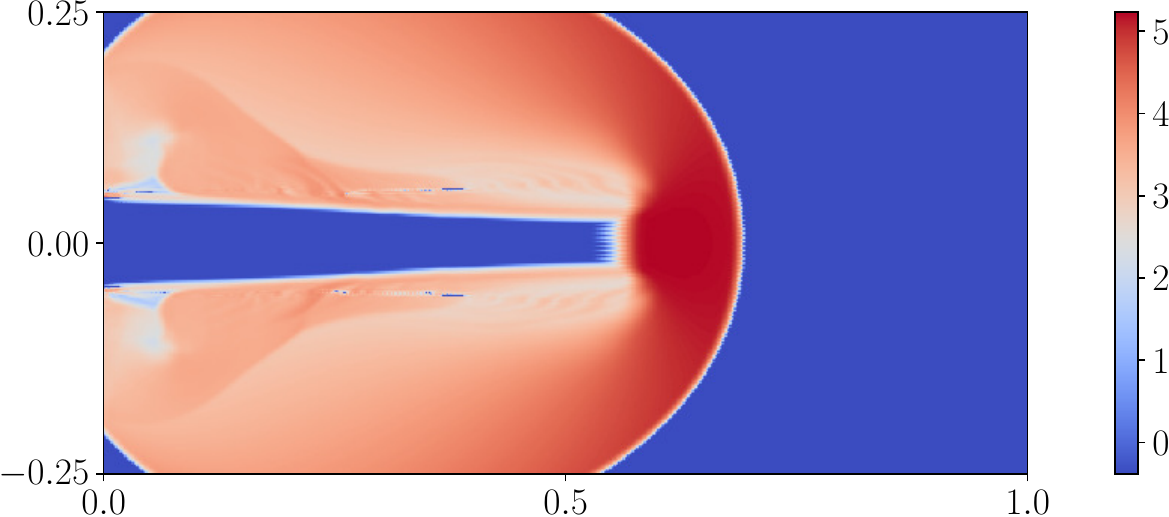}
		\end{subfigure}
		\caption{\Cref{ex:2d_jet}, the Mach $80$ jet (top row) and Mach $2000$ jet (bottom row).
		$\log_{10}\rho$ (left) and $\log_{10}p$ (right) obtained with the BP limitings and shock sensor-based limiting ($\kappa=1$ for Mach 80 and $10$ for Mach 2000, respectively).}
		\label{fig:2d_jet}
	\end{figure}
\end{example}

\section{Conclusion}\label{sec:conclusion}
In the active flux (AF) methods, it is pivotal to design suitable point values update at cell interfaces, to achieve stability and high-order accuracy.
The point value update based on the Jacobian splitting (JS) may lead to the stagnation and mesh alignment issues.
This paper proposed to use the flux vector splitting (FVS) for the point value update instead of the JS, which keeps the continuous reconstruction as the original AF methods, and offers a natural and uniform remedy to those two issues.
To further improve the robustness of the AF methods, this paper developed bound-preserving (BP) AF methods for hyperbolic conservation laws, achieved
by blending the high-order AF methods with the first-order local Lax-Friedrichs (LLF) or Rusanov methods for both the cell average and point value updates, where the convex limiting and scaling limiter were employed, respectively.
The shock sensor-based limiting was proposed to further improve the shock-capturing ability.
The challenging numerical tests verified the robustness and effectiveness of our BP AF methods,
and also showed that the LLF FVS is generally superior to others in terms of the CFL number and non-oscillatory property.
Moreover, for the forward-facing step problem, the present FVS-based BP AF method was able to capture small-scale features better compared to the third-order discontinuous Galerkin method with the TVB limiter on the same mesh resolution \cite{Cockburn_2001_Runge_JoSC}, while using fewer degrees of freedom, demonstrating the efficiency and potential of our BP AF method for high Mach number flows.


\section*{Acknowledgement}

JD was supported by an Alexander von Humboldt Foundation Research fellowship CHN-1234352-HFST-P. CK and WB acknowledge funding by the Deutsche Forschungsgemeinschaft (DFG, German Research Foundation) within \textit{SPP 2410 Hyperbolic Balance Laws in Fluid Mechanics: Complexity, Scales, Randomness (CoScaRa)}, project number 525941602.
We acknowledge helpful discussions with Praveen Chandrashekar at TIFR-CAM Bangalore on the Ducros' shock sensor.

\newcommand{\etalchar}[1]{$^{#1}$}


\appendix

\section{2D flux vector splitting}\label{sec:2d_fvs}
\subsection{Local Lax-Friedrichs flux vector splitting}
This flux vector splitting can be written as
\begin{equation*}
	\bF_\ell^\pm = \frac12(\bF_\ell(\bU) \pm \alpha_\ell \bU),
\end{equation*}
where $\alpha_\ell$ is determined by
\begin{align*}
	(\alpha_1)_{\xr,q} &= \max_{s} \left\{\left|\varrho_{1}(\bU_{s,q})\right|\right\},
	~s \in \left\{ i-\frac12, i, i+\frac12, i+1, i+\frac32\right\}, q=j,j+\frac12, \\
	(\alpha_2)_{q,\yr} &= \max_{s} \left\{\left|\varrho_{2}(\bU_{q,s})\right|\right\},
	~s \in \left\{ j-\frac12, j, j+\frac12, j+1, j+\frac32\right\}, q=i,i+\frac12, \nonumber
\end{align*}
and $\varrho_\ell$ is the spectral radius of the Jacobian matrix ${\partial{\bF_\ell}}/{\partial\bU}$.

\subsection{Upwind flux vector splitting}
The flux can also be split based on each characteristic field as follows
\begin{equation}\label{eq:upwind_fvs2}
	\bF_\ell^\pm = \frac12(\bF_\ell(\bU)\pm \abs{\bJ_\ell} \bU),\quad
	\abs{\bJ_\ell} = \bR_\ell(\bm{\Lambda}_\ell^{+} - \bm{\Lambda}_\ell^{-})\bR_\ell^{-1},
\end{equation}
with $\bJ_\ell={\partial\bF_\ell}/{\partial\bU} = \bR_\ell\bm{\Lambda}_\ell\bR_\ell^{-1}$ the eigen-decomposition of the Jacobian matrix.

For the Euler equations, the explicit expressions in the $x$-direction are
\begin{align*}
	\bF_1^{\pm} &= \begin{bmatrix}
		\frac{\rho}{2\gamma}\alpha^\pm \\
		\frac{\rho}{2\gamma}\left(\alpha^\pm v_1 + a (\lambda_2^\pm - \lambda_3^\pm)\right) \\
		\frac{\rho}{2\gamma}\alpha^\pm v_2 \\
		\frac{\rho}{2\gamma}\left(\frac12\alpha^\pm \norm{\bv}_2^2 + a v_1 (\lambda_2^\pm - \lambda_3^\pm) + \frac{a^2}{\gamma-1}(\lambda_2^\pm + \lambda_3^\pm)\right) \\
	\end{bmatrix},
\end{align*}
where
$\lambda_1 = v_\ell, ~\lambda_2 = v_\ell + a, ~\lambda_3 = v_\ell - a,
~
\alpha^\pm = 2(\gamma-1)\lambda_1^\pm + \lambda_2^\pm + \lambda_3^\pm$,
and $a = \sqrt{\gamma p / \rho}$ is the sound speed.
The expressions in the $y$-direction can be obtained using the rotational invariance.

\subsection{Van Leer-H\"anel flux vector splitting for the Euler equations}
For the $x$-direction, the flux is split according to the Mach number $M=v_1/a$ as
\begin{equation*}
	\bF_1 = \begin{bmatrix}
		\rho a M \\
		\rho a^2(M^2 + \frac{1}{\gamma}) \\
		\rho a M v_2 \\
		\rho a^3M(\frac{1}{2}M^2 + \frac{1}{\gamma-1}) + \frac{\rho a M v_2^2}{2} \\
	\end{bmatrix} = \bF_1^+ + \bF_1^-,~
	\bF_1^{\pm} = \begin{bmatrix}
		\pm\frac{1}{4}\rho a (M\pm1)^2 \\
		\pm\frac{1}{4}\rho a (M\pm1)^2 v_1 + p^{\pm} \\
		\pm\frac{1}{4}\rho a (M\pm1)^2 v_2 \\
		\pm\frac{1}{4}\rho a (M\pm1)^2 H \\
	\end{bmatrix}
\end{equation*}
with the enthalpy $H=(E+p)/\rho$,
and the pressure-splitting
$p^{\pm} = \frac{1}{2}(1\pm\gamma M)p$.

\section{Bound-preserving property of intermediate states}\label{sec:2d_llf_bp}
 Similar to the proofs in \cite{Zhang_2010_maximum_JoCP,Zhang_2010_positivity_JoCP}, the following lemmas hold.
\begin{lemma}\label{lem:2d_llf_scalar}
    For the scalar conservation laws \eqref{eq:2d_scalar}, the intermediate state $\widetilde{u}=\frac12(u_L + u_R) + \frac{1}{2\alpha}(f_\ell(u_L) - f_\ell(u_R))$ stays in $\mathcal{G}$ \eqref{eq:2d_scalar_g} if $\alpha\geqslant\max\{\varrho_\ell(u_L), \varrho_\ell(u_R)\}$.
\end{lemma}

\begin{proof}
    The partial derivatives of the intermediate state satisfy
    \begin{equation*}
	        \frac{\partial\widetilde{u}(u_L, u_R)}{\partial u_L} = \frac12\left(1 + \frac{f_\ell'(u_L)}{\alpha}\right)\geqslant 0, \quad
	        \frac{\partial\widetilde{u}(u_L, u_R)}{\partial u_R} = \frac12\left(1 - \frac{f_\ell'(u_R)}{\alpha}\right)\geqslant 0.
	    \end{equation*}
    As $\widetilde{u}(m_0, m_0) = m_0$, $\widetilde{u}(M_0, M_0) = M_0$, it holds $m_0 \leqslant \widetilde{u} \leqslant M_0$.
\end{proof}

\begin{lemma}\label{lem:2d_llf_euler}
    For the Euler equations, the intermediate state $\widetilde{\bU}=\frac12(\bU_L + \bU_R) + \frac{1}{2\alpha}(\bF_\ell(\bU_L) - \bF_\ell(\bU_R))$ stays in $\mathcal{G}$ \eqref{eq:2d_euler_g} if $\alpha\geqslant\max\{\varrho_\ell(\bU_L), \varrho_\ell(\bU_R)\}$.
\end{lemma}

\begin{proof}
    For the Euler equations, as the intermediate state is a convex combination of ${\bU}_L - \frac{1}{\alpha}\bF_\ell({\bU}_L)$ and ${\bU}_{R} + \frac{1}{\alpha}\bF_\ell({\bU}_R)$, we only need to show that the ${\bU} \pm \frac{1}{\alpha}\bF_\ell({\bU})$ belongs to $\mathcal{G}$.
    The density component $\left(\rho \pm (\rho v_\ell)/{\alpha} \right)$ is positive since $\alpha>\abs{v_\ell}$.
    The recovered internal energy is
    \begin{align*}
	        \rho e\left({\bU} \pm \frac{1}{\alpha}\bF_\ell({\bU})\right)
	        &= E\left({\bU} \pm \frac{1}{\alpha}\bF_\ell({\bU})\right) - \frac{\norm{\rho \bv\left({\bU} \pm \frac{1}{\alpha}\bF_\ell({\bU})\right)}_2^2}{2\rho\left({\bU} \pm \frac{1}{\alpha}\bF_\ell({\bU})\right)} \\
	        &= \left(1 - \frac{p^2}{2(\alpha\pm v_\ell)^2\rho^2e}\right)\left(1\pm\frac{v_\ell}{\alpha}\right)\rho e,
	    \end{align*}
    so that one has $\rho e\left({\bU} \pm \frac{1}{\alpha}\bF_\ell({\bU})\right)> 0 \iff \frac{p^2}{2\rho^2e} < (\alpha\pm v_\ell)^2 \iff \frac{\gamma-1}{2\gamma}a^2 < (\alpha\pm v_\ell)^2$ for the perfect gas EOS,
    which holds as $\alpha\geqslant \abs{v_\ell}+a$. 
\end{proof}

\section{1D bound-preserving active flux methods}\label{sec:1d_limitings}
	For the scalar conservation law \eqref{eq:1d_scalar}, its solutions
	satisfy a strict maximum principle (MP) \cite{Dafermos_2000_Hyperbolic_book}, i.e.,
	\begin{equation}\label{eq:1d_scalar_g}
		\mathcal{G} = \left\{ u ~|~ m_0 \leqslant u \leqslant M_0 \right\},
		\quad m_0 = \min_{x} u_0(x), ~M_0 = \max_{x} u_0(x).
	\end{equation}
	For the compressible Euler equations, the admissible state set is
	\begin{equation}\label{eq:1d_euler_g}
		\mathcal{G} = \left\{\bU = \left(\rho, \rho v, E\right) ~\Big|~ \rho > 0,~ p = (\gamma-1)\left(E - (\rho v)^2/(2\rho)\right) > 0 \right\}.
	\end{equation}
	which is convex, see e.g. \cite{Zhang_2011_Positivity_JoCP}.

\subsection{Convex limiting for the cell average}\label{sec:1d_limiting_average}
This section presents a convex limiting approach to achieve the BP property of the cell average update.
The low-order scheme is chosen as the first-order LLF scheme
\begin{align*}
	&\overline{\bU}^{\texttt{L}}_i = \overline{\bU}^{n}_i - \mu_{i}\left(\widehat{\bF}^{\texttt{L}}_{\xr} - \widehat{\bF}^{\texttt{L}}_{\xl}\right),\quad
	\mu_{i} = \Delta t^n / \Delta x_i, \\
	&\widehat{\bF}^{\texttt{L}}_{\xr} =
	\bF^{\texttt{LLF}}(\overline{\bU}^{n}_i, \overline{\bU}^{n}_{i+1}) = \frac12\left(\bF(\overline{\bU}^{n}_i) + \bF(\overline{\bU}^{n}_{i+1})\right) - \frac{\alpha_{\xr}}{2}\left(\overline{\bU}^{n}_{i+1} - \overline{\bU}^{n}_i\right), \\
	&\alpha_{\xr} = \max\{\varrho(\overline{\bU}^{n}_{i}), ~\varrho_{\ell}(\overline{\bU}^{n}_{i+1})\},
\end{align*}
where $\varrho$ is the spectral radius of $\partial\bF/\partial\bU$.
Note that here $\alpha_{\xr}$ is not the same as the one in the LLF FVS \eqref{eq:1d_Rusanov_alpha}.
Following \cite{Guermond_2016_Invariant_SJoNA}, the first-order LLF scheme can be rewritten as
\begin{equation}\label{eq:1d_lo_decomp}
	\overline{\bU}^{\texttt{L}}_i = \left[1-\mu_i\left(\alpha_{\xl}+\alpha_{\xr}\right)\right] \overline{\bU}^{n}_i
	+ \mu_i\alpha_{\xl}\widetilde{\bU}_{\xl} + \mu_i\alpha_{\xr}\widetilde{\bU}_{\xr},
\end{equation}
with the first-order LLF intermediate states defined as
\begin{equation}\label{eq:1d_llf_inter_states}
		\widetilde{\bU}_{i\pm\frac12} := \frac12\left(\overline{\bU}^{n}_{i} + \overline{\bU}^{n}_{i\pm1}\right)
		\pm \frac{1}{2\alpha_{i\pm\frac12}}\left[\bF(\overline{\bU}^{n}_{i}) - \bF(\overline{\bU}^{n}_{i\pm1})\right].
\end{equation}
	The proofs of $\widetilde{\bU}_{i\pm\frac12} \in \mathcal{G}$ are similar to \Cref{sec:2d_llf_bp}, for the scalar case and Euler equations.


\begin{lemma}\label{lem:1d_llf_g}
	If the time step size $\Delta t^n$ satisfies
	\begin{equation}\label{eq:1d_convex_combination_dt}
		\Delta t^n \leqslant \dfrac{\Delta x_i}{\alpha_{\xl}+\alpha_{\xr}},
	\end{equation}
	then \cref{eq:1d_lo_decomp} is a convex combination,
	and the first-order LLF scheme is BP.
\end{lemma}
The proof (see e.g. \cite{Guermond_2016_Invariant_SJoNA,Perthame_1996_positivity_NM}) relies on $\overline{\bU}_i^n, \widetilde{\bU}_{i\pm\frac12} \in \mathcal{G}$ and the convexity of $\mathcal{G}$.

Upon defining the anti-diffusive flux $\Delta \widehat{\bF}_{i\pm\frac12} := \widehat{\bF}^{\texttt{H}}_{i\pm\frac12} - \widehat{\bF}^{\texttt{L}}_{i\pm\frac12}$ with $\widehat{\bF}^{\texttt{H}}_{i\pm\frac12} := \bF(\bU_{i\pm\frac12})$, a forward-Euler step applied to the semi-discrete high-order scheme for the cell average \eqref{eq:semi_av_1d} can be written as
\begin{align}
	&\overline{\bU}^{\texttt{H}}_i = \overline{\bU}^{n}_i - \mu_i (\widehat{\bF}^{\texttt{H}}_{\xr} - \widehat{\bF}^{\texttt{H}}_{\xl}) = \overline{\bU}^{n}_i - \mu_i (\widehat{\bF}^{\texttt{L}}_{\xr} - \widehat{\bF}^{\texttt{L}}_{\xl}) - \mu_i (\Delta \widehat{\bF}_{\xr} - \Delta \widehat{\bF}_{\xl}) \nonumber\\
	&~~~~ = \left[1-\mu_i\left(\alpha_{\xl}+\alpha_{\xr}\right)\right] \overline{\bU}^{n}_i
	+ \mu_i\alpha_{\xl} \widetilde{\bU}_{\xl}^{\texttt{H},+}
	+ \mu_i\alpha_{\xr} \widetilde{\bU}_{\xr}^{\texttt{H},-},\label{eq:1d_ho_decomp}\\
	&\widetilde{\bU}_{\xl}^{\texttt{H},+} := \left(\widetilde{\bU}_{\xl} + \frac{\Delta\widehat{\bF}_{\xl}}{\alpha_{\xl}}\right),\quad
	\widetilde{\bU}_{\xr}^{\texttt{H},-} := \left(\widetilde{\bU}_{\xr} - \frac{\Delta\widehat{\bF}_{\xr}}{\alpha_{\xr}}\right).\nonumber
\end{align}
With the low-order scheme \cref{eq:1d_lo_decomp} and high-order scheme \cref{eq:1d_ho_decomp} having the same abstract form, one can blend them to define the limited scheme for the cell average as
\begin{equation}\label{eq:1d_limited_decomp}
	\overline{\bU}^{\texttt{Lim}}_i = \left[1-\mu_i\left(\alpha_{\xl}+\alpha_{\xr}\right)\right] \overline{\bU}^{n}_i
	+ \mu_i\alpha_{\xl}\widetilde{\bU}_{\xl}^{\texttt{Lim},+}
	+ \mu_i\alpha_{\xr}\widetilde{\bU}_{\xr}^{\texttt{Lim},-},
\end{equation}
where the limited intermediate states are
\begin{equation}\label{eq:1d_limited_inter_states}
	\widetilde{\bU}_{i\pm\frac12}^{\texttt{Lim},\mp} = \widetilde{\bU}_{i\pm\frac12} \mp \frac{\Delta\widehat{\bF}^{\texttt{Lim}}_{i\pm\frac12}}{\alpha_{i\pm\frac12}}
	:= \widetilde{\bU}_{i\pm\frac12} \mp \frac{\theta_{i\pm\frac12}\Delta\widehat{\bF}_{i\pm\frac12}}{\alpha_{i\pm\frac12}},
\end{equation}
	and $\theta_{i\pm \frac12}\in[0,1]$ are the blending coefficients.
	The limited scheme \cref{eq:1d_limited_decomp} reduces to the first-order LLF scheme if $\theta_{i\pm\frac12}=0$,
	and recovers the high-order AF scheme \eqref{eq:semi_av_1d} when $\theta_{i\pm\frac12}=1$.

\subsubsection{Application to scalar conservation laws}
Similar to the 2D case, the convex limiting is applied to scalar conservation laws \eqref{eq:1d_scalar},
	such that the limited cell averages \cref{eq:1d_limited_decomp} satisfy the MP $u^{\min}_i \leqslant \bar{u}_i^{\texttt{Lim}} \leqslant u^{\max}_i$,
	where $u^{\min}_i = \min\mathcal{N}$, $u^{\max}_i = \max\mathcal{N}$,
	and $\mathcal{N}$ will be defined later.
	The limited anti-diffusive flux is
	\begin{equation*}
		\Delta\hat{f}^{\texttt{Lim}}_{\xr} = \begin{cases}
			\min\left\{\Delta\hat{f}_{\xr}, ~\alpha_{\xr}(\tilde{u}_{\xr}-u^{\min}_i),
			~\alpha_{\xr}(u^{\max}_{i+1}-\tilde{u}_{\xr}) \right\}, &\text{if}~ \Delta \hat{f}_{\xr} \geqslant 0, \\
			\max\left\{\Delta\hat{f}_{\xr}, ~\alpha_{\xr}(u^{\min}_{i+1}-\tilde{u}_{\xr}),
			~\alpha_{\xr}(\tilde{u}_{\xr}-u^{\max}_{i}) \right\}, &\text{otherwise}. \\
		\end{cases}
	\end{equation*}
	Finally, the limited numerical flux is
	\begin{equation}\label{eq:1d_flux_limited_scalar}
		\hat{f}^{\texttt{Lim}}_{\xr} = \hat{f}^{\texttt{L}}_{\xr} + \Delta\hat{f}^{\texttt{Lim}}_{\xr}.
	\end{equation}
	
	If considering the global MP, $\mathcal{N} = \bigcup_i\{\bar{u}_i^n, u_{\xr}^n\}$.
	For the local MP, one can choose $\mathcal{N} = \min\left\{\bar{u}_i^n, ~\tilde{u}_{\xl}, ~\tilde{u}_{\xr}, \bar{u}_{i-1}^n, \bar{u}_{i+1}^n\right\}$,
	which consists of the neighboring cell averages and intermediate states.

\subsubsection{Application to the compressible Euler equations}

This section aims at enforcing the positivity of density and pressure.
	To avoid the effect of the round-off error, we need to choose the desired lower bounds.
	Denote the lowest density and pressure in the domain by
	\begin{equation}\label{eq:lowest_rho_prs}
		\varepsilon^{\rho}:=\min_i\{\overline{\bU}_{i}^{n,\rho}, \bU_{\xr}^{n,\rho}\},~ \varepsilon^{p}:=\min_i\{p(\overline{\bU}_{i}^{n}), p(\bU_{\xr}^{n})\},
	\end{equation}
	where $\bU^{*,\rho}$ and $p(\bU^{*})$ denote the density component and pressure recovered from $\bU^{*}$, respectively.
	The limiting \cref{eq:1d_limited_inter_states} is feasible if the constraints are satisfied by the first-order LLF intermediate states \cref{eq:1d_llf_inter_states}, thus the lower bounds can be defined as
	\begin{equation*}
		\varepsilon_{i}^\rho := \min\{10^{-13}, \varepsilon^{\rho}, \widetilde{\bU}_{\xl}^{\rho}, \widetilde{\bU}_{\xr}^{\rho}\},~
		\varepsilon_{i}^p := \min\{10^{-13}, \varepsilon^{p}, p(\widetilde{\bU}_{\xl}), p(\widetilde{\bU}_{\xr})\}.
	\end{equation*}

i) {\bfseries Positivity of density.}
The first step is to impose the density positivity
	$\widetilde{\bU}_{\xr}^{\texttt{Lim},\pm,\rho} \geqslant \bar{\varepsilon}_{\xr}^\rho:=\min\{\varepsilon_{i}^\rho, \varepsilon_{i+1}^\rho\} $.
	Similarly to the derivation of the scalar case,
the corresponding density component of the limited anti-diffusive flux is
\begin{equation*}
	\Delta\widehat{\bF}^{\texttt{Lim},*,\rho}_{\xr} = \begin{cases}
		\min\left\{\Delta\widehat{\bF}^{\rho}_{\xr}, ~\alpha_{\xr}\left(\widetilde{\bU}^{\rho}_{\xr}-\bar{\varepsilon}^\rho_{\xr}\right) \right\}, &\text{if}~ \Delta \widehat{\bF}^{\rho}_{\xr} \geqslant 0, \\
		\max\left\{\Delta\widehat{\bF}^{\rho}_{\xr}, ~\alpha_{\xr}\left(\bar{\varepsilon}^\rho_{\xr}-\widetilde{\bU}^{\rho}_{\xr}\right) \right\}, &\text{otherwise}. \\
	\end{cases}
\end{equation*}
Then the density component of the limited flux is $\widehat{\bF}_{\xr}^{\texttt{Lim}, *, \rho} = \widehat{\bF}_{\xr}^{\texttt{L}, \rho} + \Delta\widehat{\bF}_{\xr}^{\texttt{Lim}, *, \rho}$,
with the other components remaining the same as $\widehat{\bF}_{\xr}^{\texttt{H}}$.

ii) {\bfseries Positivity of pressure.}
The second step is to enforce pressure positivity
	$p(\widetilde{\bU}_{\xr}^{\texttt{Lim},\pm}) \geqslant \bar{\varepsilon}_{\xr}^p:=\min\{\varepsilon_{i}^p, \varepsilon_{i+1}^p\} $.
Since
\begin{equation*}
	\widetilde{\bU}_{\xr}^{\texttt{Lim},\pm} = \widetilde{\bU}_{\xr} \pm \frac{\theta_{\xr}\Delta \widehat{\bF}_{\xr}^{\texttt{Lim}, *}}{\alpha_{\xr}},\quad
	\Delta \widehat{\bF}_{\xr}^{\texttt{Lim}, *} = \widehat{\bF}_{\xr}^{\texttt{Lim}, *} - \widehat{\bF}_{\xr}^{\texttt{L}},
\end{equation*}
the constraints lead to two inequalities
\begin{equation}\label{eq:1d_pressure_inequality}
	A_{\xr}\theta^2_{\xr} \pm B_{\xr}\theta_{\xr} \leqslant C_{\xr},
\end{equation}
with the coefficients
\begin{align*}
	&A_{\xr} = \dfrac{1}{2} \left(\Delta \widehat\bF^{\texttt{Lim}, *, \rho v}_{\xr}\right)^2
	- \Delta \widehat\bF^{\texttt{Lim}, *, \rho}_{\xr} \Delta \widehat\bF^{\texttt{Lim}, *, E}_{\xr}, \\
	&B_{\xr} = \alpha_{\xr}\left(\Delta \widehat\bF^{\texttt{Lim}, *, \rho}_{\xr} \widetilde\bU^{E}_{\xr}
	+ \widetilde\bU^{\rho}_{\xr} \Delta \widehat\bF^{\texttt{Lim}, *, E}_{\xr}
	- \Delta \widehat\bF^{\texttt{Lim}, *, \rho v}_{\xr} \widetilde\bU^{\rho v}_{\xr}
	- \tilde{\varepsilon} \Delta \widehat\bF^{\texttt{Lim}, *, \rho}_{\xr}\right), \\
	&C_{\xr} = \alpha_{\xr}^2\left(\widetilde\bU^{\rho}_{\xr} \widetilde\bU^{E}_{\xr}
	- \dfrac{1}{2} \left(\widetilde\bU^{\rho v}_{\xr}\right)^2
	- \tilde{\varepsilon}\widetilde\bU^{\rho}_{\xr}\right), \quad
	\tilde{\varepsilon} = \bar{\varepsilon}^p_{\xr}/(\gamma-1).
\end{align*}
Following \cite{Kuzmin_2020_Monolithic_CMiAMaE}, the inequalities \cref{eq:1d_pressure_inequality} hold under the linear sufficient condition
\begin{equation*}
	\left(\max\left\{0, A_{\xr}\right\} + \left|B_{\xr}\right|\right)\theta_{\xr}\leqslant C_{\xr},
\end{equation*}
if making use of $\theta_{\xr}^2 \leqslant \theta_{\xr},~ \theta_{\xr}\in[0,1]$.
Thus the coefficient can be chosen as
\begin{equation*}
	\theta_{\xr} = \min\left\{1,~ \frac{C_{\xr}}{\max\{0, A_{\xr}\} + \abs{B_{\xr}}}\right\},
\end{equation*}
and the final limited numerical flux is
\begin{equation}\label{eq:1d_flux_limited_Euler}
	\widehat{\bF}_{\xr}^{\texttt{Lim},**} = \widehat{\bF}_{\xr}^{\texttt{L}} + \theta_{\xr}\Delta\widehat{\bF}_{\xr}^{\texttt{Lim}, *}.
\end{equation}

	\subsubsection{Shock sensor-based limiting}
	In 1D, the Jameson's shock sensor \cite{Jameson_1981_Solutions_AJ} is
	\begin{equation*}
		(\varphi_1)_{i} = \dfrac{\abs{\bar{p}_{i+1} - 2\bar{p}_{i} + \bar{p}_{i-1}}}{\abs{\bar{p}_{i+1} + 2\bar{p}_{i} + \bar{p}_{i-1}}},
	\end{equation*}
	and the modified Ducros' shock sensor reduced from the 2D case \cite{Ducros_1999_Large_JoCP} is
	\begin{equation*}
		(\varphi_2)_{i} = \max\left\{-\dfrac{\bar{v}_{i+1}-\bar{v}_{i-1}}{\abs{\bar{v}_{i+1}-\bar{v}_{i-1}}+10^{-40}}, ~0\right\}.
	\end{equation*}
	Note that $\bar{v}_i$ and $\bar{p}_i$ are the velocity and pressure recovered from the cell average $\overline{\bU}_i$.
	The blending coefficient is designed as
	\begin{align*}
		&\theta_{\xr}^{s} = \exp(-\kappa (\varphi_1)_{\xr} (\varphi_2)_{\xr})\in (0, 1],\\
		&(\varphi_s)_{\xr} = \max\left\{(\varphi_s)_{i}, (\varphi_s)_{i+1}\right\}, ~s=1,2,
	\end{align*}
	where the problem-dependent parameter $\kappa$ adjusts the strength of the limiting, and its optimal choice needs further investigation.
	The final limited numerical flux is
	\begin{equation}\label{eq:1d_flux_limited_Euler_ss}
		\widehat{\bF}_{\xr}^{\texttt{Lim}} = \widehat{\bF}_{\xr}^{\texttt{L}} + \theta_{\xr}^{s}\Delta\widehat{\bF}_{\xr}^{\texttt{Lim}, **},
	\end{equation}
	with $\Delta\widehat{\bF}_{\xr}^{\texttt{Lim},**} = \widehat{\bF}_{\xr}^{\texttt{Lim},**} - \widehat{\bF}_{\xr}^{\texttt{L}}$,
	and $\widehat{\bF}_{\xr}^{\texttt{Lim},**}$ given in \cref{eq:1d_flux_limited_Euler}.

\subsection{Scaling limiter for point value}\label{sec:1d_limiting_point}
A first-order LLF scheme for the point value update can be written as
\begin{equation}\label{eq:1d_llf_pnt}
	\bU_{\xr}^{\texttt{L}} =  \bU_{\xr}^{n} - \dfrac{2\Delta t^n}{\Delta x_i + \Delta x_{i+1}}\left(\widehat{\bF}^{\texttt{L}}_{i+1}( \bU_{i+\frac12}^n,  \bU_{i+\frac32}^n)
	- \widehat{\bF}^{\texttt{L}}_{i}( \bU_{i-\frac12}^n,  \bU_{i+\frac12}^n)\right),
\end{equation}
with the numerical flux
\begin{align*}
	&\widehat{\bF}^{\texttt{L}}_{i} = \widehat{\bF}^{\texttt{LLF}}( \bU_{i-\frac12}^n,  \bU_{i+\frac12}^n) =
	\frac12\left( \bF(\bU_{i-\frac12}^n) +  \bF(\bU_{i+\frac12}^n)\right) - \frac{\alpha_{i}}{2}\left(\bU_{i+\frac12}^n - \bU_{i-\frac12}^n\right),\\
	&\alpha_{i} = \max\{\varrho(\bU_{i-\frac12}^n), ~\varrho(\bU_{i+\frac12}^n)\}.
\end{align*}
Similarly to \Cref{lem:1d_llf_g}, it is straightforward to obtain the following Lemma.

\begin{lemma}
	The LLF scheme for the point value \cref{eq:1d_llf_pnt} is BP under the CFL condition
	\begin{equation}\label{eq:1d_pnt_llf_dt}
		\Delta t^n \leqslant \dfrac{\Delta x_i + \Delta x_{i+1}}{2  (\alpha_i + \alpha_{i+1}) } .
	\end{equation}
\end{lemma}

The limited solution is obtained by blending the high-order AF scheme \eqref{eq:semi_pnt_1d} with the forward-Euler scheme and the LLF scheme \eqref{eq:1d_llf_pnt} as $\bU_{\xr}^{\texttt{Lim}} = \theta_{\xr} \bU_{\xr}^{\texttt{H}} + (1-\theta_{\xr}) \bU_{\xr}^{\texttt{L}}$,
such that $\bU_{\xr}^{\texttt{Lim}}\in\mathcal{G}$.

\subsubsection{Application to scalar conservation laws}
	This section enforces the MP $u_{\xr}^{\min} \leqslant u_{\xr}^{\texttt{Lim}} \leqslant u_{\xr}^{\max}$ using the scaling limiter \cite{Zhang_2010_maximum_JoCP}.
	The limited solution is
	\begin{equation}\label{eq:1d_pnt_limited_state_scalar}
		u_{\xr}^{\texttt{Lim}} = \theta_{\xr} u_{\xr}^{\texttt{H}} + \left(1-\theta_{\xr}\right) u_{\xr}^{\texttt{L}},
	\end{equation}
	with the coefficient
	\begin{equation*}
		\theta_{\xr} = \min\left\{1,~ \left|\dfrac{u_{\xr}^{\texttt{L}}-u_{\xr}^{\min}}{u_{\xr}^{\texttt{L}}-u_{\xr}^{\texttt{H}}}\right|,~
		\left|\dfrac{u_{\xr}^{\max}-u_{\xr}^{\texttt{L}}}{u_{\xr}^{\texttt{H}}-u_{\xr}^{\texttt{L}}}\right| \right\}.
	\end{equation*}
%

	The bounds are determined by $u^{\min}_{\xr} = \min \mathcal{N}$, $u^{\max}_{\xr} = \max \mathcal{N}$,
	where the set $\mathcal{N}$ consists of all the DoFs in the domain,
	i.e., $\mathcal{N} = \bigcup_i\{\bar{u}_i^n, u_{\xr}^n\}$ for the global MP.
	One can also consider the local MP, e.g., $\mathcal{N} = \left\{u_{\xl}^n, u_{\xr}^n, u_{i+\frac32}^n\right\}$,
	which at least includes all the DoFs appeared in the first-order LLF scheme \cref{eq:1d_llf_pnt}.

\subsubsection{Application to the compressible Euler equations}\label{sec:1d_limiting_point_euler}
The limiting consists of two steps.

{\bfseries i) Positivity of density.}
First, the high-order solution $\bU_{\xr}^{\texttt{H}}$ is modified as $\bU_{\xr}^{\texttt{Lim}, *}$,
such that $\bU_{\xr}^{\texttt{Lim}, *, \rho} \geqslant \varepsilon^\rho_{\xr}:=\min\{10^{-13}, \varepsilon^{\rho}, \bU_{\xr}^{\texttt{L},\rho}\}$ with $\varepsilon^{\rho}$ given in \cref{eq:lowest_rho_prs}.
Solving the inequality yields
\begin{equation*}
	\theta^{*}_{\xr} = \begin{cases}
		\dfrac{\bU_{\xr}^{\texttt{L}, \rho}-\varepsilon^\rho_{\xr}}{\bU_{\xr}^{\texttt{L}, \rho}-\bU_{\xr}^{\texttt{H}, \rho}},&\text{if}~~ \bU_{\xr}^{\texttt{H}, \rho} < \varepsilon^\rho_{\xr}, \\
		1, &\text{otherwise}. \\
	\end{cases}
\end{equation*}
Then the density component of the limited solution is $\bU_{\xr}^{\texttt{Lim}, *, \rho} = \theta^{*}_{\xr} \bU_{\xr}^{\texttt{H}, \rho} + (1-\theta^{*}_{\xr}) \bU_{\xr}^{\texttt{L}, \rho}$,
with the other components remaining the same as $\bU_{\xr}^{\texttt{H}}$.

{\bfseries ii) Positivity of pressure.}
Then the limited solution $\bU_{\xr}^{\texttt{Lim}, *}$ is modified as $\bU_{\xr}^{\texttt{Lim}}$,
such that it gives positive pressure, i.e., $p(\bU_{\xr}^{\texttt{Lim}}) \geqslant \varepsilon^p_{\xr}:=\min\{10^{-13}, \varepsilon^{p}, p(\bU_{\xr}^{\texttt{L}})\}$, with $\varepsilon^{p}$ given in \cref{eq:lowest_rho_prs}.
Let the final limited solution be
\begin{equation}\label{eq:1d_pnt_limited_state_Euler}
	\bU_{\xr}^{\texttt{Lim}} = \theta^{**}_{\xr} \bU_{\xr}^{\texttt{Lim}, *} + \left(1-\theta^{**}_{\xr}\right) \bU_{\xr}^{\texttt{L}}.
\end{equation}
The pressure is a concave function of the conservative variables (see e.g. \cite{Zhang_2011_Maximum_PotRSAMPaES}), so that 
$p(\bU_{\xr}^{\texttt{Lim}}) \geqslant \theta^{**}_{\xr}p(\bU_{\xr}^{\texttt{Lim}, *}) + \left(1-\theta^{**}_{\xr}\right)p(\bU_{\xr}^{\texttt{L}})$
based on Jensen's inequality and $\bU_{\xr}^{\texttt{Lim}, *, \rho} > 0$, $\bU_{\xr}^{\texttt{L}, \rho} > 0$, $\theta_{\xr}^{**} \in [0,1]$.
Thus the coefficient can be chosen as
\begin{equation*}
	\theta^{**}_{\xr} = \begin{cases}
		\dfrac{p(\bU_{\xr}^{\texttt{L}}) - \varepsilon^p_{\xr}}{p(\bU_{\xr}^{\texttt{L}}) - p(\bU_{\xr}^{\texttt{Lim}, *})},&\text{if}~~ p(\bU_{\xr}^{\texttt{Lim}, *}) < \varepsilon^p_{\xr}, \\
		1, &\text{otherwise}. \\
	\end{cases}
\end{equation*}
		
\begin{theorem}
	If the initial numerical solution $\overline{\bU}_i^0, \bU_{\xr}^0\in\mathcal{G}$ for all $i$,
	and the time step size satisfies \cref{eq:1d_convex_combination_dt} and \cref{eq:1d_pnt_llf_dt},
	then the AF methods \eqref{eq:semi_av_1d}-\eqref{eq:semi_pnt_1d} equipped with the SSP-RK3 \eqref{eq:ssp_rk3} and the BP limitings
	\begin{itemize}[leftmargin=*]
		\item \cref{eq:1d_flux_limited_scalar} and \cref{eq:1d_pnt_limited_state_scalar} preserve the maximum principle for scalar case;
		\item \cref{eq:1d_flux_limited_Euler} and \cref{eq:1d_pnt_limited_state_Euler} preserve positive density and pressure for the Euler equations.
	\end{itemize}
\end{theorem}

\begin{remark}
		For uniform meshes, and if taking the maximal spectral radius of $\partial\bF/\partial\bU$ in the domain as $\norm{\varrho}_\infty$, the following CFL condition
		\begin{equation*}
			\Delta t^n \leqslant \frac{\Delta x}{2\norm{\varrho}_\infty}
		\end{equation*}
		fulfills the time step size constraints \eqref{eq:1d_convex_combination_dt} and \eqref{eq:1d_pnt_llf_dt}.
\end{remark}


\section{Additional numerical results}\label{sec:results_supp}

\begin{example}[1D accuracy test for the Euler equations]\label{ex:1d_accuracy}
		This test is used to examine the accuracy of using different point value updates, following the setup in \cite{Abgrall_2023_Combination_CoAMaC}.
		The domain is $[-1,1]$ with periodic boundary conditions.
		The adiabatic index is chosen as $\gamma=3$ so that the characteristic equations of two Riemann invariants $w=u\pm a$ are $w_t+ww_x=0$.
		The initial condition is $\rho_0(x) = 1+\zeta \sin(\pi x), v_0 = 0, p_0 = \rho_0^\gamma$
		and $\zeta\in(0, 1)$ controls the range of the density.
		The exact solution can be obtained by the method of characteristics, given by $\rho(x,t)=\frac12\left(\rho_0(x_1) + \rho_0(x_2)\right), v(x,t)=\sqrt{3}\left(\rho(x,t)-\rho_0(x_1)\right)$,
		where $x_1$ and $x_2$ are solved from the nonlinear equations $x+\sqrt{3}\rho_0(x_1)t-x_1=0$, $x-\sqrt{3}\rho_0(x_2)t-x_2=0$.
		The problem is solved until $T = 0.1$ with $\zeta=1-10^{-7}$.

As $\zeta=1-10^{-7}$, the minimum density and pressure are $10^{-7}$ and $10^{-21}$ respectively, so that the BP limitings are necessary to run this test case.
	The maximal CFL numbers allowing stable simulations are obtained experimentally, which are around $0.47, 0.43, 0.32, 0.18$ for the JS, LLF, SW, and VH FVS, respectively,
thus we run the test with the same CFL number as $0.18$.
\Cref{fig:1d_euler_accuracy} shows the errors and corresponding convergence rates for the conservative variables in the $\ell^1$ norm.
	It is seen that the JS and all the FVS except for the SW FVS achieve the designed third-order accuracy, showing that our BP limitings do not affect the high-order accuracy.
	To examine the reason why the scheme based on the SW FVS is only second-order accurate,
\Cref{fig:1d_euler_accuracy_test1} plots the density and velocity profiles obtained using the SW FVS with $80$ cells.
One can observe some defects in the density when the velocity is zero,
similar to the ``sonic point glitch'' in the literature \cite{Tang_2005_sonic_JoCP}.
	One possible reason is that the SW FVS is based on the absolute value of the eigenvalues, and the corresponding mass flux is not differentiable when the velocity is zero \cite{Leer_1982_Flux_InProceedings}.
Such an issue remains to be further explored in the future.

\begin{figure}[htbp]
	\centering
	\includegraphics[width=0.5\linewidth]{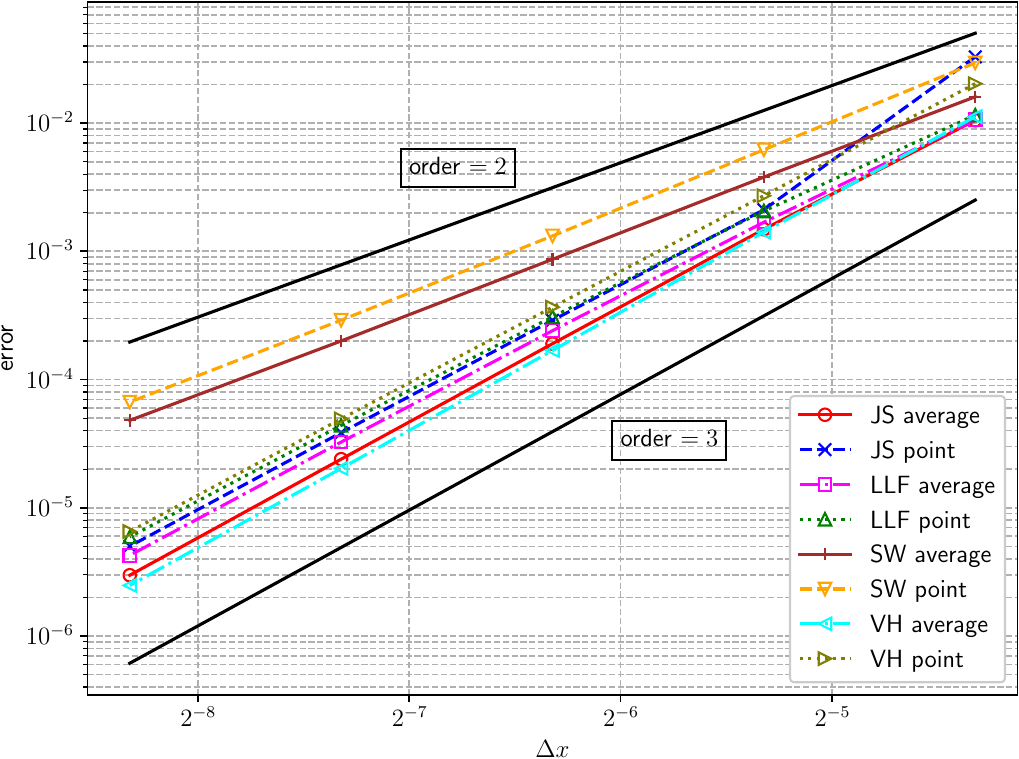}
	\caption{\Cref{ex:1d_accuracy}, the accuracy test for the 1D Euler equations.
 The BP limitings are necessary.
	}
	\label{fig:1d_euler_accuracy}
\end{figure}

\begin{figure}[htbp]
	\centering
	\begin{subfigure}[b]{0.35\textwidth}
		\centering
		\includegraphics[width=1.0\linewidth]{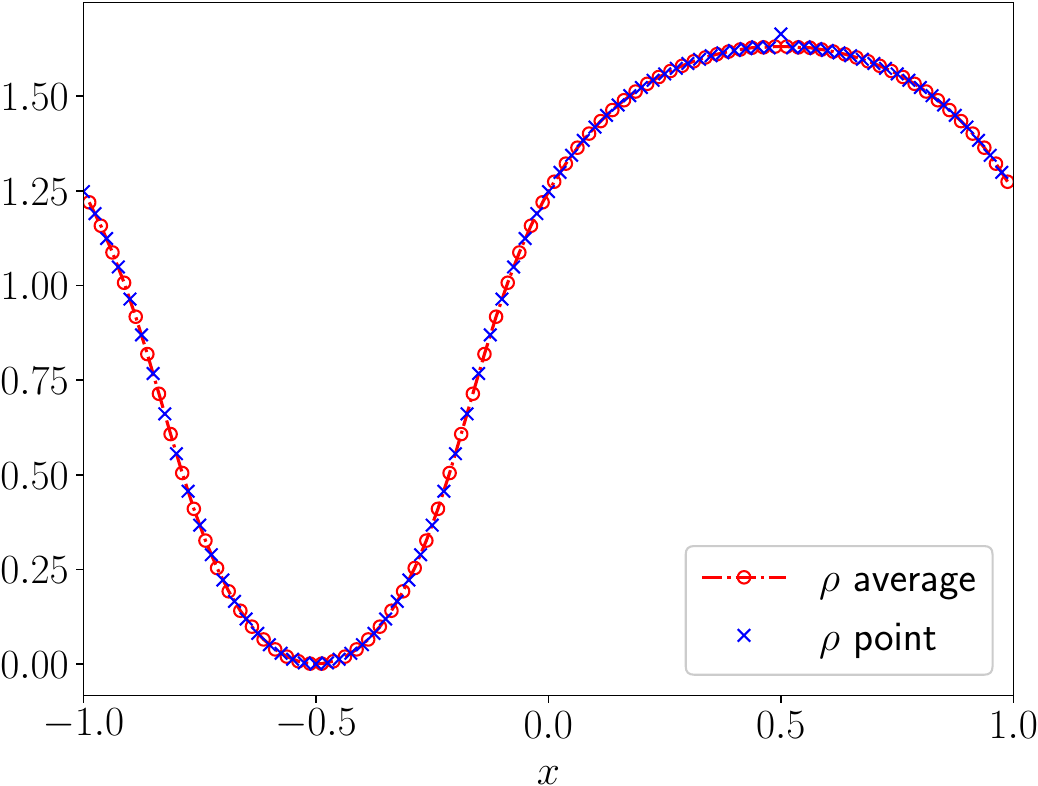}
	\end{subfigure}
	~
	\begin{subfigure}[b]{0.35\textwidth}
		\centering
		\includegraphics[width=1.0\linewidth]{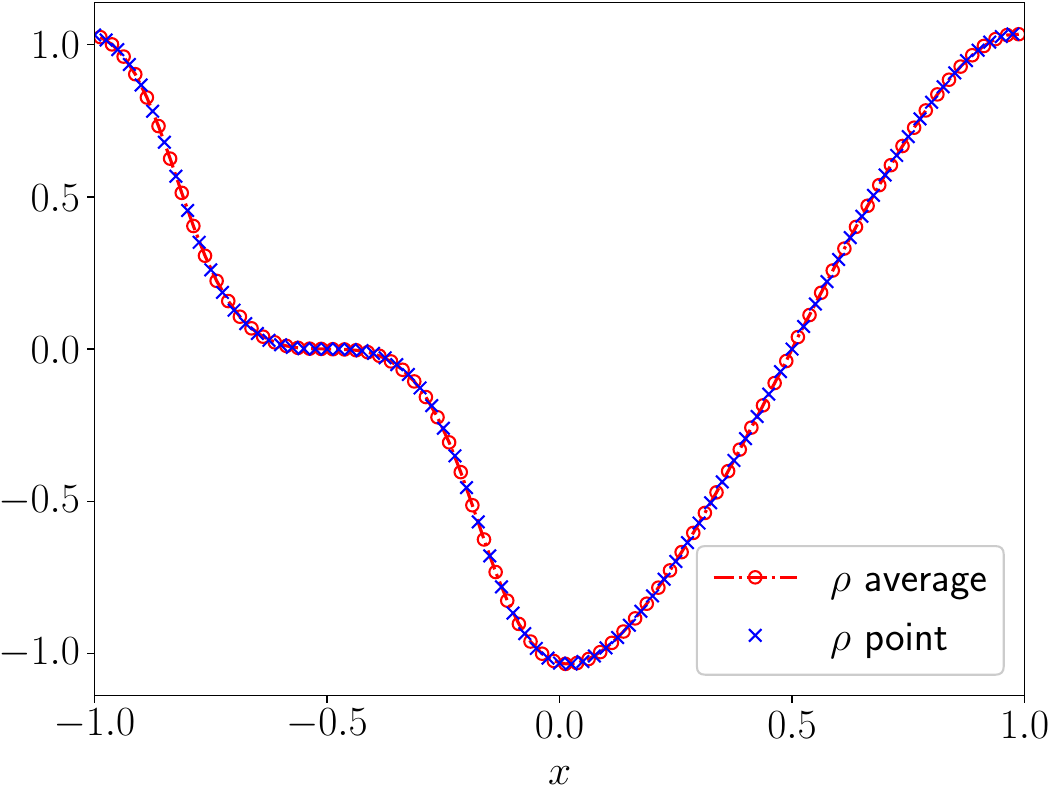}
	\end{subfigure}
	\caption{\Cref{ex:1d_accuracy}, the density (left) and velocity (right) obtained with the SW FVS and $80$ cells for the 1D Euler equations.
	}
	\label{fig:1d_euler_accuracy_test1}
\end{figure}
\end{example}

\begin{example}[Double rarefaction problem]\label{ex:1d_double_rarefaction}
	The exact solution to this problem contains a vacuum,
	so that it is often used to verify the BP property of numerical methods.
	The test is solved on a domain $[0,1]$ until $T=0.3$ with the initial data	\begin{equation*}
		(\rho, v, p) = \begin{cases}
			(7, -1, 0.2), &\text{if}~~ x<0.5,\\
			(7, 1, 0.2), &\text{otherwise}.\\
		\end{cases}
	\end{equation*}
	
	In this test, the AF method based on any kind of point value update mentioned in this paper gives negative density or pressure without the BP limitings.
	\Cref{fig:1d_double_rarefaction} shows the density computed with $400$ cells and the BP limitings for the cell average and point value updates.
	The CFL number is $0.4$ for all kinds of point value updates,
	except for $0.1$ for the VH FVS.
	One observes that the BP AF method gets good performance for this example.
	
	\begin{figure}[htbp]
		\centering
		\begin{subfigure}[b]{0.24\textwidth}
			\centering
			\includegraphics[width=\linewidth]{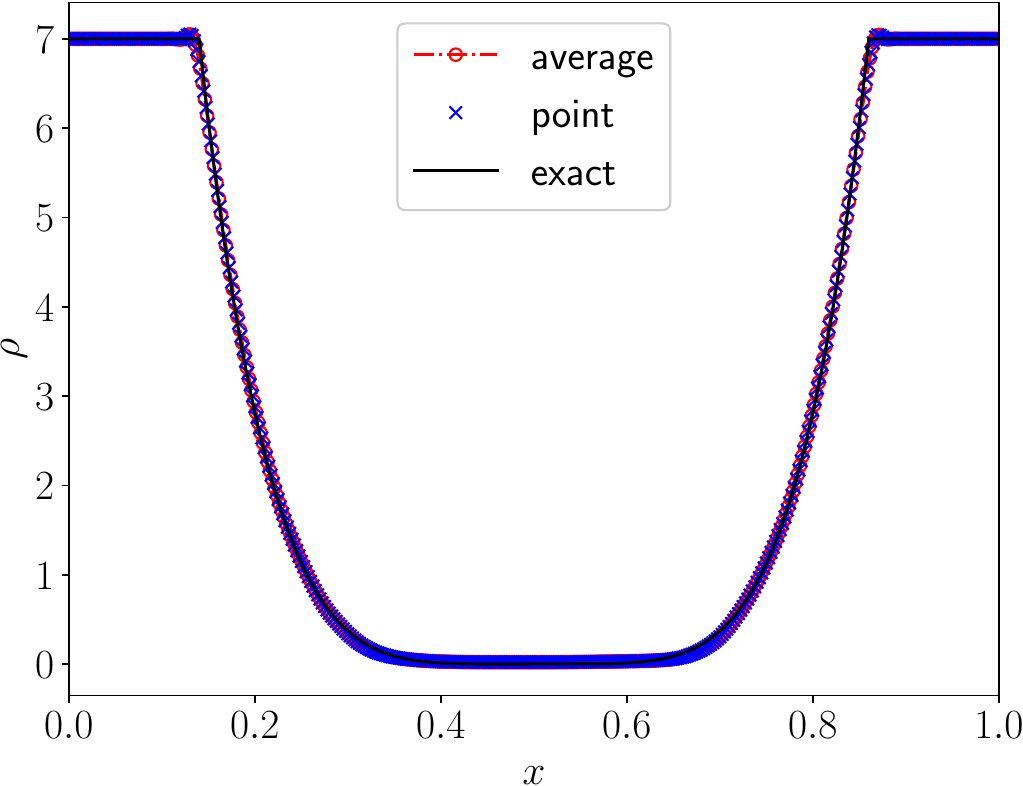}
		\end{subfigure}
		\begin{subfigure}[b]{0.24\textwidth}
			\centering
			\includegraphics[width=\linewidth]{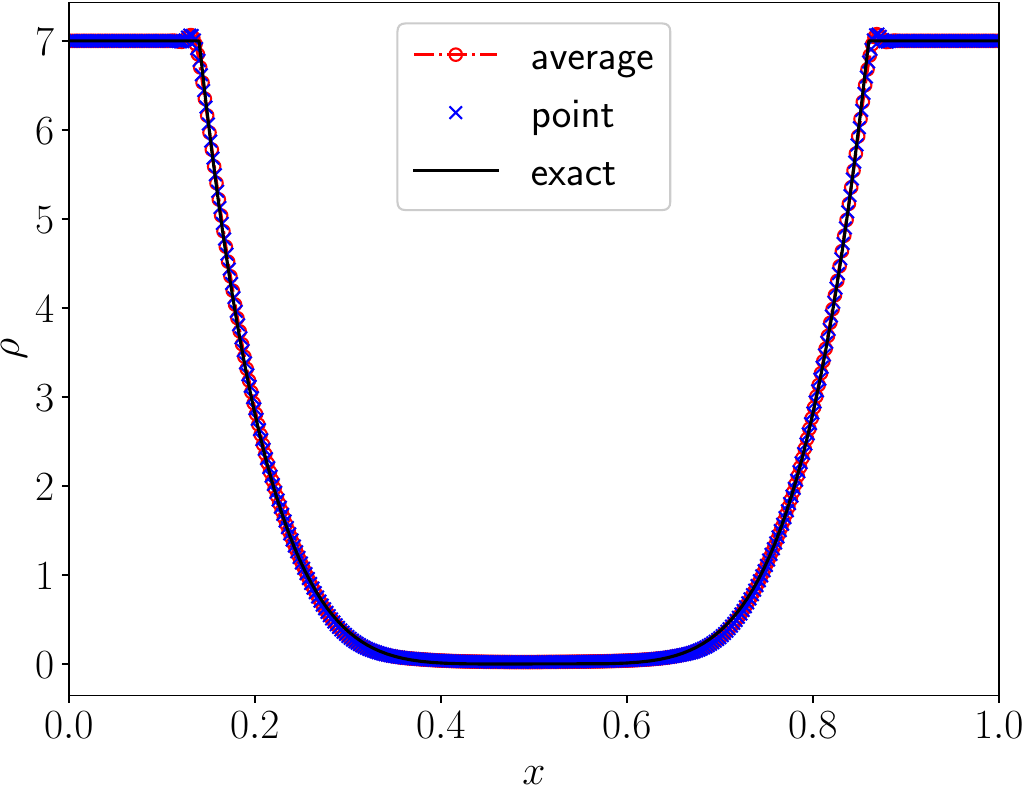}
		\end{subfigure}
		\begin{subfigure}[b]{0.24\textwidth}
			\centering
			\includegraphics[width=\linewidth]{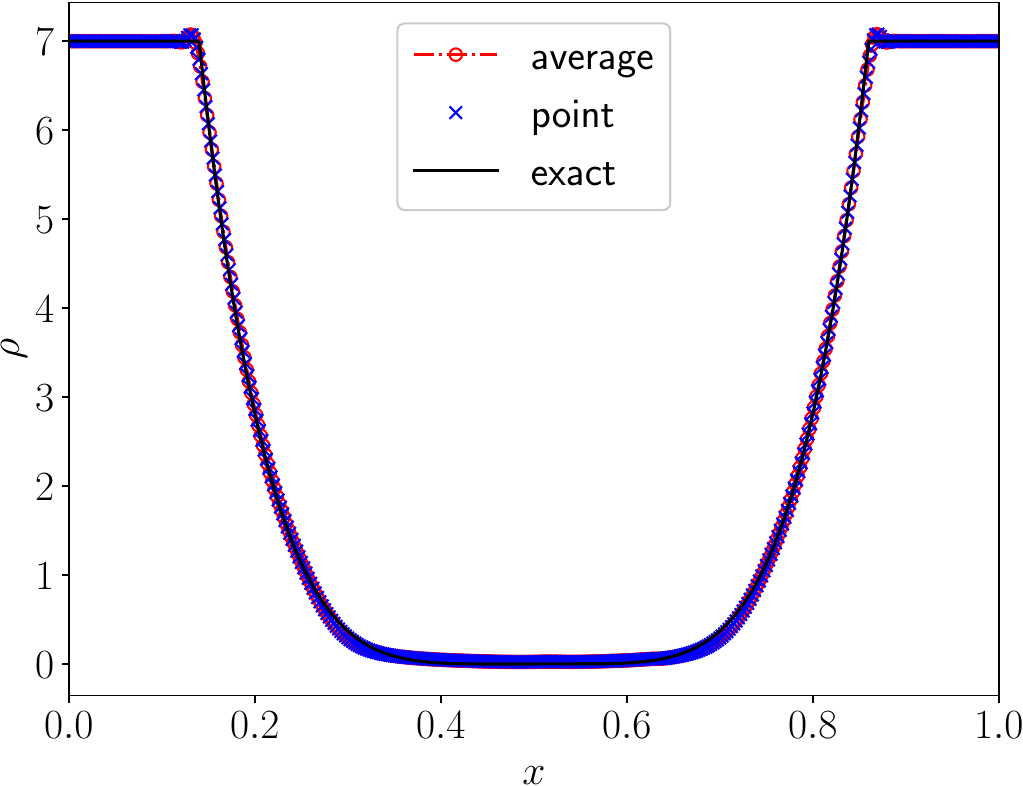}
		\end{subfigure}
		\begin{subfigure}[b]{0.24\textwidth}
			\centering
			\includegraphics[width=\linewidth]{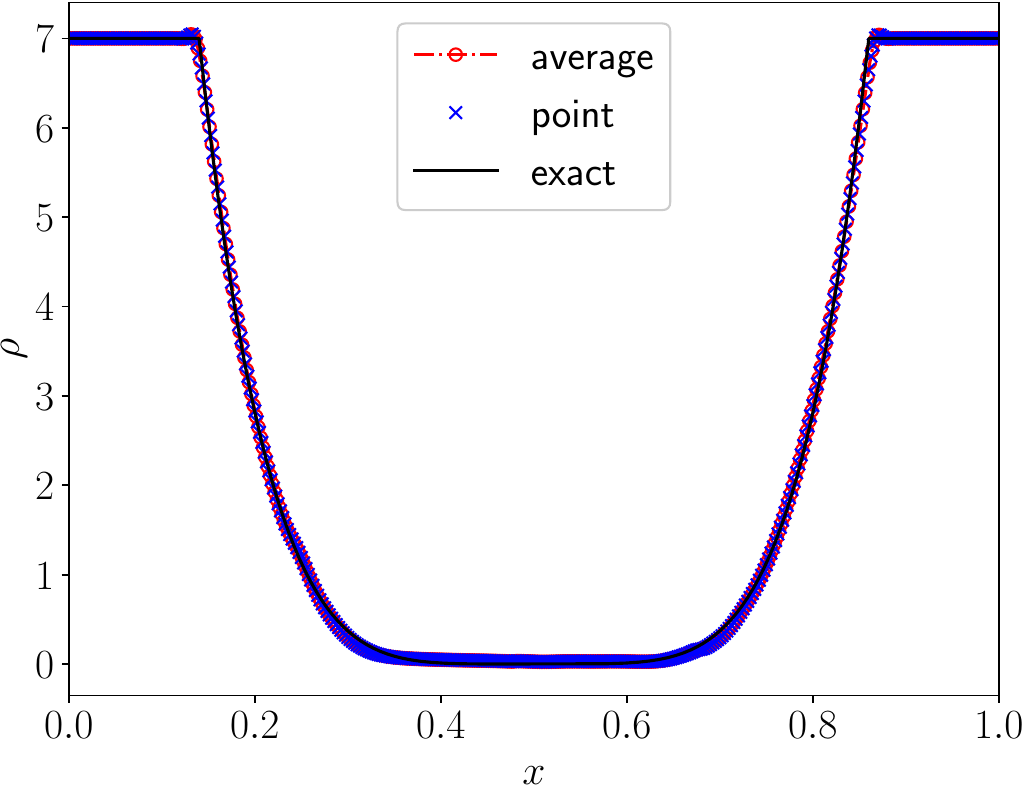}
		\end{subfigure}

	\begin{subfigure}[b]{0.24\textwidth}
		\centering
		\includegraphics[width=\linewidth]{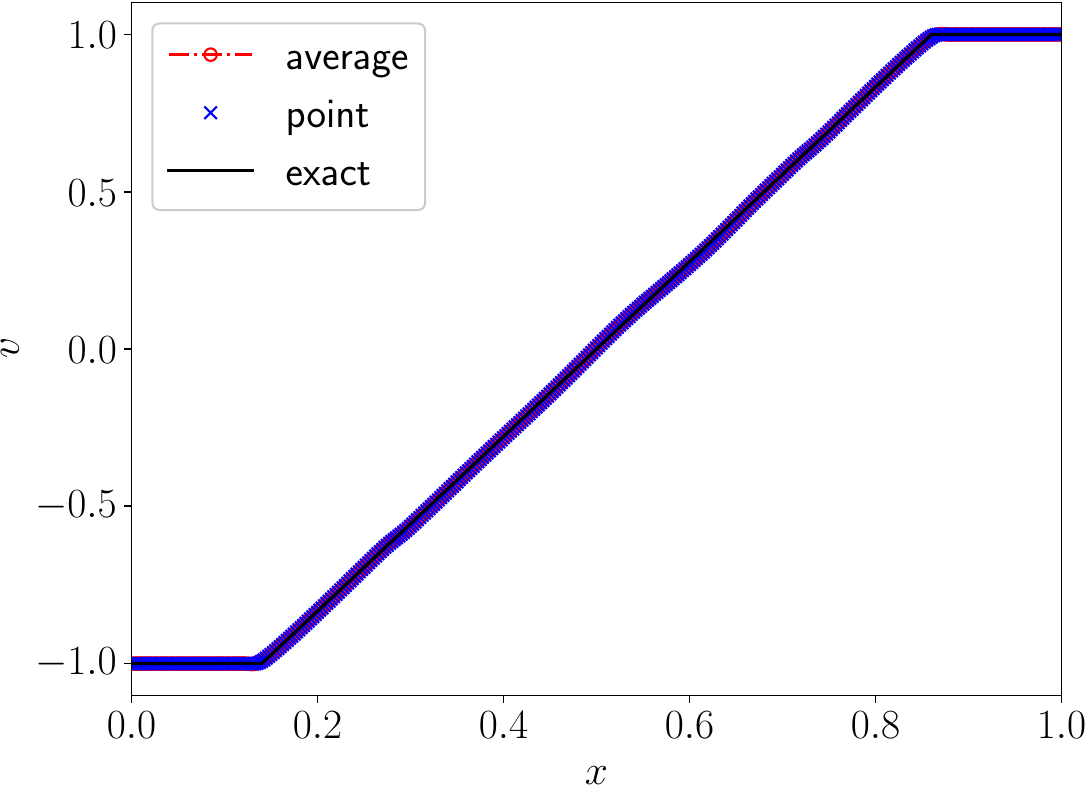}
	\end{subfigure}
	\begin{subfigure}[b]{0.24\textwidth}
		\centering
		\includegraphics[width=\linewidth]{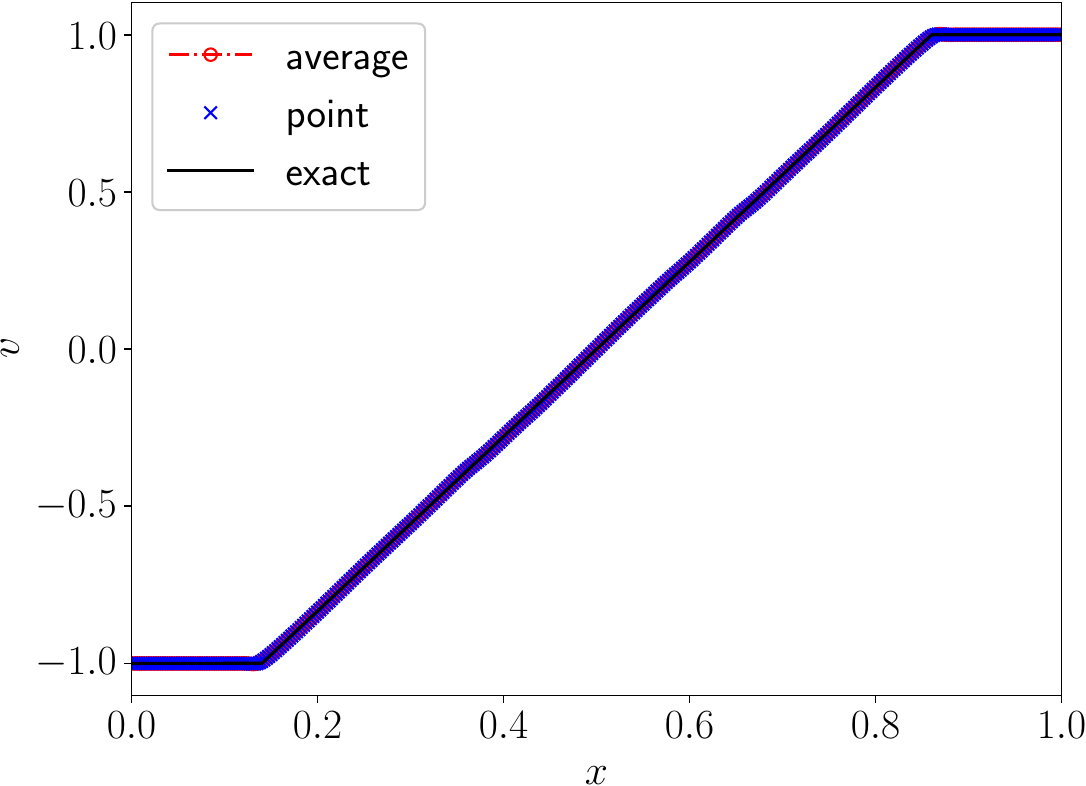}
	\end{subfigure}
	\begin{subfigure}[b]{0.24\textwidth}
		\centering
		\includegraphics[width=\linewidth]{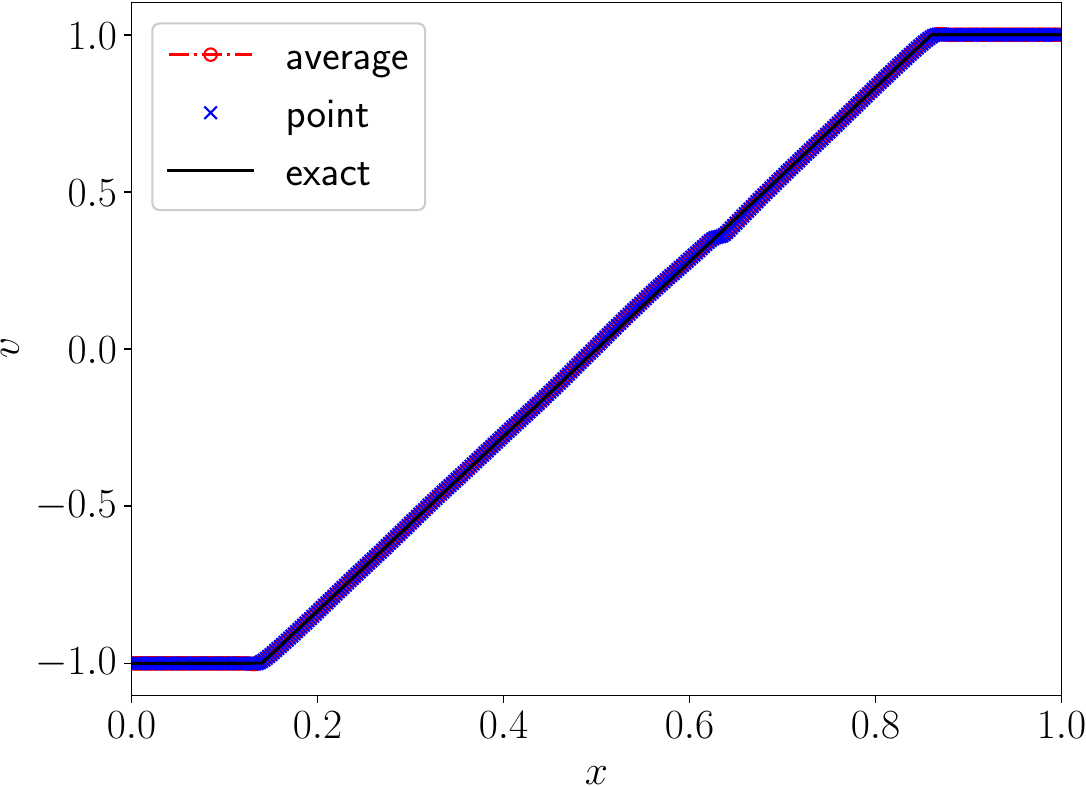}
	\end{subfigure}
	\begin{subfigure}[b]{0.24\textwidth}
		\centering
		\includegraphics[width=\linewidth]{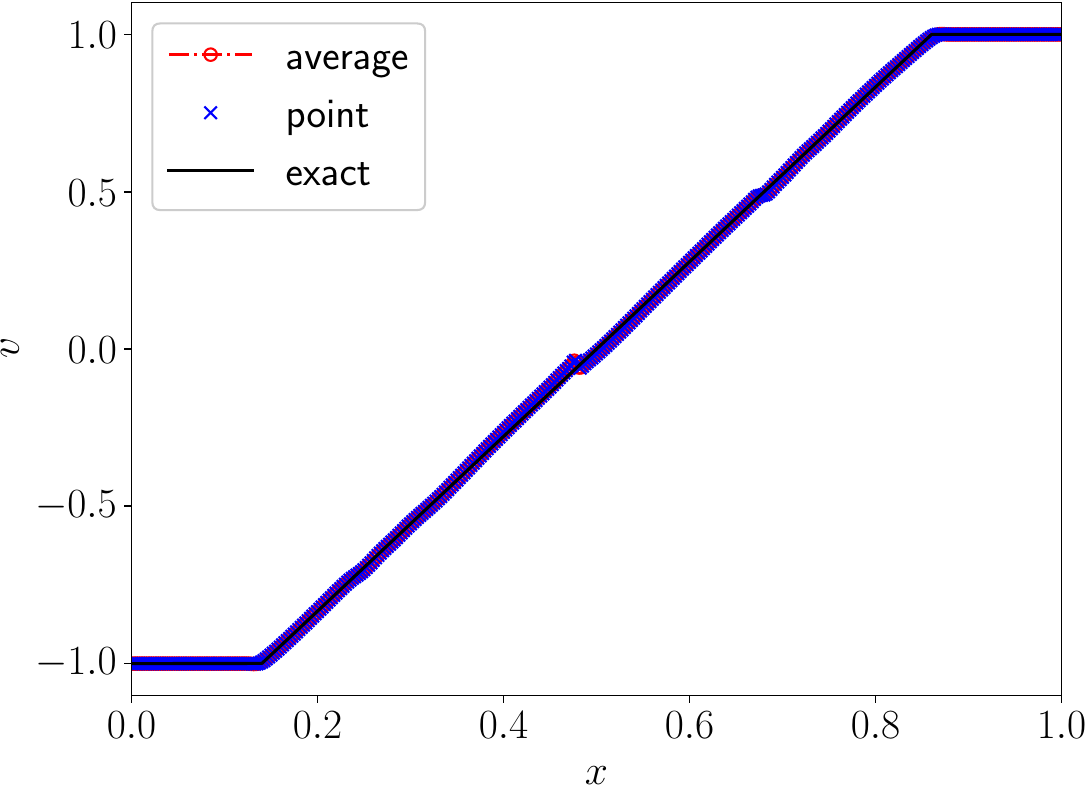}
	\end{subfigure}
	
	\begin{subfigure}[b]{0.24\textwidth}
		\centering
		\includegraphics[width=\linewidth]{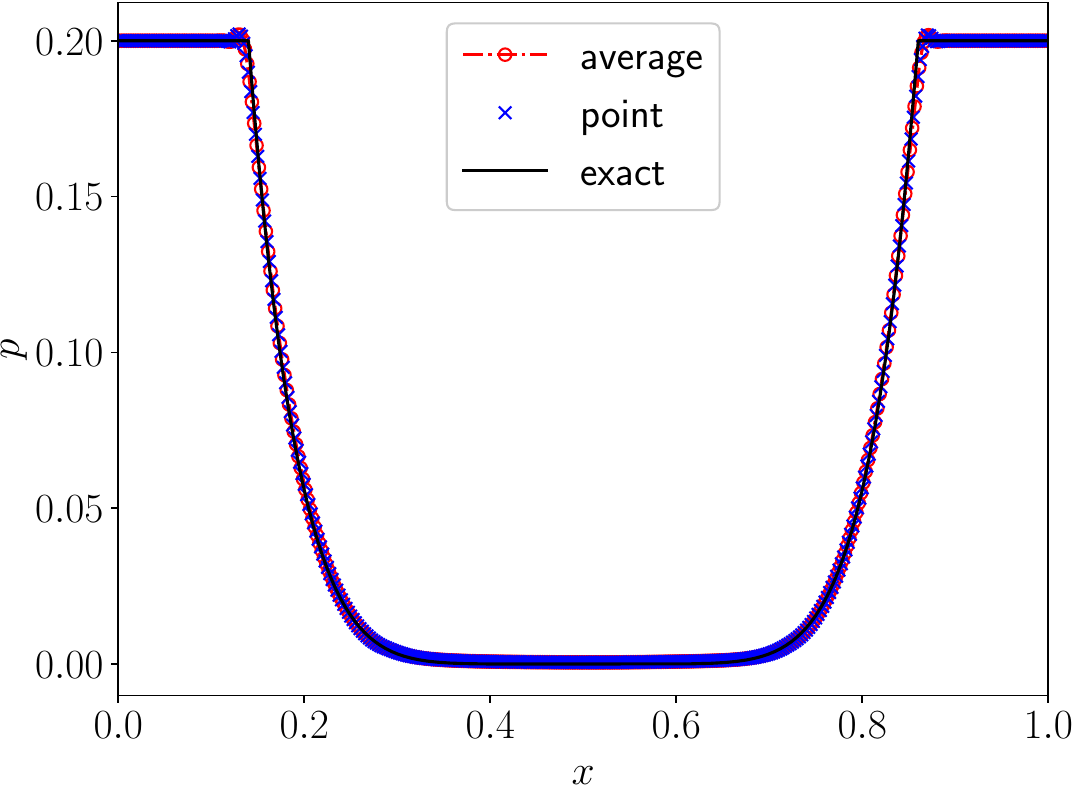}
	\end{subfigure}
	\begin{subfigure}[b]{0.24\textwidth}
		\centering
		\includegraphics[width=\linewidth]{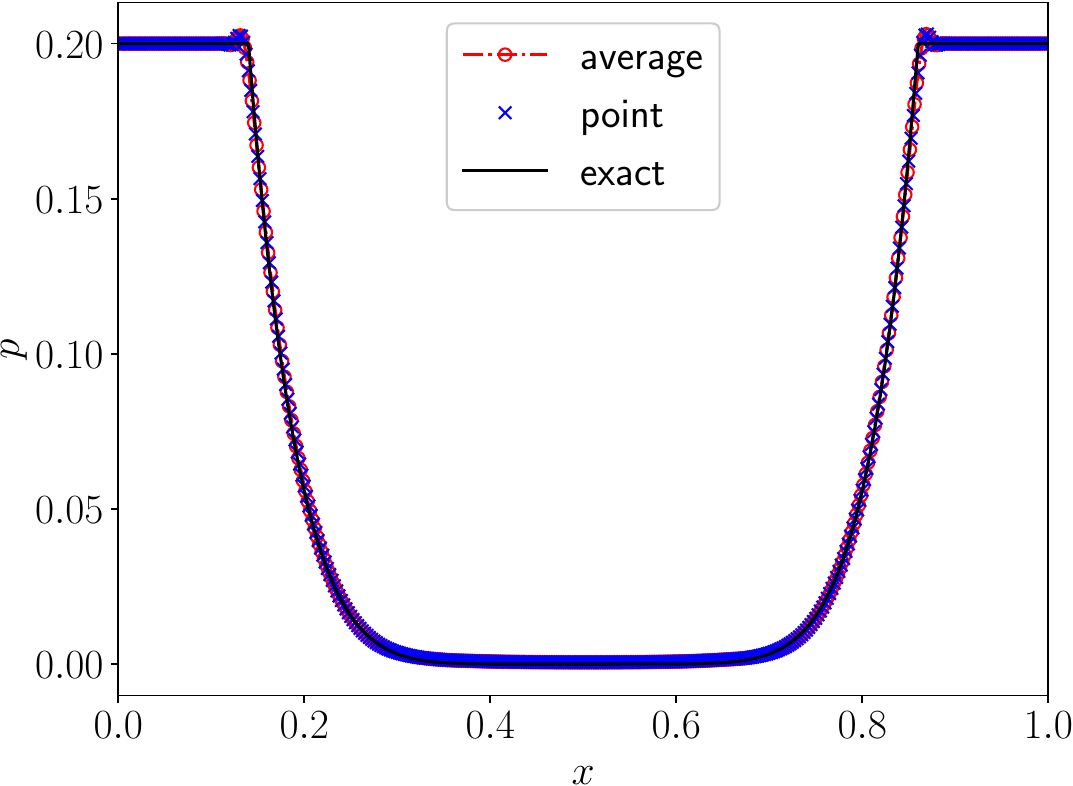}
	\end{subfigure}
	\begin{subfigure}[b]{0.24\textwidth}
		\centering
		\includegraphics[width=\linewidth]{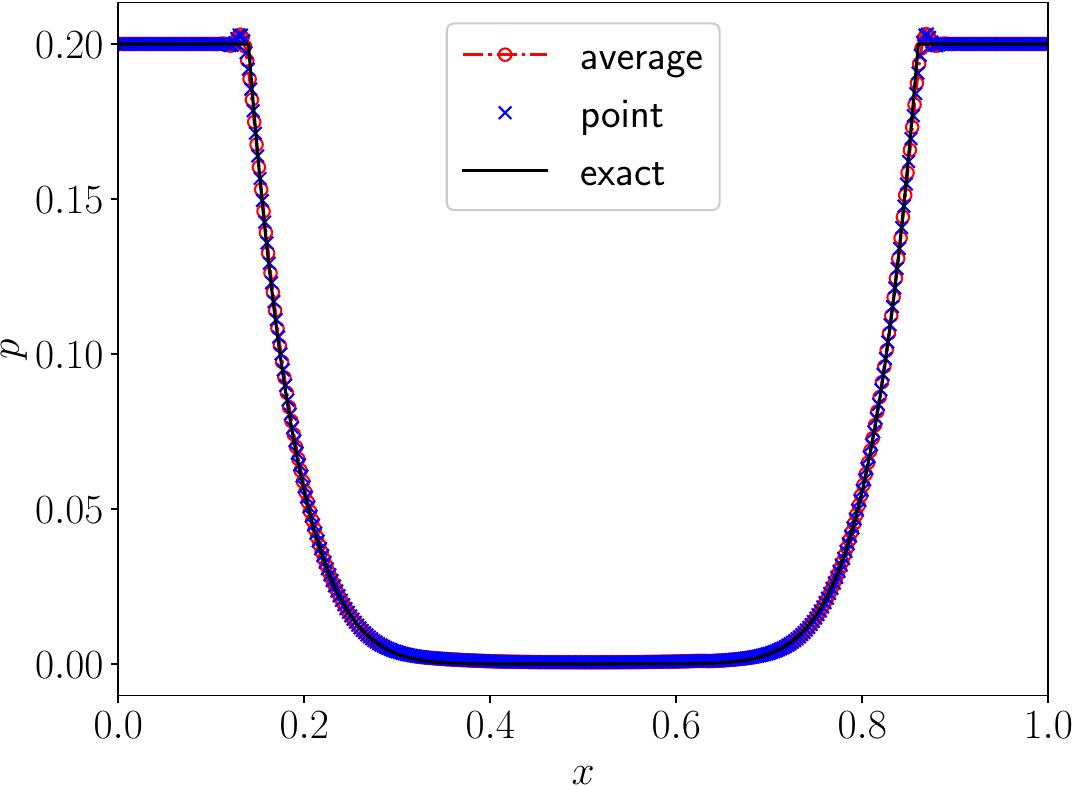}
	\end{subfigure}
	\begin{subfigure}[b]{0.24\textwidth}
		\centering
		\includegraphics[width=\linewidth]{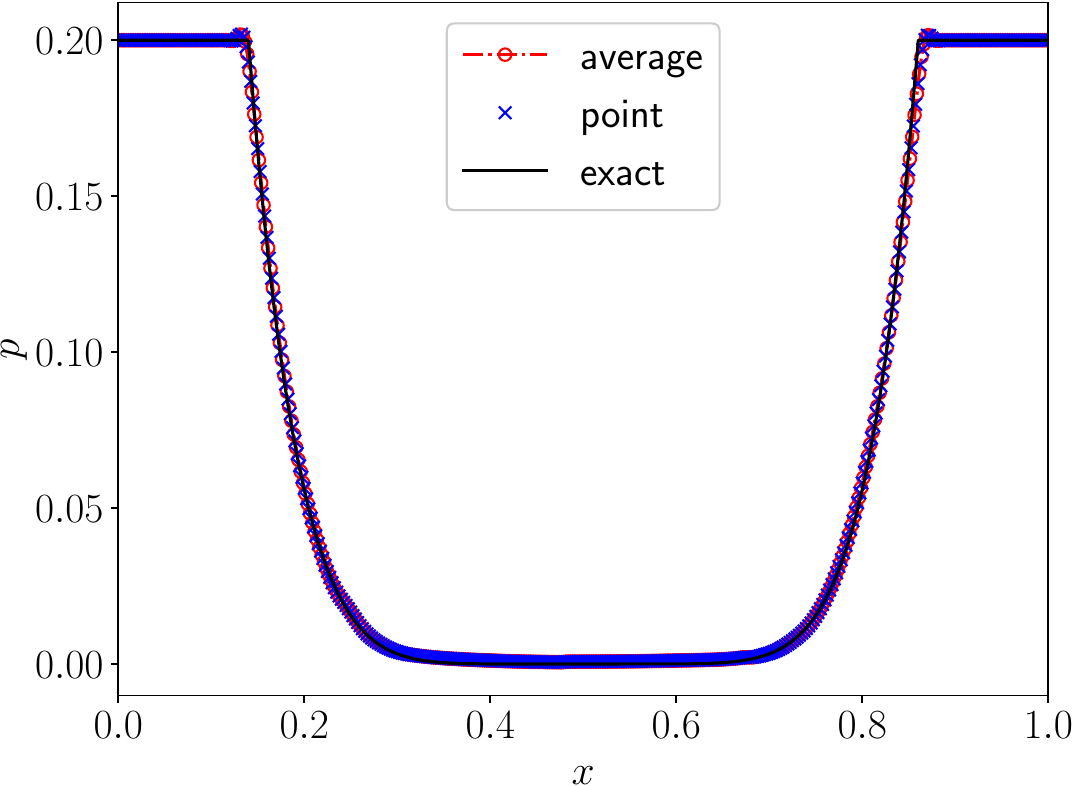}
	\end{subfigure}
	\caption{\Cref{ex:1d_double_rarefaction}, double rarefaction Riemann problem.
		The density, velocity, and pressure are computed by the BP AF methods on a uniform mesh of $400$ cells.
		From left to right: JS, LLF, SW, and VH FVS.}
	\label{fig:1d_double_rarefaction}
\end{figure}
\end{example}

\begin{example}[Blast wave interaction]\label{ex:1d_blast_wave_supp}
The power law reconstruction is useful to reduce oscillations for the fully-discrete AF method \cite{Barsukow_2021_active_JoSC}, thus we would also like to test its ability for the generalized (semi-discrete) AF method.
\Cref{fig:1d_blast_wave_cfl0.1_pwl} shows the density profiles and corresponding enlarged views obtained by using the BP limitings and power law reconstruction on a uniform mesh of $800$ cells.
It is seen that the power law reconstruction can suppress oscillations,
but the results are still more oscillatory than those using the shock sensor-based limiting.
Note that the CFL number reduces to $0.1$ when the power law reconstruction is activated.
This kind of reduction of the CFL number is also observed in other test cases thus we do not recommend using the power law reconstruction for the generalized AF methods,
which also motivates us to develop the shock sensor-based limiting.

\begin{figure}[htbp]
	\centering
	\begin{subfigure}[b]{0.24\textwidth}
		\centering
		\includegraphics[width=1.0\linewidth]{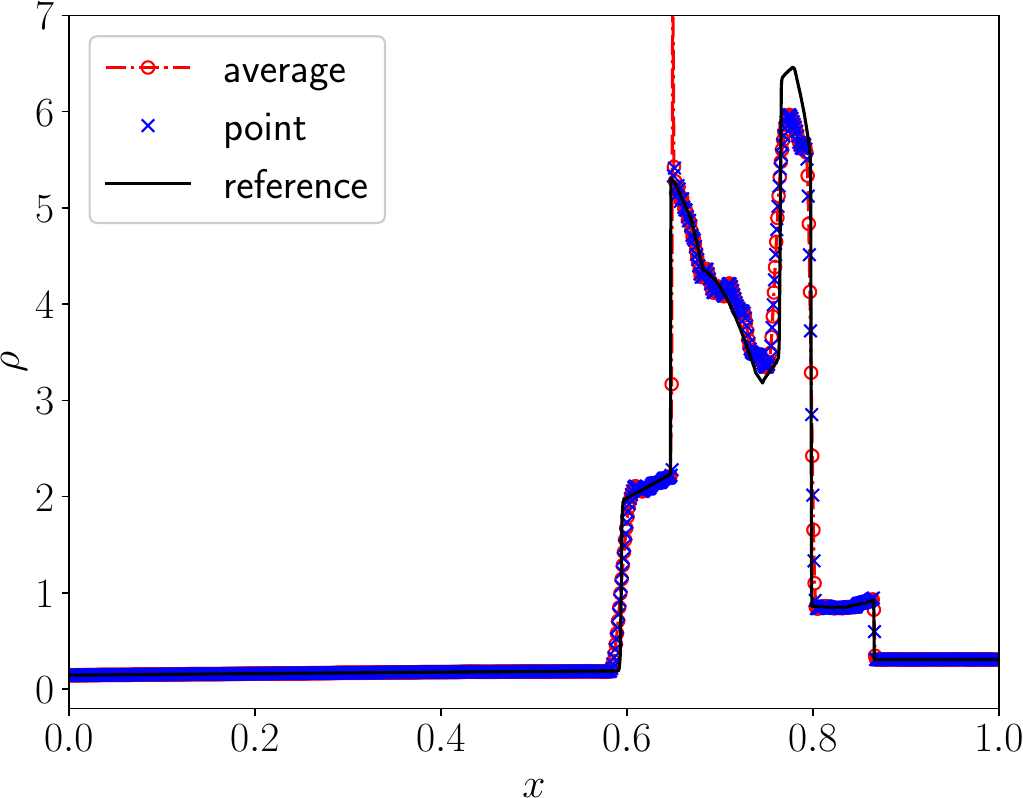}
	\end{subfigure}
	\begin{subfigure}[b]{0.24\textwidth}
		\centering
		\includegraphics[width=1.0\linewidth]{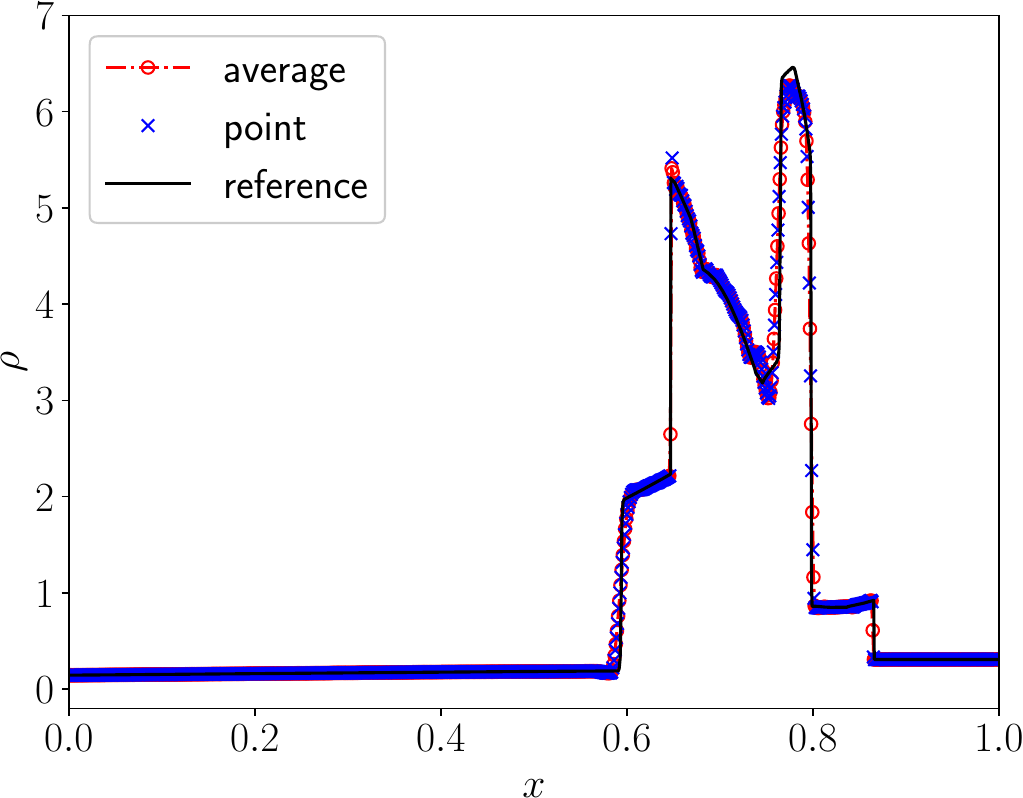}
	\end{subfigure}
	\begin{subfigure}[b]{0.24\textwidth}
		\centering
		\includegraphics[width=1.0\linewidth]{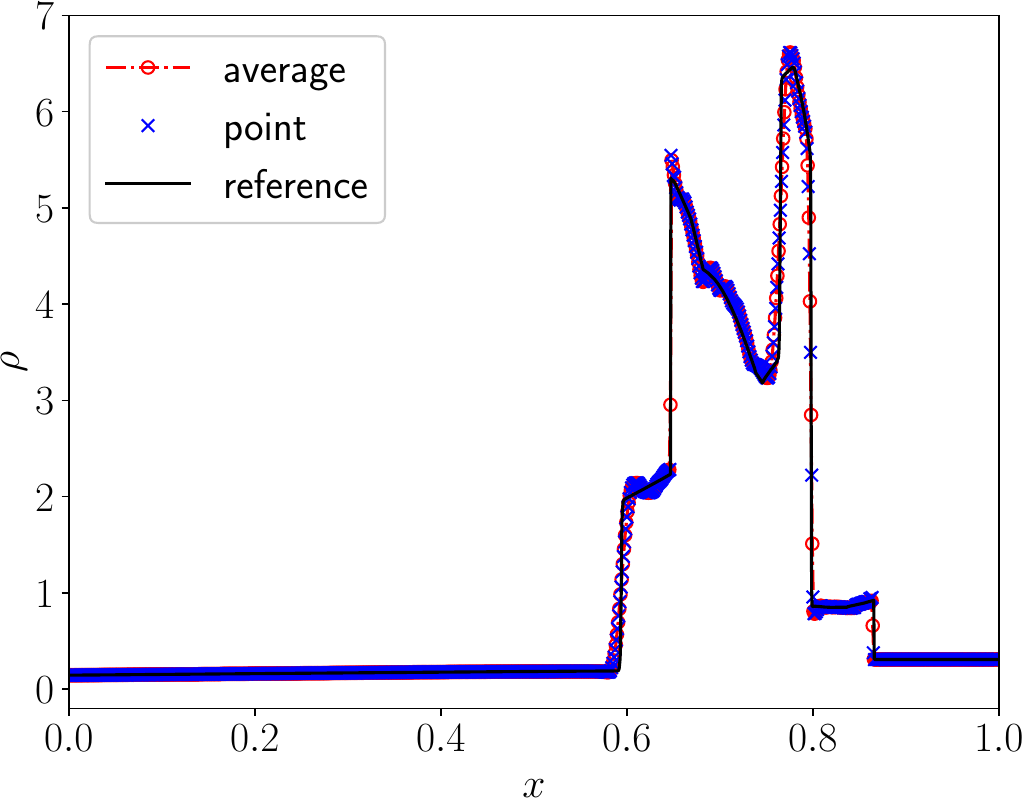}
	\end{subfigure}
	\begin{subfigure}[b]{0.24\textwidth}
		\centering
		\includegraphics[width=1.0\linewidth]{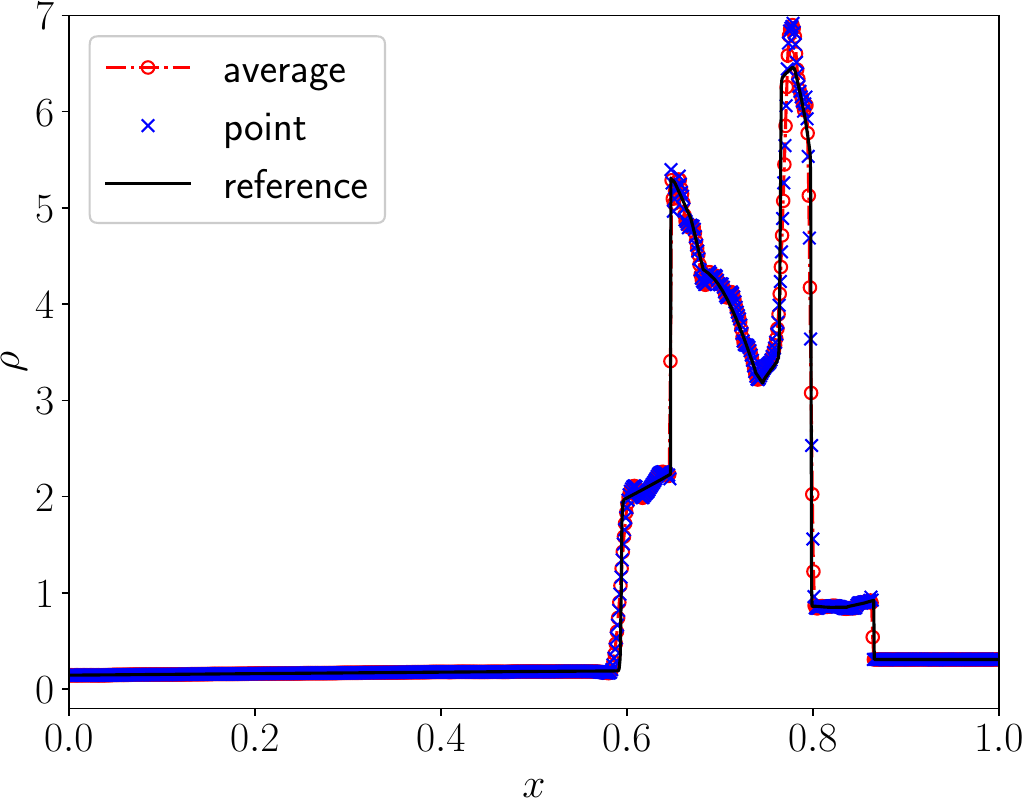}
	\end{subfigure}
	
	\begin{subfigure}[b]{0.24\textwidth}
		\centering
		\includegraphics[width=1.0\linewidth]{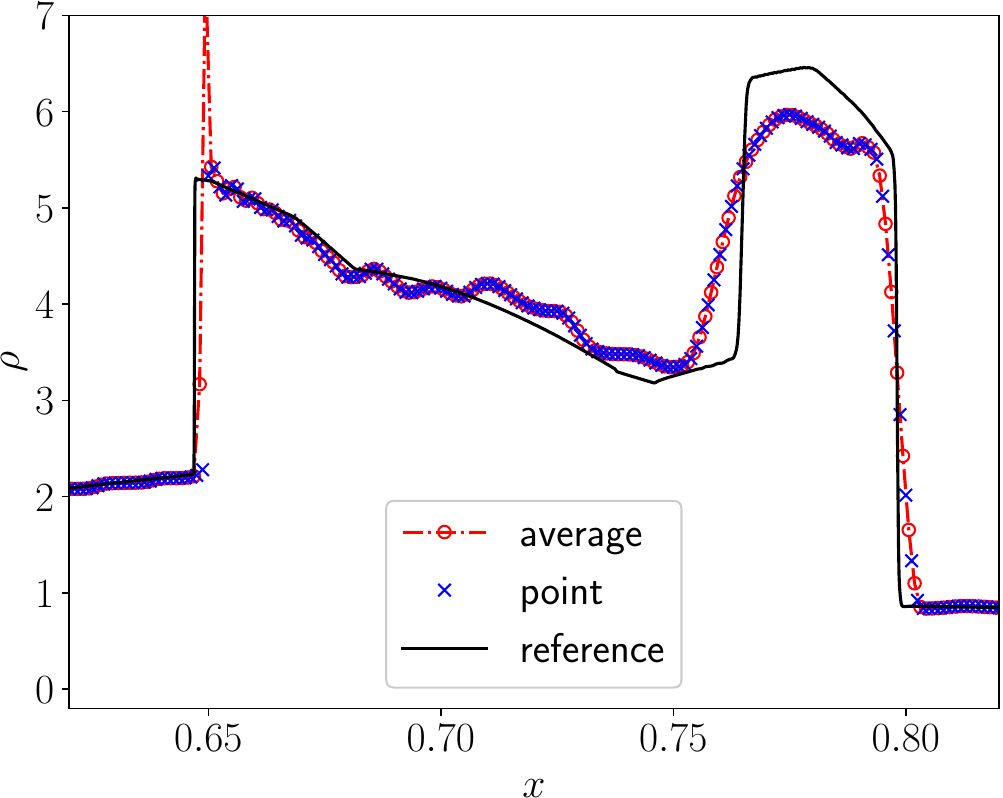}
	\end{subfigure}
	\begin{subfigure}[b]{0.24\textwidth}
		\centering
		\includegraphics[width=1.0\linewidth]{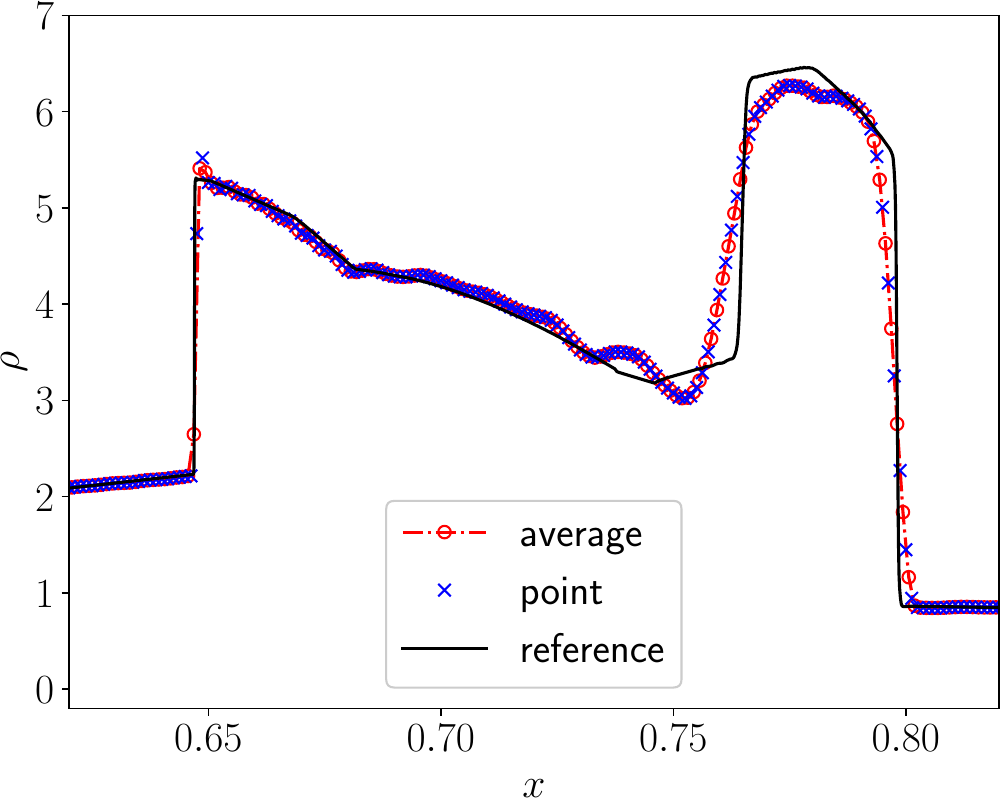}
	\end{subfigure}
	\begin{subfigure}[b]{0.24\textwidth}
		\centering
		\includegraphics[width=1.0\linewidth]{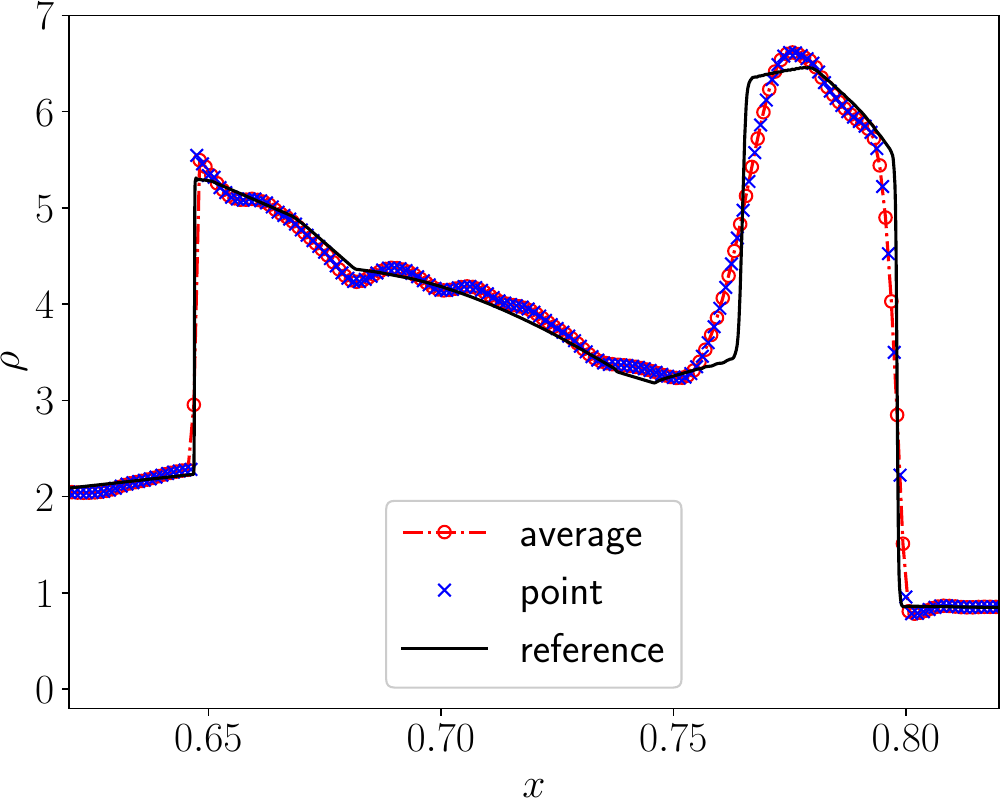}
	\end{subfigure}
	\begin{subfigure}[b]{0.24\textwidth}
		\centering
		\includegraphics[width=1.0\linewidth]{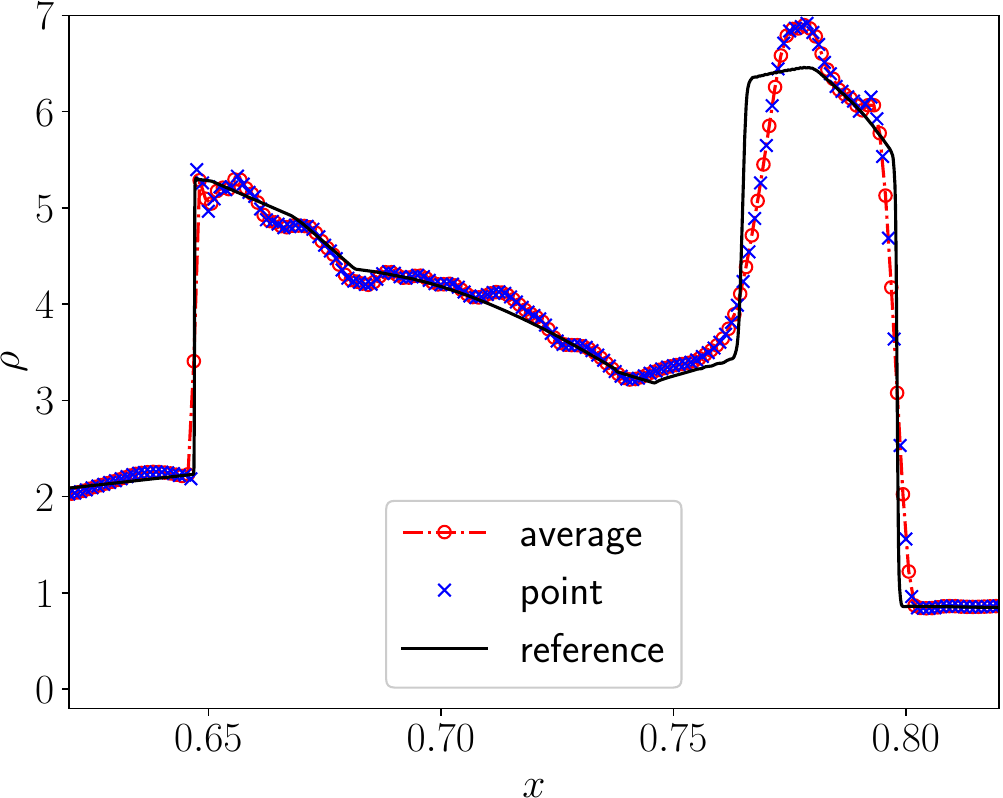}
	\end{subfigure}
	\caption{Example \ref{ex:1d_blast_wave}, blast wave interaction.
		The density computed with the power law reconstruction and BP limitings,
		and the corresponding enlarged views in $[0.62, 0.82]$ are shown in the bottom row.
		From left to right: JS, LLF, SW, and VH FVS.
	}
	\label{fig:1d_blast_wave_cfl0.1_pwl}
\end{figure}
\end{example}

\begin{example}[1D Sedov problem]\label{ex:1d_sedov}
	In this problem, a volume of uniform density and temperature is initialized, and a large quantity of thermal energy is injected at the center, developing into a blast wave that evolves in time in a self-similar fashion \cite{Sedov_1959_Similarity_book}.
	An exact analytical solution based on self-similarity arguments is available \cite{Kamm_2007_efficient}, which contains very low
	density with strong shocks.
	For the background value, the initial density is one, velocity is zero, and total energy is $10^{-12}$ everywhere except that
	in the centered cell, the total energy of the cell average and point values at two cell interfaces are $3.2\times 10^6/\Delta x$ with $\Delta x = 4/N$ with $N$ the number of cells, which is used to emulate a $\delta$-function at the center.
	The test is solved until $T= 10^{-3}$.
	
	This test is run with $N=801$ cells, and the density plots in the right half domain are shown in \cref{fig:1d_sedov}.
	The BP limitings are adopted for the cell average and point value updates.
	The LLF FVS is used and the CFL number is taken as $0.4$.
	
	\begin{figure}[htbp]
		\centering
		\includegraphics[width=0.4\linewidth]{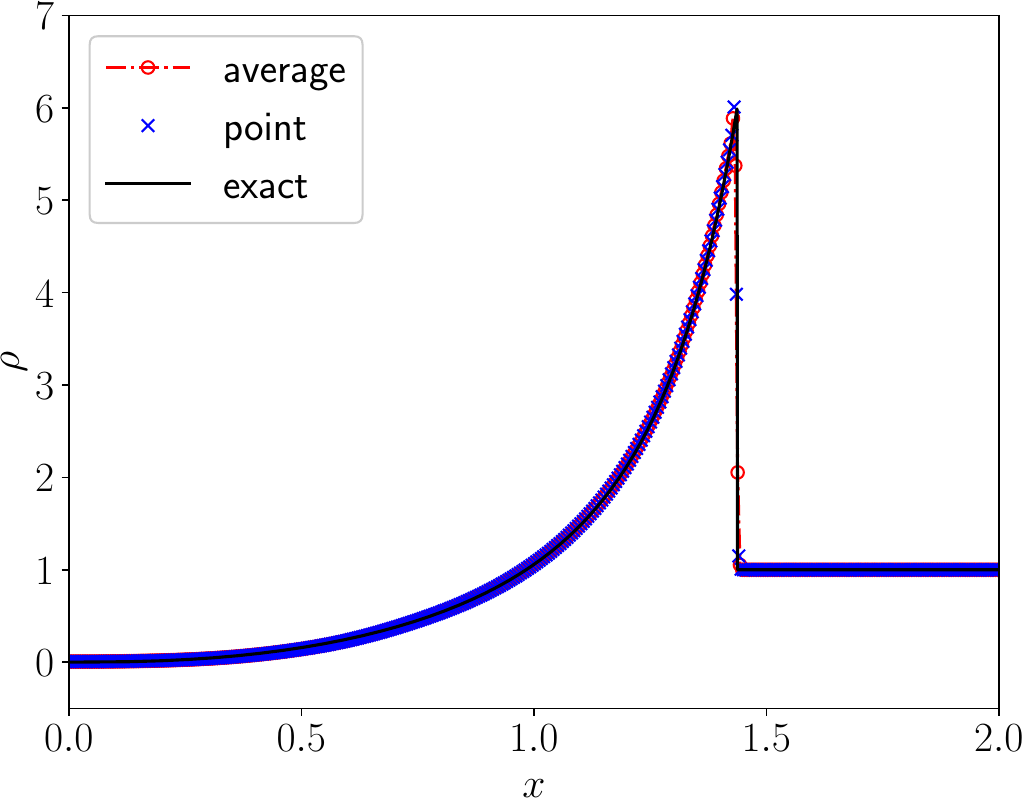}
		\caption{Example \ref{ex:1d_sedov}, 1D Sedov problem.
			The numerical solutions are computed with the LLF FVS and the BP limitings on a uniform mesh of $801$ cells.}
		\label{fig:1d_sedov}
	\end{figure}
\end{example}

\begin{example}[Shock reflection problem]\label{ex:2d_shock_reflection}
	The computational domain is $[0,4]\times[0,1]$, which is divided into a $120\times30$ uniform mesh. The boundary conditions are outflow at the right boundary, reflective at the bottom boundary, and inflow on the other two sides with the data
	\begin{equation*}
		(\rho, v_1, v_2, p) = 
		\begin{cases}
			(1.0, ~2.9, ~0.0, ~1.0/1.4), &\text{if}~ x=0,~ 0\leqslant y\leqslant 1,\\
			(1.69997, ~2.61934, -0.50632, ~1.52819), &\text{if}~ y=1,~ 0\leqslant x\leqslant 4. \\
		\end{cases}
	\end{equation*}
	This test is solved until $T=6$ thus the numerical solution converges.
	
	The density plots obtained without any limiting ($\kappa=0$) and with the shock sensor-based limiting ($\kappa=0.5$) are shown in \Cref{fig:2d_sf_density}, and the blending coefficients based on the shock sensor are plotted in \Cref{fig:2d_sf_ss}.
	The numerical solutions converge in both cases, and the shock sensor can correctly locate the shock waves.
	It is also interesting to look at the residual between two successive time steps, presented in \Cref{fig:2d_sf_residual}, with respective to the number of iterations.
	The limiting based on the shock sensor accelerates the convergence after the reflective shock is fully formed, showing the advantage of using the shock sensor.
	
	\begin{figure}[htbp!]
		\centering
		\begin{subfigure}[b]{0.48\textwidth}
			\centering
			\includegraphics[width=\linewidth]{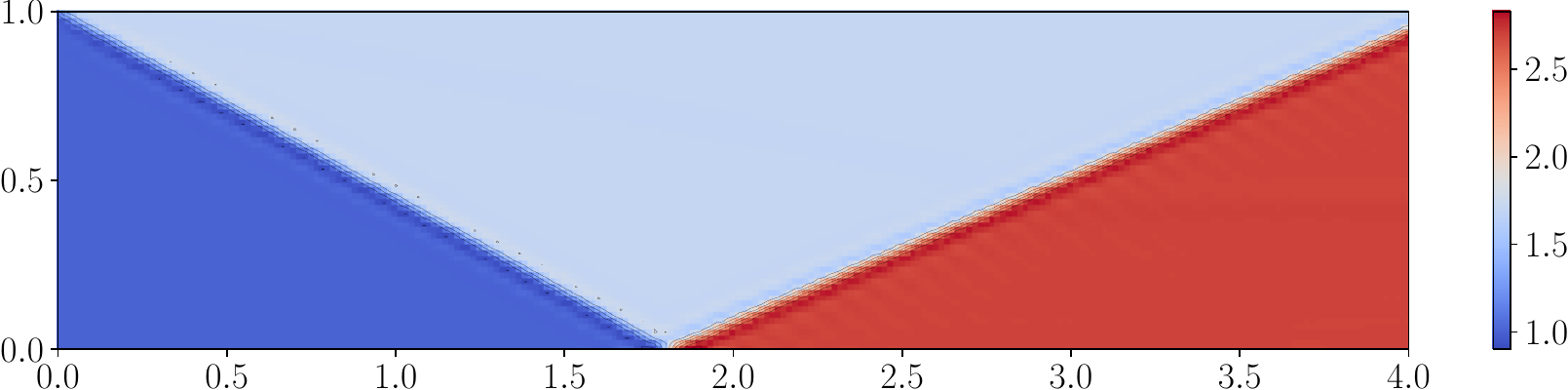}
		\end{subfigure}
		\quad
		\begin{subfigure}[b]{0.48\textwidth}
			\centering
			\includegraphics[width=\linewidth]{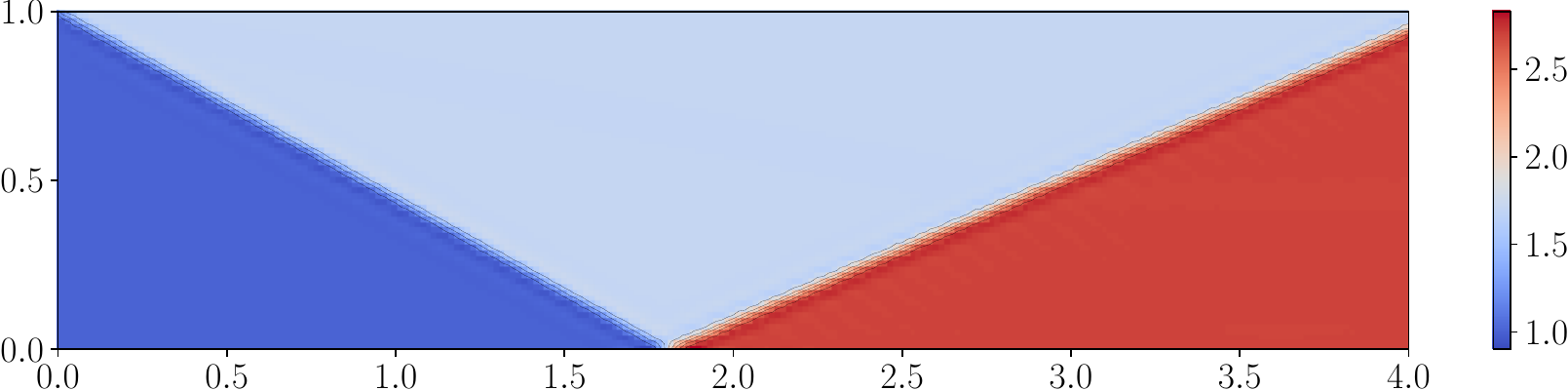}
		\end{subfigure}
		\caption{Example \ref{ex:2d_shock_reflection}, shock reflection problem.
			The density obtained without ($\kappa=0$, left) or with the shock sensor ($\kappa=0.5$, right) on the $120\times30$ uniform mesh. $10$ equally spaced contour lines from $0.901$ to $2.829$ are shown.}
		\label{fig:2d_sf_density}
	\end{figure}
	
	\begin{figure}[htbp!]
		\centering
		\begin{subfigure}[b]{0.48\textwidth}
			\includegraphics[width=\linewidth]{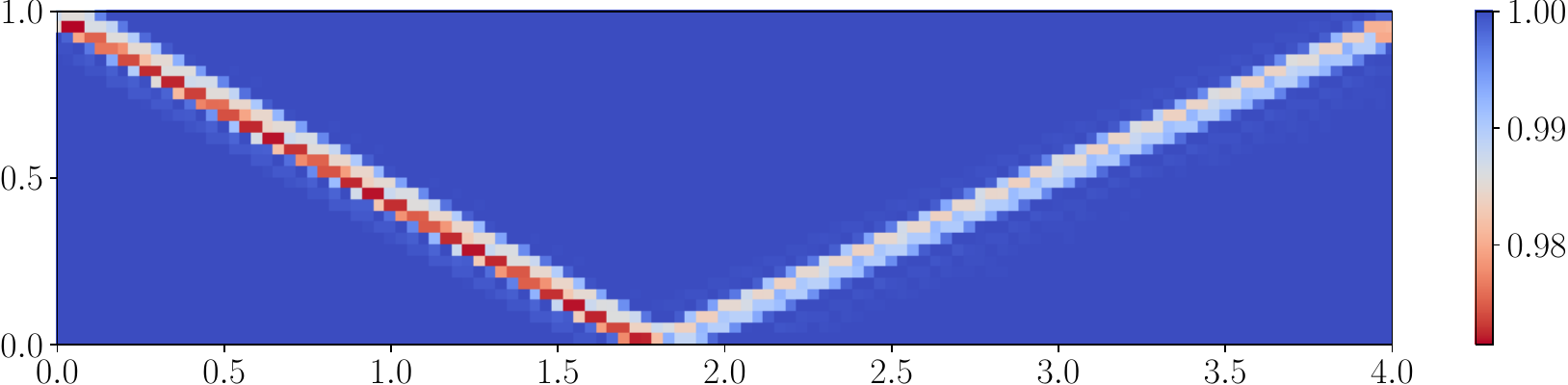}
		\end{subfigure}
		\quad
		\begin{subfigure}[b]{0.48\textwidth}
			\centering
			\includegraphics[width=\linewidth]{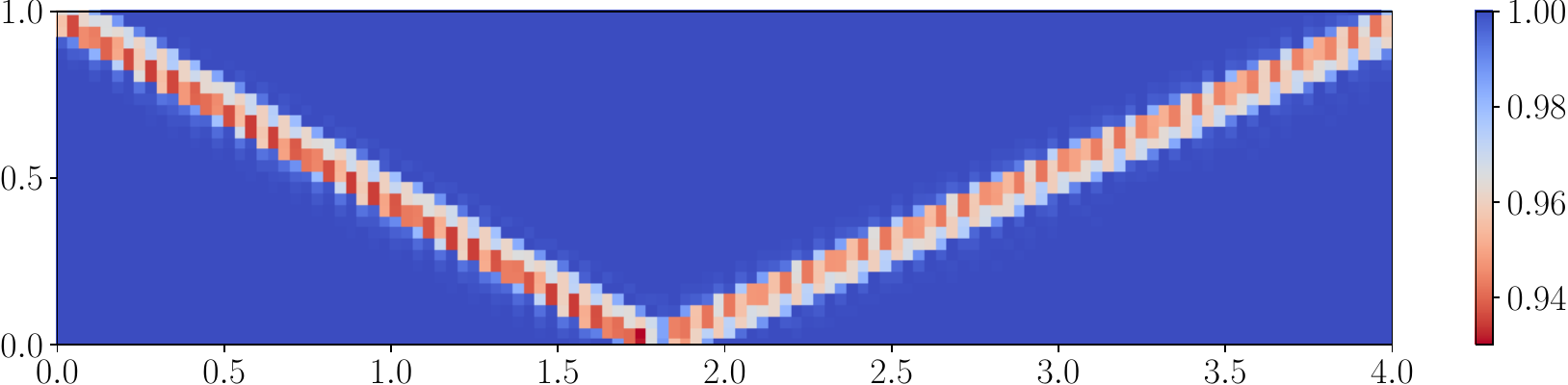}
		\end{subfigure}
		\caption{Example \ref{ex:2d_shock_reflection}, shock reflection problem.
			The shock sensor-based blending coefficients $\theta_{\xr,j}^{s}$ (left) and $\theta_{i,\yr}^{s}$ (right) on the $120\times30$ uniform mesh. $\kappa=0.5$.}
		\label{fig:2d_sf_ss}
	\end{figure}
	
	\begin{figure}[htbp!]
		\centering
		\includegraphics[width=0.4\linewidth]{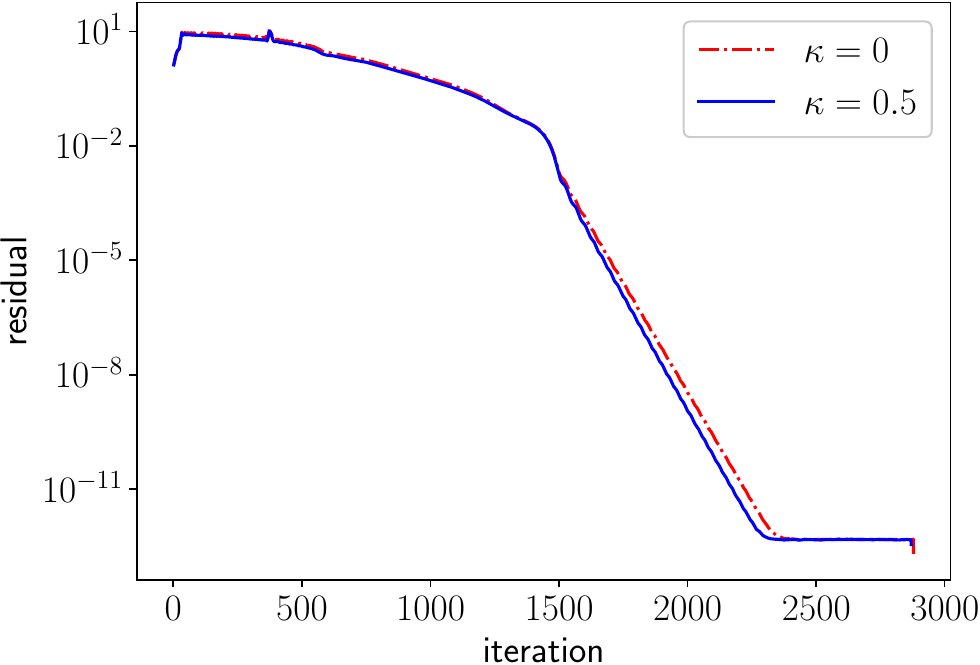}
		\caption{Example \ref{ex:2d_shock_reflection}, shock reflection problem.
			The residual decay with respect to the number of iterations.}
		\label{fig:2d_sf_residual}
	\end{figure}
\end{example}

\begin{example}[2D Riemann problem]\label{ex:2d_rp}
	This problem corresponds to the configuration $3$ in \cite{Lax_1998_Solution_SJoSC}, containing four initial shock waves, with the initial data
	\begin{equation*}
		(\rho, v_1, v_2, p) = 
		\begin{cases}
			(1.5, ~0, ~0, ~1.5), & x > 0.8, ~y > 0.8, \\
			(0.5323, ~1.206, ~0, ~0.3), & x < 0.8, ~y > 0.8, \\
			(0.138, ~1.206, ~1.206, ~0.029), & x < 0.8, ~y < 0.8, \\
			(0.5323, ~0, ~1.206, ~0.3), & x > 0.8, ~y < 0.8. \\
		\end{cases}
	\end{equation*}
	The test is solved on the domain $[0,1]\times[0,1]$ until $T=0.8$.
	
	Without the BP limitings, the simulation crashes due to negative pressure.
	The density plots obtained without ($\kappa=0$) and with the shock sensor ($\kappa=0.5$) are shown in \Cref{fig:2d_rp_density}.
	Without the shock sensor, the numerical solutions contain spurious oscillations, and they are reduced drastically by the shock sensor-based limiting.
	As mesh refinement, the shock waves are captured sharply, and the small-scale features are preserved well, as evidenced by the roll-ups around the mushroom-shaped jet, which are in good agreement with the results in the literature.
	The values of the shock sensor-based blending coefficients $\theta_{\xr,j}, \theta_{i,\yr}$ are also plotted in \Cref{fig:2d_rp_ss}, which indicates that the shock sensor can locate the shock waves correctly.
	
	\begin{figure}[htbp!]
		\begin{subfigure}[b]{0.3\textwidth}
			\centering
			\includegraphics[width=\linewidth]{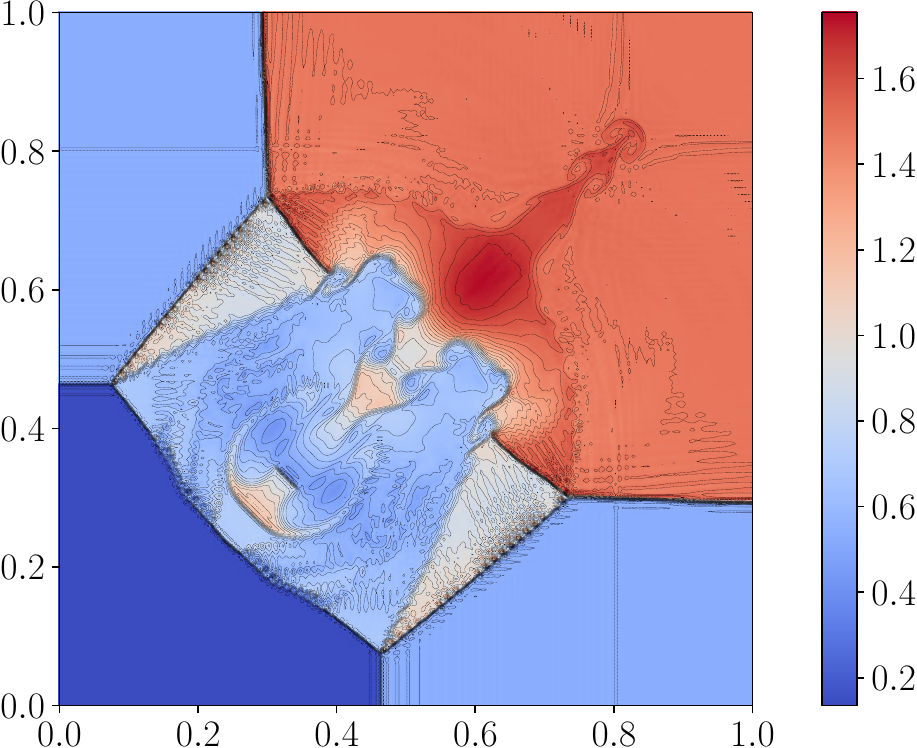}
		\end{subfigure}
		~
		\begin{subfigure}[b]{0.3\textwidth}
			\centering
			\includegraphics[width=\linewidth]{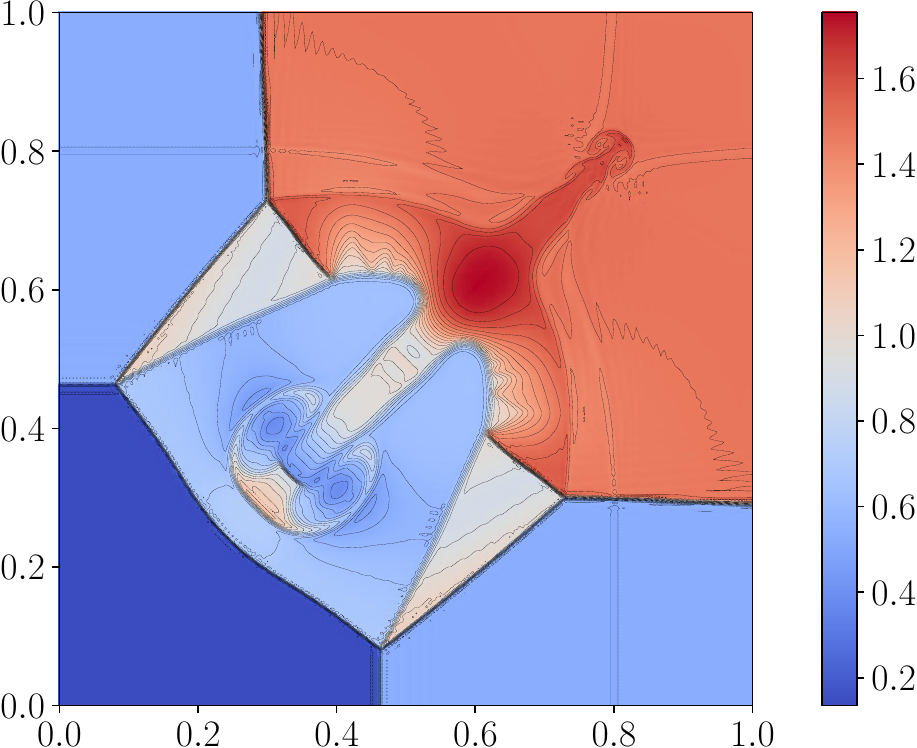}
		\end{subfigure}
		~
		\begin{subfigure}[b]{0.3\textwidth}
			\centering
			\includegraphics[width=\linewidth]{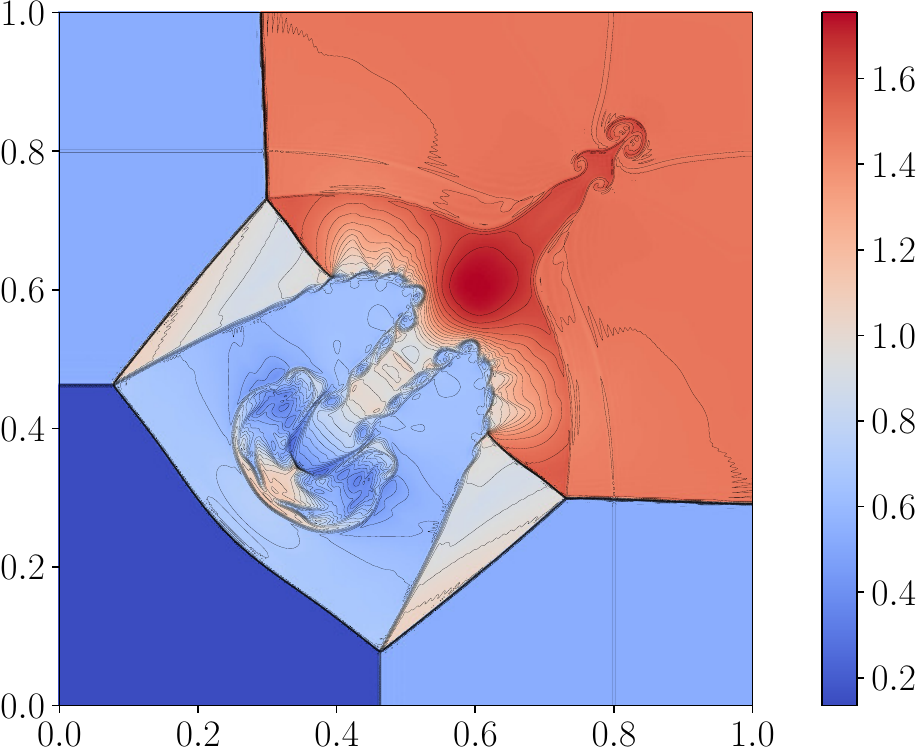}
		\end{subfigure}
		\caption{Example \ref{ex:2d_rp}, 2D Riemann problem.
			The density obtained with the BP limitings and without or with the shock sensor.
			From left to right: $200\times200$ mesh with $\kappa=0$, $200\times200$ mesh with $\kappa=0.5$, $400\times400$ mesh with $\kappa=0.5$.
			$30$ equally spaced contour lines from $0.135$ to $1.754$. }
		\label{fig:2d_rp_density}
	\end{figure}
	
	\begin{figure}[htbp!]
		\centering
		\begin{subfigure}[b]{0.3\textwidth}
			\centering
			\includegraphics[width=\linewidth]{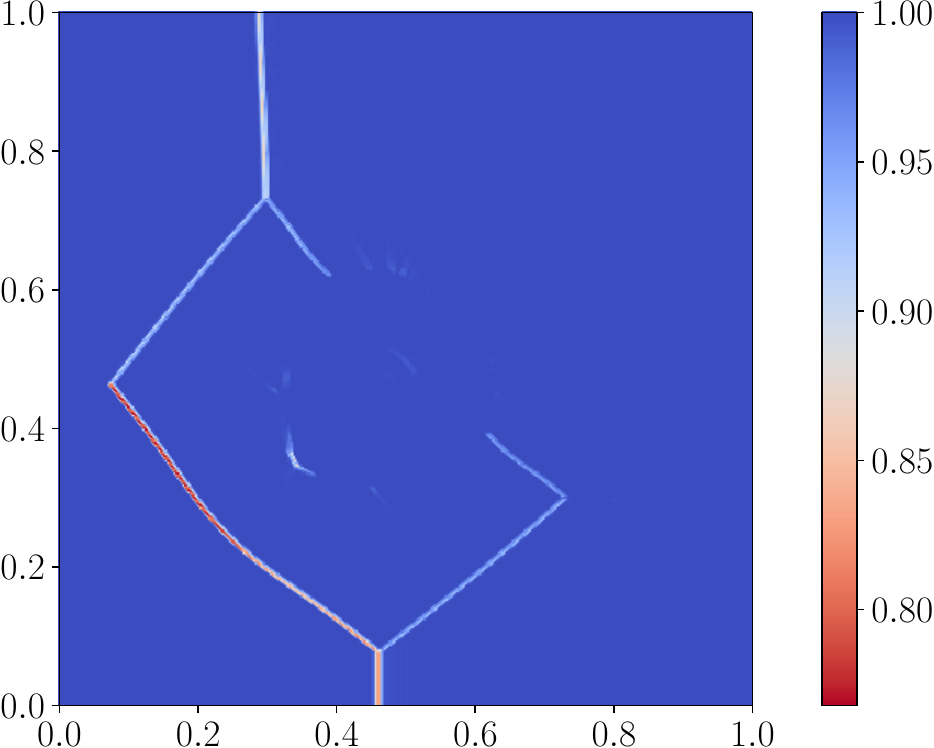}
		\end{subfigure}
		\qquad\qquad\qquad
		\begin{subfigure}[b]{0.3\textwidth}
			\centering
			\includegraphics[width=\linewidth]{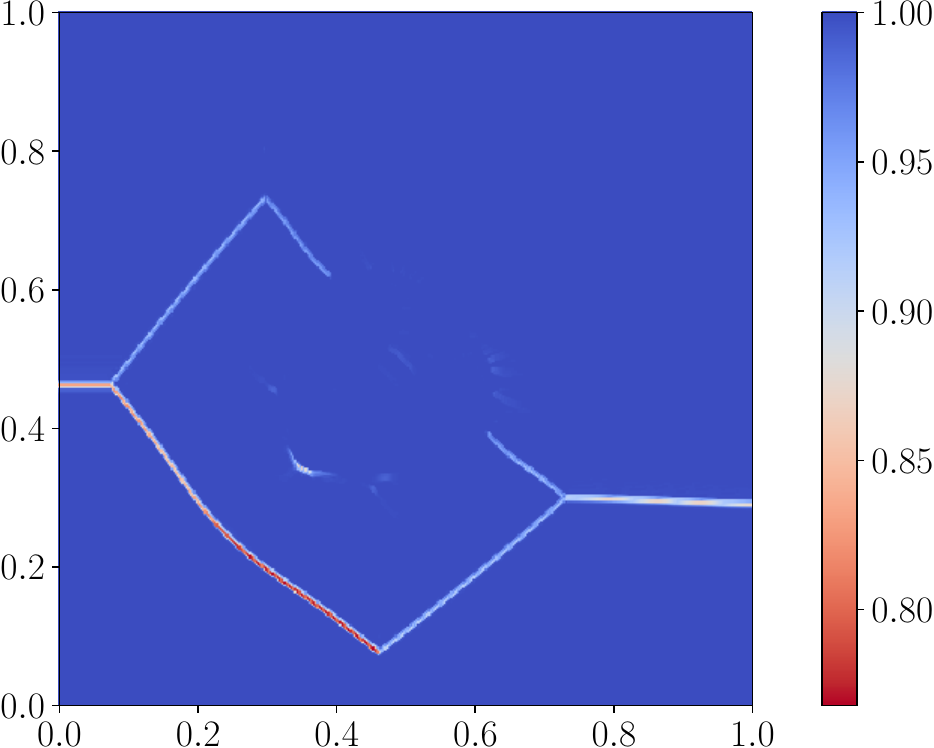}
		\end{subfigure}
		\caption{Example \ref{ex:2d_rp}, 2D Riemann problem.
			The shock sensor-based blending coefficients $\theta_{\xr,j}^{s}$ (left) and $\theta_{i,\yr}^{s}$ (right) on the $400\times400$ uniform mesh.}
		\label{fig:2d_rp_ss}
	\end{figure}
\end{example}

\begin{example}[Double Mach reflection]\label{ex:2d_dmr}
	The computational domain is $[0, 3]\times[0, 1]$ with a reflective wall at the bottom starting from $x = 1/6$.
	A Mach $10$ shock is moving towards the bottom wall with an angle of $\pi/6$.
	The pre- and post-shock states are
	\begin{equation*}
		(\rho, v_1, v_2, p) = 
		\begin{cases}
			(1.4, ~0, ~0, ~1), & x \geqslant 1/6 + (y+20t)/\sqrt{3}, \\
			(8, ~8.25\cos(\pi/6), -8.25\sin(\pi/6), ~116.5), & x < 1/6 + (y+20t)/\sqrt{3}. \\
		\end{cases}
	\end{equation*}
	The reflective boundary condition is applied at the wall, while the exact post-shock condition is imposed at the left boundary and for the rest of the bottom boundary (from $x = 0$ to $x = 1/6$).
	At the top boundary, the exact motion of the Mach $10$ shock is applied and the outflow boundary condition is used at the right boundary.
	The results are shown at $T = 0.2$.
	
	The AF method without the BP limitings gives negative density or pressure near the reflective location $(1/6, 0)$, so the BP limitings are necessary for this test.
 The numerical solutions are computed
 without or with the shock sensor ($\kappa=1$) on a series of uniform meshes.
	The density plots with enlarged views around the double Mach region are shown in \Cref{fig:2d_dmr_density},
  and the blending coefficients based on the shock sensor are shown in \Cref{fig:2d_dmr_ss}.
	When the shock sensor is not activated, the noise after the bow shock is obvious, and it is damped with the help of the shock sensor.
	As mesh refinement, the numerical solutions converge with a good resolution and are comparable to those in the literature.
	Compared to the third-order $P^2$ DG method using the TVB limiter \cite{Cockburn_2001_Runge_JoSC} with the same mesh resolution ($\Delta x=\Delta y=1/480$), the roll-ups and vortices are comparable while the AF method uses fewer DoFs ($4$ versus $6$ per cell).
	
	\begin{figure}[htbp!]
		\centering
		\begin{subfigure}[b]{0.6\textwidth}
			\centering	\includegraphics[width=\linewidth]{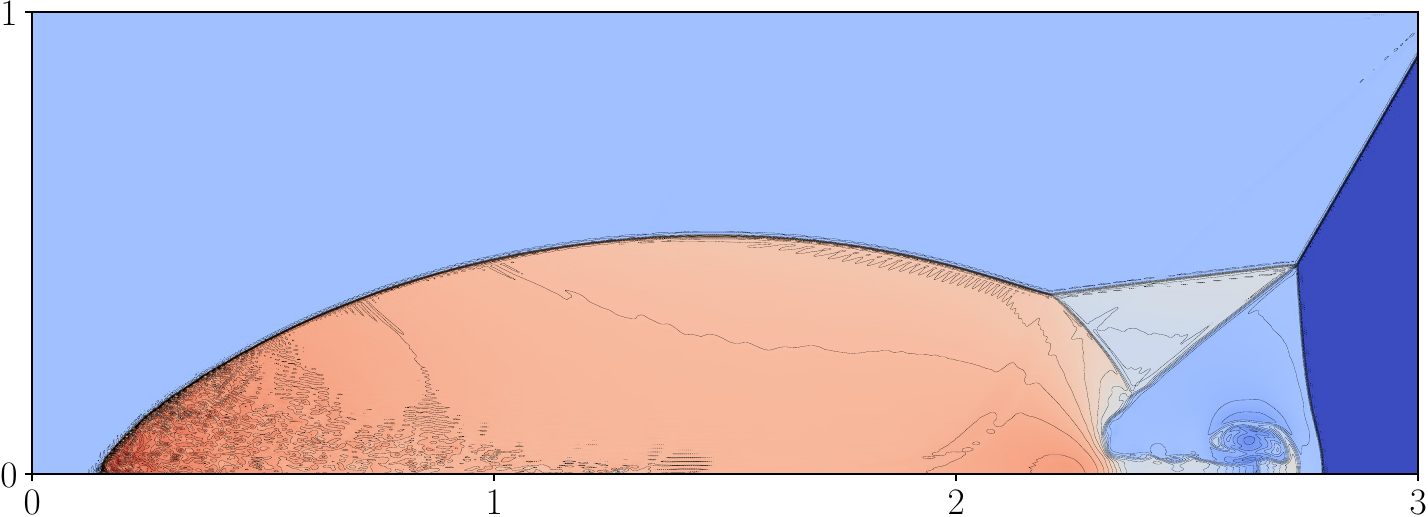}
		\end{subfigure}
		\begin{subfigure}[b]{0.384\textwidth}
			\centering
			\includegraphics[width=\linewidth]{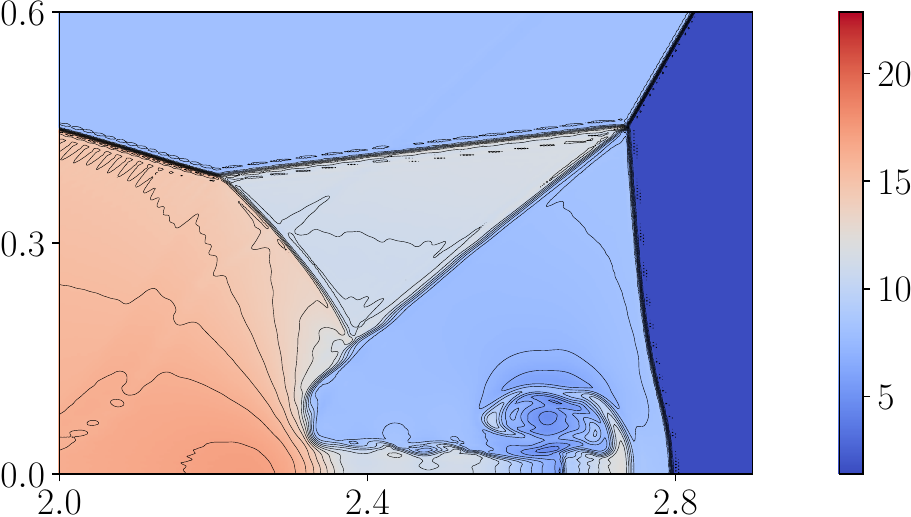}
		\end{subfigure}
		
		\begin{subfigure}[b]{0.6\textwidth}
			\centering
			\includegraphics[width=\linewidth]{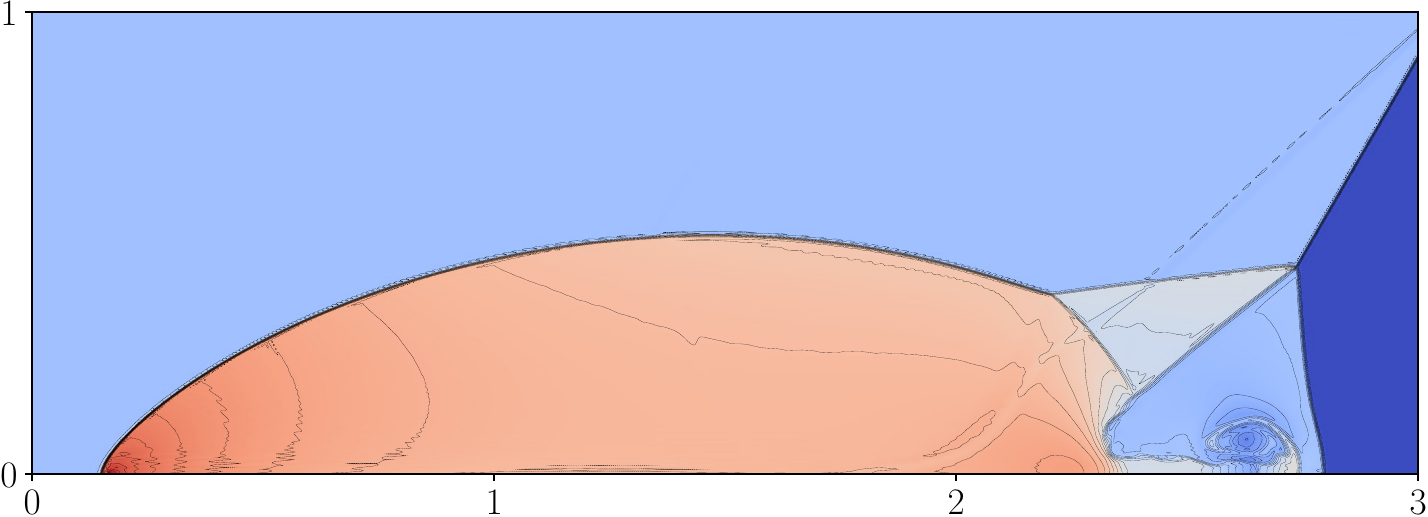}
		\end{subfigure}
		\begin{subfigure}[b]{0.384\textwidth}
			\centering
			\includegraphics[width=\linewidth]{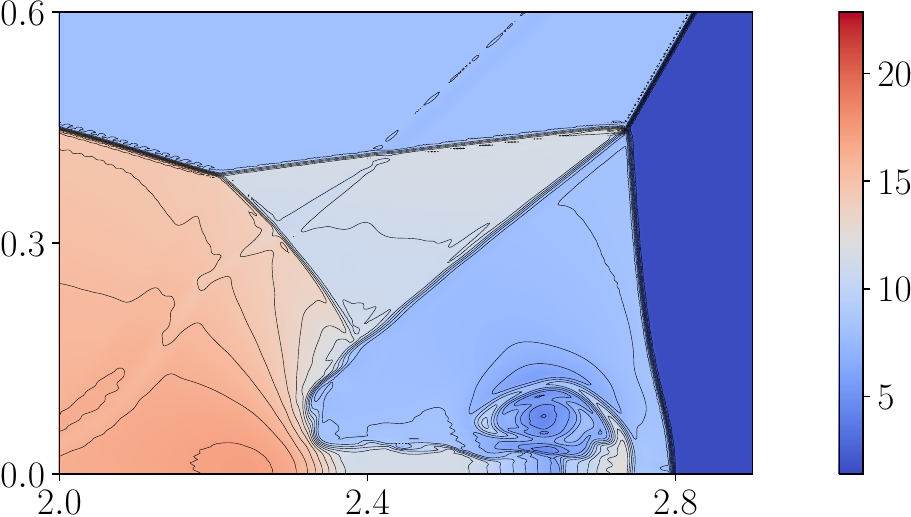}
		\end{subfigure}
		
		\begin{subfigure}[b]{0.6\textwidth}
			\includegraphics[width=\linewidth]{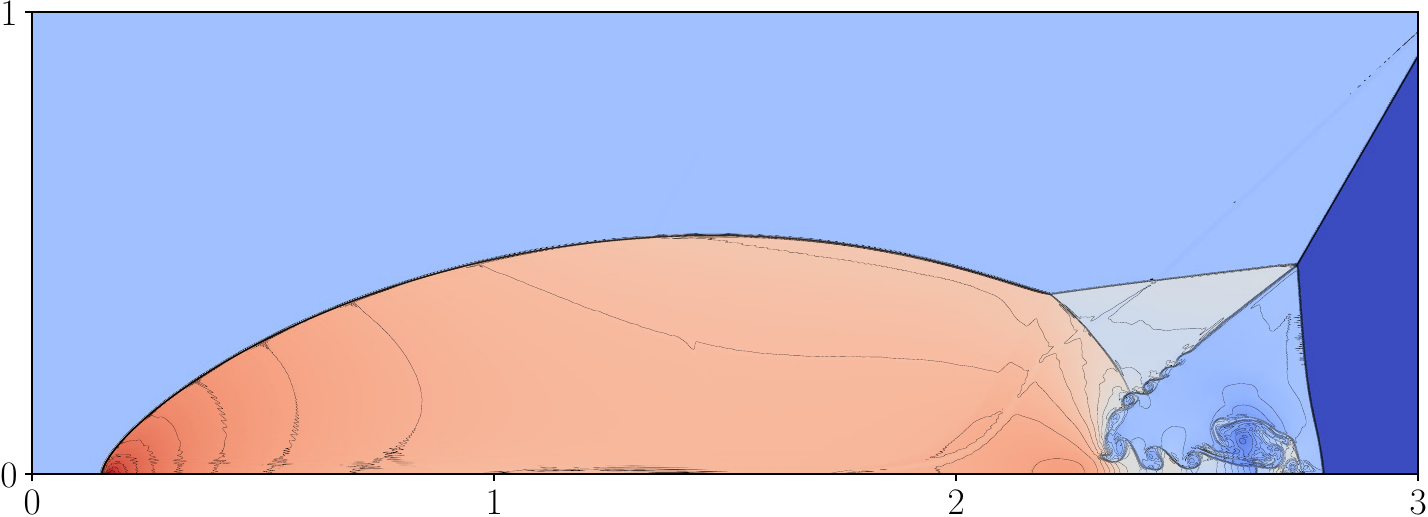}
		\end{subfigure}
		\begin{subfigure}[b]{0.384\textwidth}
			\centering
			\includegraphics[width=\linewidth]{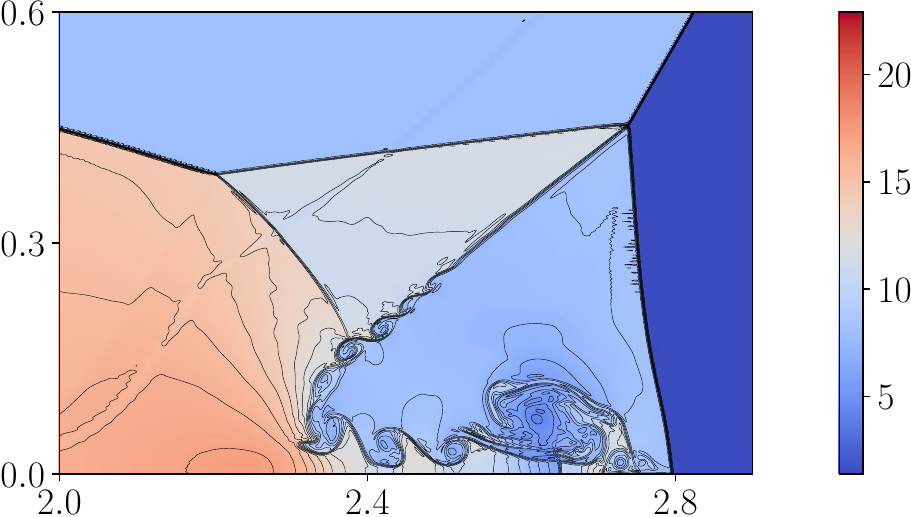}
		\end{subfigure}
		
		\caption{Example \ref{ex:2d_dmr}, double Mach reflection.
			The density obtained with the BP limitings and without or with the shock sensor.
			From top to bottom: $720\times240$ mesh without shock sensor, $720\times240$ mesh with $\kappa=1$, $1440\times480$ mesh with $\kappa=1$.
			$30$ equally spaced contour lines from $1.390$ to $22.861$.}
		\label{fig:2d_dmr_density}
	\end{figure}
	
	\begin{figure}[htbp!]
		\centering		
		\begin{subfigure}[b]{0.48\textwidth}
			\centering
			\includegraphics[width=\linewidth]{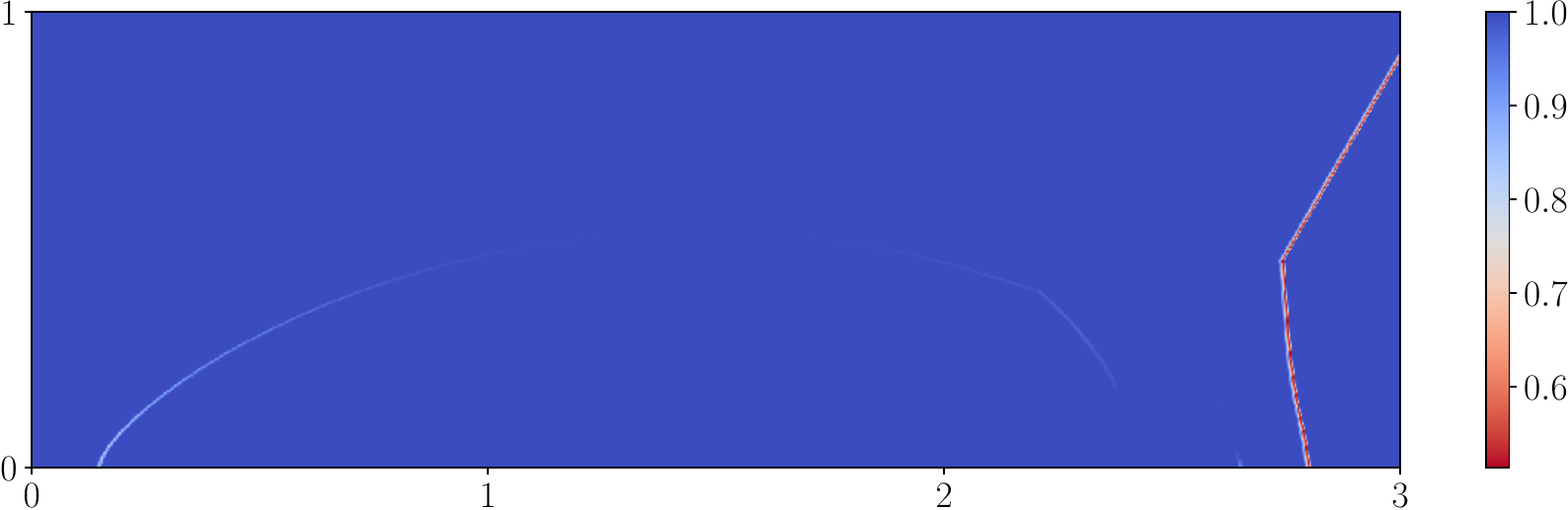}
		\end{subfigure}
		\quad
		\begin{subfigure}[b]{0.48\textwidth}
			\centering
			\includegraphics[width=\linewidth]{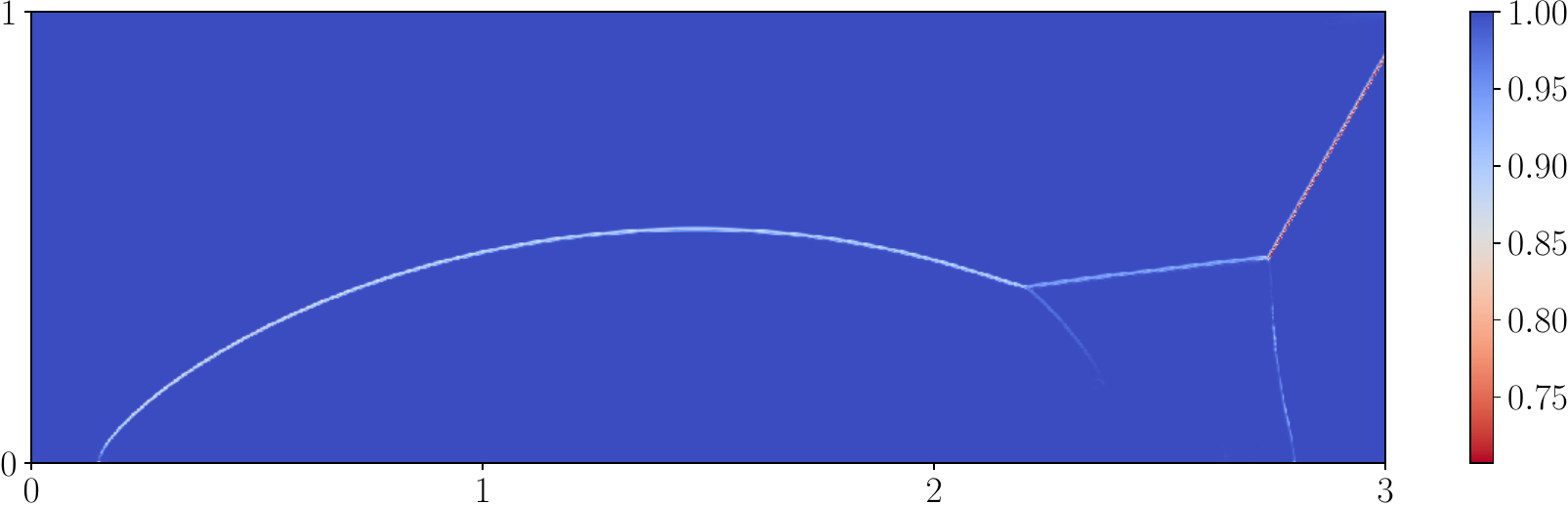}
		\end{subfigure}
		\caption{Example \ref{ex:2d_dmr}, double Mach reflection.
			The blending coefficients $\theta_{\xr,j}^{s}$ (left) and $\theta_{i,\yr}^{s}$ (right) based on the shock sensor with $\kappa=1$ on the $1440\times 480$ mesh.}
		\label{fig:2d_dmr_ss}
	\end{figure}
\end{example}




\end{document}